\documentclass[11pt,a4paper,reqno]{amsart}
\usepackage{a4wide}
\usepackage{graphicx}
\usepackage{latexsym}
\usepackage{epsfig}
\usepackage{ esint }
\usepackage{amssymb}
\usepackage{amstext}
\usepackage{amsgen}
\usepackage{amsxtra}
\usepackage{amsgen}
\usepackage{mathtools}
\usepackage{amsthm}
\usepackage{enumitem}
\usepackage{tcolorbox}
\usepackage{mathrsfs,dsfont,amsfonts}
\usepackage[T1]{fontenc}
\usepackage{bbm}
\usepackage{lmodern}

\usepackage[numbers,sort&compress]{natbib}
\let\oldthebibliography\thebibliography
\renewcommand{\thebibliography}[1]{%
  \oldthebibliography{#1}%
  \footnotesize 
}

\usepackage[top=2.5cm, bottom=2.5cm, left=2.5cm, right=2.5cm]{geometry}

\usepackage{url}

\mathtoolsset{showonlyrefs}

\usepackage{comment}

\usepackage[pagewise]{lineno}

\usepackage{color}
\usepackage[colorlinks=true, linkcolor=blue, citecolor=red]{hyperref}

\newtheorem{thm}{Theorem}[section]
\newtheorem{prop}[thm]{Proposition}
\newtheorem{lemma}[thm]{Lemma}
\newtheorem{cor}[thm]{Corollary}

\theoremstyle{remark}
\newtheorem{rem}[thm]{Remark}

\theoremstyle{definition}
\newtheorem{definition}[thm]{Definition}

\numberwithin{equation}{section}

\newcommand{\eps}{\varepsilon}

\newcommand{\rr}{\mathbb{R}}
\newcommand{\nn}{\mathbb{N}}

\newcommand{\nchi}{{\raise.3ex\hbox{$\chi$}}}

\newcommand{\sfd}{{\sf d}}
\newcommand{\Lip}{{\rm Lip}}

\renewcommand{\phi}{\varphi}
\newcommand{\restr}[1]{\lower3pt\hbox{$|_{#1}$}}

\newcommand{\X}{{\rm X}}

\newcommand{\fr}{\penalty-20\null\hfill$\blacksquare$} 
\definecolor{mygray}{gray}{0.9}

\newcommand{\la}{\langle}
\newcommand{\ra}{\rangle}
\renewcommand{\div}{{\rm div}}
\newcommand{\mea}{\mathfrak{m}}
\newcommand{\mm}{\mathfrak{m}}

\newcommand{\LIP}{\mathsf{LIP}}

\newcommand{\test}{{\sf{Test}}}

\renewcommand{\d}{{\, \mathrm d}} 

\newcommand{\loc}{\mathsf{loc}}
\newcommand{\W}{\mathit{W}^{1,2}}

\newcommand{\supp}{\mathop{\rm supp}\nolimits}

\newcommand{\Xdm}{(\X,\sfd,\mm)}
\newcommand{\RCD}{\mathrm{RCD}}
\newcommand{\CD}{\mathrm{CD}}

\newcommand{\rcd}{\mathrm{RCD}}

\newcommand{\dom}{{\sf D}}

\renewcommand{\limsup}{\varlimsup}
\renewcommand{\liminf}{\varliminf}

\theoremstyle{plain}

\newcommand{\AC}{{\sf AC}}

\newcommand{\gradu}{{|\nabla u|}}

\newcommand{\alme}{\mm\text{-a.e.}}

\newcommand{\spn}{{\rm span}}
\newcommand{\hess}{{\rm Hess}}
\newcommand{\ww}{\wedge}

\renewcommand{\aa}{\mathfrak{a}}

\begin{document}

	\title[Regularity for Elliptic Equations in Metric Spaces]{Regularity  for quasilinear elliptic equations \\ in metric measure spaces
	}

	\author[S.~M.~Schulz]{Simon Schulz}
	\address[S.~M.~Schulz]{Laboratoire de Math\'ematiques de Versailles, UVSQ, Universit\'e Paris-Saclay, CNRS, 45 Av.~des \'Etats-Unis, 78000 Versailles Cedex, France}\email{simon.schulz@uvsq.fr}
	
	\author[I.~Y.~Violo]{Ivan Yuri Violo}
	\address[I.~Y.~Violo]{Universit\`a di Pisa, Dipartimento di Matematica, Largo Bruno Pontecorvo, 5, 56127 Pisa, Italy}
	\email{ivanyuri.violo@dm.unipi.it }

	

	\begin{abstract}
    {
    In the present article we prove second-order and Lipschitz regularity  for quasilinear elliptic equations in metric spaces endowed with a lower bound on the Ricci curvature. The estimates we obtain  are quantitative and  cover a large class of  elliptic equations with polynomial growth. As a particular case we settle the Lipschitz regularity of $p$-harmonic functions for all values of $p\in(1,\infty)$, {proving also a Cheng-Yau type inequality}. These results are the first in this setting that simultaneously address  a wide family of elliptic operators and extend beyond the classical H\"older regularity theory. Our strategy rests on the use of Galerkin's method, which we employ as an alternative to the traditional difference quotients technique.} 
	\end{abstract}
	
	\maketitle

    \begin{small}
	\setcounter{tocdepth}{1}
	\tableofcontents
	\end{small}
	
	\section{Introduction}
	This manuscript is concerned with the study of the regularity properties of weak solutions of quasilinear elliptic equations in metric measure spaces $\Xdm$, of the form 
	\begin{equation}\label{eq:intro eq}
		\div(\Psi(|\nabla u|)  \nabla u) = f, 
	\end{equation}
	where the conductivity $\Psi$ is a prescribed non-negative scalar function satisfying suitable coercivity and growth conditions {(see \eqref{eq:psi' condition}, \eqref{eq:psi growth})}, and $f$ is a given source term. A model example is the classical $p$-Laplacian operator
	\begin{equation}\label{eq:intro eq p laplace}
		\Delta_p u=\div(|\nabla u|^{p-2}  \nabla u) = f,
	\end{equation}
	which for $p=2$ reduces to the standard Laplacian $\Delta u$. For the rigorous definition of the above equations in metric setting we refer to \S \ref{pre:sobolev}. Observe that \eqref{eq:intro eq} is the Euler-Lagrange equation of the following functional
	\begin{equation}\label{eq:functional intro}
		F_\Phi(u)\coloneqq \int_\X \Big( \Phi(|\nabla u|)+f u \Big) \d \mm, \qquad 
		\Phi(t)\coloneqq \int_0^t t\Psi(t). 
	\end{equation}

	Over the past decades, the study of nonlinear elliptic operators and their solutions in non-smooth settings has attracted considerable attention.
	A well-established fact is that the De Giorgi-Nash-Moser method  can be extended to metric measure spaces under two fundamental assumptions: {firstly,} the space supports a \textit{local Poincaré inequality}, and {secondly} the reference measure satisfies a \textit{doubling condition}. These hypotheses make it possible to obtain interior H\"older regularity and Harnack inequalities. {These results} rely primarily on the observation that {the aforementioned} assumptions suffice to establish a local Sobolev embedding, which in turn provides the necessary framework to carry out the classical iteration techniques (see \textit{e.g.}~\cite{KS01,SaloffeHarnack,VIOLO2025}). Let us mention that there is by now an extensive literature about elliptic equations in this setting (see \textit{e.g.}~\cite{nastasi2025regularity,BB13,LMP06,lahti2018analog,maly2018neumann} and references therein).
	
	On the other hand it is natural to ask if  smoothness of solutions holds beyond H\"older regularity which, being in a metric setting, translates to Lipschitz regularity. Under  {the aforementioned set of}  assumptions this is not possible, as there are examples of harmonic functions  which are not better than H\"older continuous \cite{Koskela-Rajala-Shanmugalingam03}. The natural additional assumption turns out to be a suitable notion of Ricci curvature bounded below.

	The notion of Ricci curvature bounded below for metric spaces was introduced in the seminal works of Lott and Villani \cite{Lott-Villani09}  and Sturm \cite{sturmII}. They introduced the  $\CD(K,N)$ \textit{Curvature Dimension condition} which prescribes in a synthetic way a lower {bound} for the Ricci curvature by $K\in \rr$  and an upper bound for the dimension $N\in [1,\infty]$. This definition is given by suitable convexity inequalities for displacement interpolation via optimal transport.  In fact, to deal with PDEs it {is} more convenient to use the stricter $\RCD(K,N)$ \textit{Riemannian Curvature Dimension condition}. Roughly speaking, this allows for refined calculus tools {and}  differential operators and, more importantly, integration by parts. For example the Laplacian operator in $\CD$ spaces might not be even linear (see \cite{Gigli12}). For background we refer to the surveys \cite{AmbICM,Gigli23_working,Villani2016,sturmSurvey}.
	
	Recall that a Ricci curvature lower bound is a natural assumption in the context of  PDEs on curved spaces, already in the setting of Riemannian manifolds. The most famous {classical estimates  which are tightly connected to the Ricci curvature}  are the Cheng--Yau gradient bound for harmonic functions \cite{ChengYau75}, generalized to  $p$-harmonic functions by Wang and Zhang \cite{WZ11}, and the Li--Yau inequality for the heat equation \cite{LiYau}.
	
	In non-smooth $\RCD(K,N)$ spaces it was shown in \cite{Jiang13} (building upon \cite{Koskela-Rajala-Shanmugalingam03}) that harmonic functions are indeed locally Lipschitz (see also \cite{Gigli2023Regularity,MondinoSemola2022Lipschitz} for the case of harmonic maps). More generally it is known that a solution to 
	\[
	\Delta u=f \in L^q,
	\]
	is locally Lipschitz provided $q>N$ (see \cite{Kell13,Jiang12}). Generalized Cheng--Yau { and Li--Yau} inequalities also holds in this setting \cite{hua2013harmonic,garofalo2014li} together with a version of the Weyl's lemma \cite{ZZweyl}.
	
	In the present work, we extend the local Lipschitz regularity result to the $p$-Laplacian.
	\begin{thm}[Regularity for the $p$-Poisson equation]\label{thm:main plap}
		Fix $p\in(1,\infty).$ Let $\Xdm$ be an $\RCD(K,N)$ space, $N<\infty,$ and $\Omega\subset \X$ be open and bounded.  Suppose $u\in W^{1,p}(\Omega)$  solves
		\begin{equation}\label{eq:p poisson thm}
			\Delta_p u=f \in L^q(\Omega).
		\end{equation} for some $q>N$. Then $u \in \LIP_{\loc}(\Omega).$
	\end{thm}
	As a particular case we obtain the Lipschitz regularity of $p$-harmonic functions.  { In fact we also prove a Cheng-Yau type inequality, stated below.}   We emphasise at this point that the results obtained in this manuscript extend to a class of more general elliptic operators with $p$-growth (see our third main result, Theorem \ref{thm:p-delta regularity}). 
	\begin{cor}[Cheng-Yau type gradient estimate]\label{cor:pharm}
		{Let $p\in(1,\infty)$, $\Xdm$ be an $\RCD(K,N)$ space, $N<\infty,$  and $u$ be  positive  and $p$-harmonic in $B_R(x)$. Then $u \in \LIP_\loc(B_R(x))$ and
        \begin{equation}\label{eq:CY}
            \||\nabla \log(u)|\|_{L^\infty(B_{R/2}(x))}\le C_N\frac{1+R\sqrt{K^-}}{R},
        \end{equation}
        where $C_N$ is a constant depending only on $N$, and $K^- = -\min\{0,K\}$.
        }
	\end{cor}
	Both Theorem \ref{thm:main plap} and Corollary \ref{cor:pharm} are new for all values of $K,N$ and $p\neq 2$, even in the case of Ricci limit spaces and finite dimensional Alexandrov spaces with sectional curvature bounded below ({for the case $p=2$ in Alexandrov spaces see \cite{petrunin1996subharmonic,zhang2012yau,zhang2018lipschitz}}). In particular Corollary \ref{cor:pharm} gives an answer to \cite[Question 3.11]{Gong2010} and  to the open question at the end of the introduction of \cite{LZZ24}.

	Previously, the second author and Benatti \cite{ivan1} obtained {the conclusion of} Theorem \ref{thm:main plap} under the additional assumptions that $\X$ is bounded,  $p$ is sufficiently close to 2  and only for global solutions, in the sense that $\Omega=\X$.  Under the same  restrictions (only removing assumption $\Omega=\X$) {the authors also obtained} the local Lipschitz regularity of  the particular subclass of $p$-harmonic functions  in $\Omega$ having relatively compact level sets. Let us stress that the present work  is not a mere technical improvement, since the extra assumptions present in \cite{ivan1} seem  be an intrinsic  limitation  of the techniques employed  {therein}, especially concerning the range of the parameter $p$ and the issue with local solutions. 
	
	For context, we recall  the usual strategy employed in $\rr^n$ to obtain regularity for the model case of the $p$-Laplacian. We consider for simplicity the case of the Dirichlet problem 
	\begin{equation*}
		\begin{cases}
			\Delta_p u =f, & \text{in $\Omega\subset \rr^n$},\\
			u\in W^{1,p}_0(\Omega).
		\end{cases}
	\end{equation*}
	The  {standard} argument  {in the Euclidean setting} can be divided in three steps:
	\begin{enumerate}
		\item Prove the regularity of solutions for the following regularized problem for all $\delta>0:$ 
		\begin{equation*}
			\begin{cases}
				\Delta_{p,\delta}\,u_\delta =f, & \text{in $\Omega$},\\
				u_\delta\in W^{1,p}_0(\Omega), 
			\end{cases}
		\end{equation*}
		{where the \textit{$(p,\delta)$-Laplacian} is given by  $ \Delta_{p,\delta}\,u_\delta\coloneqq \div((|\nabla u_\delta|^2+\delta)^\frac{p-2}{2}\nabla u_\delta)$;}
		\item Obtain uniform regularity estimates for $u_\delta$ independent of $\delta$;
		\item Show that $u_\delta \to u$ in $W^{1,p}(\Omega)$ and deduce the regularity of $u.$
	\end{enumerate}
	In the classical setting of $\rr^n$ the main technical step  is (2). Indeed the regularized operator is \textit{non-degenerate uniformly elliptic}  and thus (1) follows from standard results   \cite{UralLadybook,giustibook}, while step (3) is well known \textit{e.g.}~by variational arguments. On the other hand, in the non-smooth setting step (1) also turns out to be far from trivial. The reason is the lack of a  general second order and Lipschitz  regularity theory for elliptic operator in metric spaces, even with Ricci curvature bounded below. In the classical setting, {the orthodox approach is to proceed by means of \textit{difference quotients}}, which allow us to treat in a unified way a fairly general class of equations. { However, there seems to be little hope of extending this tool  to the non-smooth  and curved settings. To the knowledge of the authors, few attempts have been made in this direction and with limited potential for generality; the underlying idea is to study the variations of the solution along a geodesic {or suitable coordinates vector fields} and then infer the regularity of the solution itself from the composition of the maps, \textit{cf.}~\cite{MingioneHeisenberg} in the Heisenberg group and the ideas presented in \cite{JacquesSimonComposition}. {In our setting, the charts and families of independent vector fields possess merely Borel regularity, which makes this strategy particularly difficult to implement, if not impossible.} }  
    
    The only  second-order regularity result known so far, under synthetic Ricci curvature bounded below,  is the analogue of the classical Calder\'on--Zygmund estimate for $p=2$:
	\begin{equation}\label{eq:CZ intro}
		\Delta u \in L^2 \implies u \in W^{2,2}.
	\end{equation}
	This was proved in \cite{Gigli14} building upon \cite{Savare07} (see Lemma \ref{lem:gradgrad} for the precise statement). The proof of \eqref{eq:CZ intro} uses the special role of the Laplacian operator in the Bochner inequality combined with heat flow  and does not obviously generalize to any  operator other than the Laplacian.
	The definition of the space $W^{2,2}$ is also due to \cite{Gigli14} and it being  non-trivial  is also thanks to  \eqref{eq:CZ intro}. 
    
    In \cite{ivan1},  to circumvent the absence of the difference quotients method, the authors exploit that the operator $\Delta_{p,\delta}$  for $p$ close to 2 is, roughly speaking, {a small perturbation of } the Laplacian.  Building upon  this fact and with a quite involved combination of fixed point arguments allowed to transfer the second order regularity from $\Delta $ to the regularized operator $\Delta_{p,\delta}$. However, this only works for a suitable range of $p$ and under the extra assumptions reported above.
	
	 The method that we will use here is substantially different from \cite{ivan1}.  We will still follow the path to regularity  (1)-(2)-(3) and the main effort is still to deal with the regularity issue in (1), however  we employ an alternative strategy. The core idea is to apply \textit{Galerkin’s method} to approximate the solutions of a PDE. This method tackles boundary value problems by projecting them onto a finite-dimensional subspace, representing the solution as a linear combination of basis functions. In our case, the natural choice of basis functions will be the\textit{ eigenfunctions of the Laplacian} with zero Dirichlet boundary conditions. The key point  is that these functions are already known to be smooth, {\textit{i.e.}\ in $W^{2,2}$, by \eqref{eq:CZ intro} (\textit{cf}.\ Remark \ref{rmk:trhess})}. {We stress that, even with this choice {of basis functions}, our final results hold in  general and not only for zero Dirichlet boundary conditions.} A major advantage of this approach is its flexibility. As a matter of fact it allows us to treat not only the $p$-Laplacian and $(p,\delta)$-Laplacian, but also a wide class of quasilinear operators in a unified way.

	We recall that Galerkin's method is a classical technique, which has been traditionally used to obtain the existence of solutions to elliptic (resp.~parabolic) problems by reducing the PDE to a linear system (resp.~ODE system). To the knowledge of the authors, employing Galerkin's method as a tool for regularity is a relatively underutilized strategy, despite the idea being mentioned in classical texts (\textit{cf.}~\textit{e.g.}~\cite[\S 7.1.3]{Evans}). This is likely because proceeding by means of difference quotients is typically more direct in the Euclidean setting. Nevertheless, in $\mathbb{R}^n$, the method has been used to \emph{change from the divergence to the non-divergence form} of the equation by exploiting cancellations (\textit{cf.}~\textit{e.g.}~\cite{reg2,reg2Err}) and thereby produce improved estimates---we highlight that, in the setting of \cite[\S 6.2]{reg2}, both difference quotients and mollification arguments destroy the aforementioned cancellations and fail to deliver the sought regularity results.
    
    One of the secondary objectives of this paper is to illustrate the robustness of Galerkin's method as a means of obtaining regularity, {also in the abstract setting of metric measures spaces}. More generally, even in $\mathbb{R}^n$, we believe that the techniques presented in this manuscript can be used where the classical method of difference quotients fails. 
    
    Let us also mention that  the $p$-Laplacian operator and non-linear equations  of the form \eqref{eq:intro eq} have applications to geometric analysis, potential theory and geometric measure theory in both the smooth and non-smooth setting, where regularity plays a fundamental role (see \textit{e.g.}\ \cite{benatti2024minkowski,mari_flowlaplaceapproximationnew_2022,benatti2024fine,cucinotta2025minimal,agostiniani2023new}).  {Finally we highlight  that, while this work is focused on the metric space setting, the study of the regularity of solutions of variational problems in $\mathbb{R}^n$ with $p$-growth (or, more generally, $(p,q)$-growth) prevails as an active field of study. We refer both to the classical works \cite{giustibook,GiaquintaBook,lieberman1991natural,tolksdorf_regularitymoregeneralclass_1984,dibenedetto_alphalocalregularityweak_1983}, as well as the more recent contributions \cite{CianchiMazya,DMnonuniformly,mingioneSurvey,minDesurvey}. 
    }

	{
		Below we present the main outcome of our analysis. We obtain two general regularity results for uniformly elliptic equations in the form \eqref{eq:intro eq}, Theorems \ref{thm:main quasilinear} and \ref{thm:p-delta regularity}:  the first in the case for $\Psi$ bounded above and away from zero  and second for   $\Psi$ satisfying suitable $p$-growth conditions.
	}

	{We first need to introduce our main ellipticity condition on $\Psi$.}
	We assume $\Psi \in {\sf AC}_\loc(0,\infty)$ is  positive and  satisfies
	\begin{equation}\label{eq:psi' condition}
		-1<\lambda\le \frac{t\Psi'(t)}{\Psi(t)}\le \Lambda<\infty,  \quad \text{a.e.\ $t>0 $,}
	\end{equation}
	for some $\lambda,\Lambda \in \rr$.  In the smooth setting assumption \eqref{eq:psi' condition} is standard  and  it ensures uniform ellipticity\footnote{
			For quasilinear equations in $\rr^n$ uniform ellipticity classically means that: $\sup_{x\in \Omega,z\in \rr^n} \frac{\Lambda(x,z)}{\lambda(x,z)}<\infty$, where $\lambda(x,z)$ and $\Lambda(x,z)$ denote the minimum and maximum eigenvalue of the matrix coefficients $a_{i,j}(x,z)$. In particular the $p$-Laplacian \textit{is uniformly elliptic} (see \textit{e.g.}~\cite[Chapter 10]{GilbargTrudinger} or \cite{mingioneSurvey}).
		} of the operator $\div(\Psi(|\nabla u|)\nabla u)$ (see \textit{e.g.}~\cite{CianchiMazya,lieberman1991natural,mingioneSurvey}). The $p$-Laplacian and $(p,\delta)$-Laplacian correspond respectively to  $\Psi(t)=t^{p-2}$ and  $\Psi(t)=(t^2+\delta)^\frac{p-2}2$, which both satisfy \eqref{eq:psi' condition} with  $\lambda=\Lambda=p-2$. Note finally that the first in \eqref{eq:psi' condition} implies $t\Psi(t)$ is non-decreasing and thus the corresponding functional \eqref{eq:functional intro} is convex.
    
    We can now state our first general regularity statement for quasilinear equations.
    \begin{thm}[Non-degenerate quasilinear equations]\label{thm:main quasilinear}
		Let $\Xdm$ be a locally compact $\RCD(K,\infty)$  space and $\Omega\subset \X$ be open, bounded with $\mm(\X\setminus\Omega)>0$. Fix   $f\in L^2(\Omega)$ and $g\in \W(\Omega)$.  Suppose $\Psi\in C^1(0,\infty)$ satisfy \eqref{eq:psi' condition} and  
        \begin{equation}\label{eq:above and below}
            c\le \Psi \le c^{-1}, \quad \text{for some constant $c\in(0,1).$}
        \end{equation}
		Then, there exists a unique weak solution $u\in g+\W_0(\Omega)$ of \footnote{
			Throughout the  {paper} the product  $\Psi(|\nabla u|)\nabla u$ is taken to be zero whenever $|\nabla u|=0$.
		} 
		\begin{equation}\label{eq:quasilinear equation thm6}
			\div(\Psi(|\nabla u|)\nabla u)=f.
		\end{equation}
		Moreover, $\Delta u \in L^2_\loc(\Omega)$ and for all $\Omega'\subset \subset \Omega$  it holds
		\begin{equation}\label{eq:quasilinear estimate}
			\int_{\Omega'} |\Delta u|^2\d \mm \le C \int_\Omega \Big( f^2+ g^2+|\nabla g|^2 \Big) \d \mm,
		\end{equation}
		where $C>0$ is a constant depending only on  $c,\lambda $, $K, \Omega $ and $\Omega'.$
	\end{thm}
	{Note that the above theorem immediately extends to local solutions, \textit{i.e.}\ if $u\in \W(\Omega)$ solves \eqref{eq:quasilinear equation thm6} in $\Omega$ then $\Delta u \in L^2_\loc(\Omega)$,  with no need of boundary conditions or assumptions on  $\Omega$. }
	We also highlight in passing that we do not require a compatibility condition between $f$ and $g$ in the statement of Theorem \ref{thm:main quasilinear}. This is because we are concerned with boundary value problems, not Neumann-type boundary conditions.

    In the next result we replace \eqref{eq:above and below} by the following $p$-growth condition:
	\begin{equation}\label{eq:psi growth}
		\nu^{-1} t^{p-2}\le \Psi(t) \le \nu t^{p-2}, \quad  \text{for all $t\ge 1$,}
	\end{equation}
	for   $p\in(1,\infty)$ and some constant $\nu >1.$    The main role of \eqref{eq:psi growth}  is to give  existence of energy solutions of \eqref{eq:intro eq} in the space $W^{1,p}$ and to allow for variational methods. {Together   \eqref{eq:psi' condition} and \eqref{eq:psi growth} are a particular case of the  classical conditions introduced by   Ladyzhenskaya and Uraltseva \cite{UralLadybook} for quasilinear uniformly  elliptic operator with $p$-growth  (see \textit{e.g.}~\cite{mingioneSurvey,minDesurvey,lieberman1991natural,tolksdorf_regularitymoregeneralclass_1984}).} {Requiring \eqref{eq:psi growth}  only for $t\ge 1$ allows some extra flexibility in the fact that $\Psi$ does not need to have $p$-growth near zero. For example $\Psi(t)=\min\{t^{p_1-2},t^{p_2-2}\}$ satisfies \eqref{eq:psi growth} with $p=\max\{p_1,p_2\}$, and \eqref{eq:psi' condition}. }

	\begin{thm}[Quasilinear operators with $p$-growth]\label{thm:p-delta regularity}
		Let $\Psi \in C^1(0,\infty)$ satisfy \eqref{eq:psi' condition}  with $\lambda,\Lambda$ and the $p$-growth condition \eqref{eq:psi growth} with $p\in (1,\infty)$ and $\nu>1$.
		Let $\Xdm$ be an $\RCD(K,N)$ space, $N<\infty$, and $\Omega\subset \X$ be open.  Assume that for some $f\in L^{2}(\Omega)$ and $u\in W^{1,p}(\Omega)$  it holds
		\begin{equation}\label{eq:quasilinear equation delta}
			\div\big(\, \Psi(|\nabla u|\big)\,\nabla u\,\big)=f, \quad \text{in $\Omega.$}
		\end{equation}
		Then $ \Psi(|\nabla u|\big)\,\nabla u \in  H^{1,2}_{C,\loc}(T\X;\Omega)$ and in particular $\Psi(|\nabla u|)|\nabla u|\in \W_\loc(\Omega).$ Additionally, for all $B_R(x)\subset \subset \Omega$ with $R\le 1,$ the following estimates hold$:$
		\begin{enumerate}[label=\roman*)]
			\item 
			\begin{equation}\label{eq:key applied udelta}
						\fint_{B_{R/4}(x)}     \big|\,\nabla \big(\Psi(|\nabla u|)\nabla u\big)\, \big|^2 \d \mm \le C
						\fint_{B_R(x)}   f^2 \d \mm 
						+  C\left(\fint_{B_R(x)} \Psi(|\nabla u|\big)|\nabla u|\d \mm\right)^2,
			\end{equation}
			where  $C>0$ is a constant depending only on $K,N,\Lambda,\lambda$ and $p$.
			\item 
			{if, {additionally}, $\fint_{B_{R}(x)} |f|^{q} \d \mm \le C_0$ for some $q>{\max\{N, 2\}}$, then 
				\begin{equation}\label{eq:gradient estimate for udelta}
					\| \nabla u \|_{L^\infty(B_{R/4}(x))}\le \tilde C \bigg(1+\fint_{B_{R}(x)}   \Psi(|\nabla u|\big)|\nabla u|\, \d \mm   \bigg)^\frac{1}{p-1},
				\end{equation}
				where $\tilde C>0$ is a constant depending only on $p,N,K,q,C_0,\lambda, \Lambda, \nu$.}
		\end{enumerate}

	\end{thm}

{Note that, once Theorem \ref{thm:p-delta regularity} is proved, Theorem \ref{thm:main plap}   follows immediately as special case.} { Corollary \ref{cor:pharm} is instead proved in Section \ref{sec:p harmonic}.} {Some comments are in order. 
\begin{itemize}
\item[--] The $\RCD(K,N)$ assumption in Theorem \ref{thm:p-delta regularity} can be replaced by\textit{ $\RCD(K,\infty)$ plus locally doubling}, since  finite dimensionality is used only for the local Sobolev embedding.
    \item[--] 	Both Theorem \ref{thm:main quasilinear} and Theorem \ref{thm:p-delta regularity} hold replacing the assumption $\Psi \in C^1(0,\infty)$ with \textit{$\Psi'$ continuous up to a negligible set of points}. In fact we will prove both results directly in this higher generality.  The advantage is that this weaker assumption  is preserved by truncations of the type $\Psi(t\wedge M).$ 
    \item[--]   The assumptions of Theorem \ref{thm:p-delta regularity}  (in the more general form given by the previous point) are  preserved by taking $\min$ or $\max$ and  by the transformations $\Psi(t\wedge M)$ and $\Psi(\sqrt{t^2+\delta})$, possibly  changing the constants $\lambda,\Lambda $ and $p.$ 
    \item[--]  If in Theorem \ref{thm:p-delta regularity} for $p\ge 2$ we add the non-degenerate\footnote{Note that \eqref{eq:psi growth} coupled with \eqref{eq:nondeg} includes the standard non-degenerate $p$-growth assumption \\ $\nu^{-1}(t^2+1)^\frac{p-2}{2}\le \Psi(t)\le \nu (t^2+1)^\frac{p-2}{2}$ for $t>0$ for $p\ge 2.$} assumption
    		\begin{equation}\label{eq:nondeg}
			0<c\le 	\Psi(t) , \quad t> 0,
		\end{equation}
then $u\in H^{2,2}_\loc(\Omega)$.  This matches the corresponding  result in the Euclidean case (see \textit{e.g.}\ \cite{UralLadybook,giustibook}).     For $1<p<2$ the situation is more subtle, see Remark \ref{rem:nondeg detailed} for details.
\item[--]  Morally, Theorem \ref{thm:main quasilinear} is the non-degenerate version of Theorem \ref{thm:p-delta regularity} for $p=2$. In fact, roughly speaking, we will prove the latter using the former via the transformation $\Psi(\sqrt{(t \wedge M)^2+\delta})$, which clearly satisfies \eqref{eq:above and below}, and by sending $M\to +\infty$ and $\delta \to 0.$ This interplay between exponent $p$ and 2 will require careful and rather technical approximation arguments  in \S \ref{sec:p harmonic}-\ref{sec:proof of propositions}.
\end{itemize}}

		\begin{rem}[Non-autonomous equations]
			Theorem \ref{thm:p-delta regularity} holds 
			also for equations 
			\[
			\div\big(\aa(x)\Psi(|\nabla u|)\nabla u\big)=f,
			\]
			for any {scalar function} $\aa\in \LIP(\Omega)$ satisfying $A^{-1}\le \aa\le A $ for some constant $A\ge 1$ (taking the constants in the regularity estimates depending on $\Lip(\aa)$ and $A$ as well). For clarity of exposition we present only the simpler autonomous case, but straightforward modifications to the argument lead to this more general version, see Remark \ref{rmk:non autonomous summary} for more details.
			\fr
		\end{rem}

{
\begin{rem}[No $C^{1,\alpha}$-regularity]
    In the Euclidean setting $p$-harmonic functions and more generally solutions of quasilinear uniformly elliptic equations of the form \eqref{eq:intro eq} with $\Psi\in C^1$ are $C^{1,\alpha}$ (see \cite{tolksdorf_regularitymoregeneralclass_1984, dibenedetto_alphalocalregularityweak_1983}). However we cannot expect to obtain any continuity estimate on the gradient in our setting, even assuming a lower bound on the sectional curvature (see \cite{de2023behavior}).\fr 
\end{rem}}

   {
	We conclude the introduction by presenting some applications of our main results. {The proofs are given at the end of \S \ref{sec:p harmonic}.}

    \begin{cor}[Regularity for Sobolev minimizers]\label{cor:sobolev}
        Let $\Xdm$ be an $\RCD(K,N)$ space, $N<\infty$. Suppose that for some $p\in(1,N)$ it holds
        \begin{equation}\label{eq:sobolev}
            \|v\|_{L^{p*}(\mm)}^p\le A \||\nabla v|\|_{L^p(\mm)}^p+B\|v\|_{L^p(\mm)}^p, \quad \text{for all }  v \in W^{1,p}_\loc\cap L^{p^*}(\X),
        \end{equation}
        with $p^*\coloneqq \frac{pN}{N-p}$, $A>0$ and $B\in \rr.$ Then any $u \in W^{1,p}_\loc\cap L^{p^*}(\X)$ which satisfies equality in \eqref{eq:sobolev} is locally Lipschitz and $|\nabla u|^{p-1}\in \W_\loc(\X)$.
    \end{cor}
    The case $p=2$ of the above result was observed in \cite{NobiliViolo24}.

	\begin{cor}[Regularity for minimal surface equation]\label{cor:minimal surf}
        Let $\Xdm$ be an $\RCD(K,N)$ space, $N<\infty$, and let $\Omega\subset \X$ be open. Suppose that $u\in \LIP_\loc(\Omega)$ satisfies
        \begin{equation}\label{eq:minimal}
             \div\bigg( \frac{\nabla u}{\sqrt{|\nabla u|^2+1}}\bigg)=0, \quad \text{in $\Omega$.}
        \end{equation}
        Then $\Delta u \in L^2_\loc(\Omega).$
    \end{cor}
 Equation \eqref{eq:minimal}  in the RCD setting was considered in \cite{cucinotta2025minimal} in connection with perimeter minimizers. {See \cite{BombieriDGMirandaMinSurf,TrudingerMinSurf} for background on the condition $u \in \LIP_\loc$  {in the minimal surface equation in $\rr^n$}. }

	\begin{cor}[Regularity for $p$-eigenfunctions]\label{cor:eigein}
        Fix $p\in(1,\infty)$. Let $\Xdm$ be an $\RCD(K,N)$ space, $N<\infty$, and let $\Omega\subset \X$ be open. Suppose that $u \in W^{1,p}(\Omega)$ satisfies
        \[
        \Delta_p u=-g |u|^{p-2}u, \quad \text{in $\Omega$.}
        \]
        for some $g\in L^\infty(\Omega)$. Then $u \in \LIP_\loc(\Omega)$  and $|\nabla u|^{p-1}\in \W_\loc(\Omega)$.
    \end{cor}
    $p$-eigenfunctions in RCD setting were studied in \cite{MondinoSemola20,CavallettiMondino17} (see also \cite{NobiliViolo2025}). 
	}

	\section{Preliminaries}\label{sec:prelim}

	\subsection{Notations and conventions} 
	$(\X,\sfd)$ is a separable metric space and $\mm$ is a non-negative Borel measure finite on bounded sets with $\supp(\mm)=\X$. $(\X,\sfd)$ is said to be proper if closed bounded sets are compact.   $\LIP(\X)$ is the space of Lipschitz functions in $\X.$ For $\Omega\subset \X$ open we denote by $\LIP_c(\Omega)$  the subsets of $\LIP(\X)$ of functions having compact support  in $\Omega$.  $L^p_\loc(\Omega)$ is the space of functions $L^p$-integrable in all compact subsets of $\Omega.$ We denote by $\mathds{1}_E$ the indicator function of a set $E.$ We write $E\subset\subset \Omega$ to say  that $E$ is relatively compact in $\Omega.$ With $L^0(\Omega)$ we denote the space of measurable functions finite $\mm$-a.e. 
    The notation $\liminf,\limsup$ are respectively used to denote the liminf and limsup of a sequence,  $a \wedge b \coloneqq \min\{a,b\}$,  $a \vee b \coloneqq \max\{a,b\}$.

	\subsection{Sobolev spaces and divergence operators}\label{pre:sobolev}
	
	We assume the reader to be familiar with the notion of Sobolev spaces in metric measure spaces and refer to \cite{Gigli23_working,GP20} for an introduction. 
	
	Throughout the {paper} we assume the metric measure space $\Xdm$ to be \textit{infinitesimally Hilbertian}, \textit{i.e.}\ $\W(\X)$ is a Hilbert space (see \cite{Gigli12}), and to have \textit{strong independence of the weak upper gradient on the exponent} (see \cite{GigliNobili21}). This is true in particular in the setting of $\RCD(K,\infty)$ spaces of this article (see \cite{GigliHan14} and \S \ref{subsec:rcd}).
	These assumptions allow to define general first-order calculus objects without dependence on the Sobolev exponent.
	
	In particular we are allowed to denote by $|D f|\in L^p(\mm)$ the minimal weak upper gradient for a function $f\in W^{1,p}(\X)$, omitting the dependence on $p$. Moreover if $|D f|,f\in L^q(\mm)$ then $f\in W^{1,q}(\X).$ It was shown \cite{GigliNobili21}, that under these assumptions there exists a notion of tangent bundle $L^0(T\X)$ that has the structure of $L^0$-normed Hilbert module (see \cite{Gigli14,GP20} for the language of modules). We  recall  that $L^0(T\X)$ is endowed with a pointwise norm $|\cdot |: L^0(T\X)\to L^0(\mm)$, a pointwise bilinear scalar product $\la \cdot ,\cdot \ra$ which satisfies $\la v,v\ra=|v|^2$ $\mm$-a.e.\ and that elements $v\in L^0(T\X)$ can be multiplied by functions $f\in L^0(\mm)$ so that $|fv|=|f||v|$ $\mm$-a.e. Finally, as shown in \cite{GigliNobili21}, it can be defined  a linear gradient operator $\nabla: \cup_{p\in (1,\infty)} W^{1,p}(\X) \to L^0(T\X)$ satisfying $|\nabla f|=|Df|$ $\mm$-a.e.
	We always write $|\nabla f|$ in place of $|Df|$ to simplify the notation. We denote by $L^0(T\X)\restr \Omega$ the sub-module $\mathds{1}_\Omega\cdot L^0(T\X)$ (see \cite[\S 3.1]{GP20}) and by $L^p(T\X)\restr \Omega$ (resp.\ $L^p_\loc(T\X)\restr \Omega$) the subsets of elements $v\in L^0(T\X)\restr \Omega$ such that $|v|\in L^p(\Omega)$ (resp.\ $L^p_\loc(\Omega)$).
	
	We also have that (see \cite[Prop.\ 4.4]{GigliNobili21} for a proof):
	\begin{equation}
		\parbox{12cm}{ $W^{1,p}(\X)$ is uniformly convex (and thus reflexive) and  separable for all $p\in (1,\infty)$. Moreover the space $\LIP_{bs}(\X)$ is strongly dense in $W^{1,p}(\X)$.}
	\end{equation}
	
	 By $W^{1,p}_\loc(\Omega)$ we denote the space of functions $f\in L^p_\loc(\Omega)$ such that $f\eta \in W^{1,p}(\X)$ for all $\eta \in \LIP_c(\Omega).$ Moreover $W^{1,p}(\Omega)=\{f \in W^{1,p}_\loc(\Omega) \ : \ |\nabla f|\in L^p(\Omega)\}.$ All the above objects and properties naturally generalize to local Sobolev spaces $W^{1,p}(\Omega)$ and $W^{1,p}_\loc(\Omega)$. 
	For all $\Omega\subset \X$ open and $p\in(1,\infty)$ we define the space $W^{1,p}_0(\Omega)$ as the closure of $\LIP_{c}(\Omega)$ in $W^{1,p}(\X)$ and
	\begin{equation}\label{eq:hat w1p}
		\widehat W^{1,p}_0(\Omega) \coloneqq \{f \in W^{1,p}(\X) \ : \  f=0\,\, \mm\text{-a.e.\ in $\X\setminus \Omega$ }\}.
	\end{equation}

	The following is proved in \cite{AHlocal} for $p=2$, but the argument works for all $p\in(1,\infty)$.
	\begin{lemma}\label{lem:good balls}
		Let $\Xdm$ be a proper metric measure space 
		with $W^{1,p}(\X)$ 
		separable and fix 
		$x\in \X.$ Then, it holds
		$W^{1,p}_0(B_r(x))=\widehat W^{1,p}_0(B_r(x))$ for a.e.\ $r>0$. 
	\end{lemma}

	\begin{definition}[Divergence operator]\label{def:div}
		Let $\Omega\subset \X$ be open and $v \in L^1_\loc(T\X)\restr{\Omega}$. We say that $v \in \dom(\div)$ provided there exists a function $f \in L^1_\loc(\Omega)$ satisfying
		\begin{equation}\label{eq:def div}
			-\int_\Omega \la v, \nabla \phi \ra\d \mm=\int_\Omega f \phi \d \mm, \quad \forall \phi \in \LIP_{c}(\Omega),
		\end{equation}
		in which case such $f$ is unique and will be denoted by $\div(v).$ 
	\end{definition}
	The dependence on $\Omega$ will be always clear from context, when needed we may write $\dom(\div,\Omega)$ and $\div_\Omega(v).$ We will use the short-hand   $\div(v)\in V$ (resp.\ $\div(v)=f$), for some space $V\subset L^1_\loc(\Omega)$ (resp.\ function $f\in L^1_\loc(\Omega)$), to say at once that $v \in \dom(\div)$ and $\div(v)\in V$ (resp.\ $\div(v)=f$).   The same for the Laplacian and $p$-Laplacian operators defined below.

	\begin{rem}[Testing against Sobolev functions]\label{rmk:sobolev test}
		Suppose $v\in \dom(\div)$ and $|v| \in L^{q}(\Omega)$ for some $q$. Then, \eqref{eq:def div} holds also for all $\phi \in W^{1,p}_0\cap L^\infty(\Omega)$ with $p=q/(q-1)$ the usual conjugate exponent. Additionally if  $\div(v)\in L^{q}(\Omega)$  then \eqref{eq:def div}  holds 
		for all $\phi \in W^{1,p}_0(\Omega)$.  Both assertions can be 
		checked by a density argument and will be used without further notice. \fr
	\end{rem}
	
	\begin{definition}[Laplacian]
		Let $\Omega\subset \X$ be open and let $u\in W^{1,2}(\Omega).$ We say that $u \in \dom(\Delta)$ provided $\nabla u\in \dom(\div)$, in which case we set $\Delta u\coloneqq \div(\nabla u)$.
	\end{definition}

	\begin{definition}[$p$-Laplacian]
		Let $p\in(1,\infty)$, $\Omega\subset \X$ be open, and $u\in W^{1,p}(\Omega).$ We write $u \in \dom(\Delta_p)$ provided $|\nabla u|^{p-2}\nabla u\in \dom(\div)$, in which case we set $\Delta_p u\coloneqq \div(|\nabla u|^{p-2}\nabla u)$.
	\end{definition}
	
	For the following  see \cite[Prop.\ 4.2]{GigliViolo23} or also \cite{Gigli12}.
	\begin{prop}[Leibniz rule for the divergence]\label{prop:leibniz rule div}
		Let $\Omega\subset \X$ be open. Let $u \in \W(\Omega)$ be such that $u \in \dom(\Delta,\Omega)$ with $\Delta u \in L^2(\Omega)$ and $g \in \W_\loc\cap L^\infty(\Omega)$. Then $g\nabla u\in \dom(\div,\Omega)$ and
		\begin{equation}
			\div(g\nabla u)=g\Delta u +\la \nabla g,\nabla u\ra, \quad \alme \text{ in $\Omega$.}
		\end{equation}
	\end{prop}

	In the paper we recurrently use the following elementary convergence result. {In order to state it, we recall the notion of having a continuous derivative up to a negligible set: for $\Psi \in \LIP(\mathbb{R})$, we say that $\Psi'$ is \emph{continuous up to a negligible set of points} if there exists a Lebesgue-null set $N \subset \mathbb{R}$ such that $\Psi$ is differentiable in $\mathbb{R}\setminus N$ and $\Psi'$ is continuous when restricted to $\mathbb{R} \setminus N$. This definition is independent of the representative chosen for $\Psi'.$
	}

	\begin{lemma}\label{lem:psi' convergence}
		Let $\Xdm$ be a metric measure space and $f_n\in \W(\X)$ be a sequence converging strongly to $f\in \W(\X).$ Let also $\Psi \in \LIP(\rr)$ be such that $\Psi'$ is continuous up to a negligible set of points. Then, up to a subsequence which we do not relabel, there holds 
		\begin{equation}\label{eq:grad psi convergence}
			|\nabla (\Psi(f_n)-\Psi(f))|\to 0, \quad \mm\text{-a.e..}
		\end{equation}
		In particular, provided $\Psi(0)=0,$ we have $\Psi(f_n)\to \Psi(f)$ strongly in $\W(\X)$.
	\end{lemma}
	
	 The proof is an exercise using the chain rule; we omit it for concision.

	\subsection{$\RCD$ spaces}\label{subsec:rcd}
    We introduce the class of $\RCD(K,N)$ spaces, using the equivalent  definition due to \cite{EKS15,AmbrosioGigliSavare12,AmbrosioMondinoSavare13}.
	\begin{definition}[RCD space]
		An infinitesimally Hilbertian metric measure space $\Xdm$ is said to be an $\RCD(K,N)$ space, $K\in \rr$, $N\in [1,\infty]$ provided:
		\begin{enumerate}[label=\roman*)]
			\item $\mm(B_r(x_0))\le a e^{br^2}$ for all $r>0$ and some $x_0 \in \X$ and constants $a>0,\,b>0;$
			\item \textit{Sobolev-to-Lipschitz}: $f\in \W(\X)$ with $|\nabla f|\le 1$ $\implies$ $f$ has  1-Lipschitz representative;
			\item \textit{Weak Bochner inequality}: for all $u,\phi \in \W(\X)$ with $\Delta u \in \W(\mm)$ and $\phi,\Delta \phi \in L^\infty(\mm)$ with $\phi\ge 0$ it holds
			\begin{equation}\label{eq:weak bochner}
				\frac12\int \Delta \phi|\nabla u|^2 \d \mm \ge \int  \bigg(\frac{(\Delta u)^2}{N}+\la \nabla u,\nabla \Delta u\ra  + K|\nabla u|^2\bigg)\phi \d \mm.
			\end{equation}
		\end{enumerate}
	\end{definition}
	In the sequel  we refer to \cite{GP20,Gigli14,Gigli23_working} for the definition of the space $W^{2,2}(\X)$. We limit ourselves to recall that for all $u\in W^{2,2}(\X)$ there is a notion of the Hessian, denoted by $\hess (u)$, which can be equivalently  regarded as a linear map $\hess (u): L^0(T\X)\to L^0(T\X)$ satisfying $|\hess(u)(v)|\le |\hess u||v|$, where  $|\hess (u)|\in L^2(\mm)$ is a generalized Hilbert-Schmidt norm of $\hess (u).$ Likewise, $W^{2,2}(\Omega)$ and $W^{2,2}_\loc(\Omega)$ for $\Omega \subset \X$ open can be defined in the natural way.
	
	We recall the following  crucial regularity result obtained in \cite{Gigli14} building upon \cite{Savare07}.
	\begin{lemma}\label{lem:gradgrad}
		Let $\Xdm$ be an $\RCD(K,\infty)$ space and $u \in \W(\X)$ with $\Delta u \in L^2(\mm)$. Then $u \in W^{2,2}(\X)$, $|\nabla u|\in \W(\X)$ and
		\begin{equation}\label{eq:gradgrad}
			|\nabla |\nabla u||\le  |{\rm Hess} (u)|, \quad \alme
		\end{equation}
	\end{lemma}
	The closure of $\{u \in \W(\X) \ : \ \Delta u \in L^2(\mm)\}$ in $W^{2,2}(\X)$ is denoted by $H^{2,2}(\X).$
	We will need also the following extension property.
	\begin{lemma}\label{lem:cut-off lemma}
		Let $\Xdm$ be a {locally compact} $\RCD(K,\infty)$ space,  $\Omega\subset \X$ be open, $C\subset \Omega$  be compact and $u \in \W_\loc(\Omega)$ {(resp.\ $\LIP_\loc(\Omega)$)} be so that $\Delta u \in L^2_\loc(\Omega)$. Then there exists $\bar u\in \W(\X)$ {(resp.\ $\LIP_c(\X)$)} with $\Delta \bar u\in L^2(\mm)$ such that $u=\bar u$ in a neighborhood of $C.$ 
	\end{lemma}
	\begin{proof}
		From \cite[Lemma 6.7 and Remark 6.8]{AmbrosioMondinoSavare13-2} there exists a cut-off function $\chi \in \LIP_{c}(\Omega)$ such that $\chi \equiv 1$ in a neighborhood of $C$ and $\Delta \chi \in L^\infty\cap \W(\X).$ Then $\bar u\coloneqq \chi u$ satisfies the required properties; by the Leibniz rule, $\Delta \bar u=\chi \Delta u+\Delta \chi u + 2\la \nabla \chi, \nabla u\ra \in L^2(\mm)$.
	\end{proof}
	Thanks to Lemma \ref{lem:cut-off lemma} we can localize Lemma \ref{lem:gradgrad} in all open sets $\Omega\subset \X$  in the sense that:
	\begin{equation}\label{eq:gradgrad local}
		u\in \W(\Omega), \quad \Delta u \in L^2(\Omega)\implies u \in H^{2,2}_ \loc(\Omega),\quad |\nabla u|\in \W_\loc(\Omega).
	\end{equation}
	
	\begin{rem}\label{rmk:trhess}
		In this setting the Laplacian is not always equal to the trace of the Hessian. It is also not  known whether there are (non-constant) functions with Hessian in $L^q$, $q>2$. Hence, in view of Lemma \ref{lem:gradgrad},  functions with $L^2$-Laplacian  are the smoothest  we can get. 
		\fr
	\end{rem}
	
	Recall  the following class of test functions
	\[
	\test(\X) \coloneqq \left \{ u \in \LIP\cap W^{1,2}(\X) \ : \  \Delta u \in \W(\X)  \right\}.
	\]
	We report the following approximation result which we will use often in the sequel.
	\begin{lemma}[{\cite[Lemma 2.11]{ivan1}}]\label{lem:approximation lemma}
		Let $\Xdm$ be an $\RCD(K,\infty)$ space and $u\in \W(\X)$ be such that $\Delta u \in L^2(\mm)$. Then there exists a sequence $u_n \in \test(\X)$ such that  $\Delta u_n\to  \Delta u$ in $L^2(\mm)$, $u_n\to u$ in $W^{2,2}(\X)$ and  $|\nabla u_n|\to |\nabla u|$ in $\W(\X)$.
	\end{lemma}

		We recall  the space of vector fields with local covariant derivative in $L^2$ introduced in \cite{ivan1}. For the definition of $H^{1,2}_C(T\X)$ see \cite{Gigli14,GP20}.
		\begin{definition}[Local covariant derivative]\label{def:local covariant}
			Let $\Xdm$ be an $\rcd(K,\infty)$ space and $\Omega\subset \X$ be open. We define the space 
			\begin{equation}\label{eq:def loc covariant}
				H^{1,2}_{C,\loc}(T\X;\Omega)\coloneqq \{ v \in L^0(T\X)\restr{\Omega} \ : \ \eta v\in H^{1,2}_C(T\X),\, \forall \, \eta \in \LIP_{c}(\Omega)\}.
			\end{equation}
		\end{definition}
		For every $v \in H^{1,2}_{C,\loc}(T\X;\Omega)$ the function $|\nabla v|\in L^2_\loc(\Omega)$ is well defined by  locality (see \cite[Prop.\ 3.4.9]{Gigli14}). Moreover for all $v\in H^{1,2}_{C,\loc}(T\X;\Omega)$ it holds that $|v|\in W^{1,2}_\loc(\Omega)$ with $|\nabla |v||\le |\nabla v|$.

	We conclude recalling that $\RCD(K,\infty)$ spaces posses a local Poincaré inequality (see  \cite{Rajala12-2}). The following version  follows easily from the local one (see \textit{e.g.}~\cite[Corollary 5.54]{BB13}). 
	\begin{prop}[Poincar\'e inequality]\label{prop:poincare} Fix $p\in(1,\infty)$.
		Let $\Xdm$ be an $\RCD(K,\infty)$ space, with $K\in \rr,$ and $\Omega\subset \X$ be  bounded open  such that  $\mm(\X\setminus \Omega)>0$. Then
		\begin{equation}\label{eq:poincare}
			\int_{\Omega} |u|^p\d\mm\le C_{p,\Omega} \int_\Omega |\nabla u|^p\d \mm,\quad \forall u \in W^{1,p}_0(\Omega),
		\end{equation}
		where $C_{p,\Omega}>0$ is a constant depending only on $p$ and $\Omega.$
	\end{prop}

	\subsection{Bochner type inequalities}
	In  metric  spaces (or even Riemannian manifolds), the relationship between the Hessian and the Laplacian is more subtle than in the Euclidean setting, because of the presence of curvature (see also Remark \ref{rmk:trhess}). In particular the identity $\int |\hess(u)|^2 = \int |\Delta u|^2$ is not available and needs to be replaced by the Bochner inequality. In this subsection we outline the versions of this inequality that we use in this paper (see also \cite[Theorem 3.8.8]{Gigli14}, \cite{Savare13}, or also \cite[Corollary 6.2.17]{GP20}).
	  
	 {For the first statement  we need to introduce the notation $|{\bf H} u|^2\coloneqq |{\rm Hess} (u)|^2 + \frac{(\Delta u - tr{\rm Hess}(u) )^2}{N-n}$, where $n\coloneqq \dim(\X)$ (see \cite{Han18} for details).}
	\begin{thm}[{Improved Bochner inequality,  \cite{Gigli14,Han18}}]\label{thm:improved Bochner}
		Let $\Xdm$ be an $\RCD(K,\infty)$ space and $u\in \test(\X)$. Then $|\nabla u|^2\in W^{1,2}(\X)$ and 
		\begin{equation}\label{eq:bochner}
			-\int_{\X} \frac{\la \nabla \phi,\nabla {(} |\nabla u|^2 {)} \ra}{2} \d \mm \ge \int_{\X} \Big(|{\bf H} u|^2  +  \la \nabla u, \nabla \Delta u\ra  + K|\nabla u|^2\Big) \phi  \d \mm
		\end{equation}
		holds for all $\phi \in\W\cap L^\infty(\X)$ with $\phi\ge0$.  {If instead $u\in \LIP_{bs}(\X) $ with $\Delta u \in L^2(\mm)$ then \eqref{eq:bochner} still holds replacing $ \la \nabla u, \nabla \Delta u\ra\phi$ with  $-\div(\varphi \nabla u) \Delta u$}.
	\end{thm}
    \begin{proof}
       { The first part was proven in \cite{Han18}, while the second follows easily by integration by parts and approximation via the heat flow (alternatively apply Lemma \ref{lem:bochner type 2} with $\psi\equiv 1$).}
    \end{proof}
	
	Integrating \eqref{eq:bochner} we obtain the following result, which extends the classical Calderón–Zygmund inequality in $\rr^n$. In particular it quantifies the regularity result in Lemma \ref{lem:gradgrad}.
	\begin{cor}[Calderón–Zygmund inequality]\label{cor:hessian laplacian estimate}
		Let $\Xdm$ be a {locally compact} $\RCD(K,\infty)$ space, $\Omega\subset \X$ be open,  $u\in \W(\Omega)$ be such that $\Delta u \in L^2(\Omega)$. Then for all $\phi \in \LIP_{c}(\Omega)$  it holds 
		\begin{equation}\label{eq:hessian laplacian estiamte}
			\int_\Omega |{\rm Hess} (u)|^2\phi^2 \d \mm\le  4\int_\Omega (\Delta u)^2 \phi^2 \d \mm +C_K \int_\Omega  |\nabla u|^2(|\nabla \phi|^2 +\phi^2) \d \mm,
		\end{equation}
		where $C_K>0$ is a constant depending only on $K.$
	\end{cor}

	We will also need the following variant of \eqref{eq:bochner}.
	\begin{prop}[Bochner-type inequality I]\label{prop:bochner type}
		Let $\Xdm$ be a {locally compact} $\RCD(K,\infty)$ space and $u\in \W(\Omega)$ be such that $\Delta u \in \W(\Omega)$. Let also $\Psi\in \LIP(\rr)$ be  non-negative, with $\Psi(0)=0$ and such that $\Psi'$ is continuous up to a negligible set of points.  Then for all $\eta \in \LIP_c(\Omega)$ and constants $M>0$ it holds
		\begin{equation}\label{eq:bochner type}
			-\int_\Omega |\nabla u|\la \nabla \Psi_{M,\eta} ,\nabla |\nabla u| \ra \d \mm \ge \int_\Omega \Big( |{\rm Hess} (u)|^2  +  \la \nabla u, \nabla \Delta u\ra  + K|\nabla u|^2 \Big) \Psi_{M,\eta} \  \d \mm,
		\end{equation}
		where {$($with slight abuse of notation$)$ we write} $\Psi_{M,\eta} \coloneqq \Psi((|\nabla u|\wedge M)\eta^2) \in\W\cap L^\infty(\X)$.
	\end{prop}
	\begin{proof}
		Note that $\Psi_{M,\eta}\in \W\cap L^\infty(\X)$ with compact support in $\Omega$, which can be seen using the chain rule, that $|\nabla u|\wedge M\in \W_\loc(\Omega)$ by Lemma \ref{lem:gradgrad} and the assumption $\Psi(0)=0.$

		\medskip
		\noindent \textit{Step 1:}  Observe that it suffices to show that
		\begin{equation}\label{eq:bochner type modified}
			\begin{aligned} -\int_\Omega |\nabla u|\la \nabla \Psi_{M,\eta} ,\nabla |\nabla u| \ra \d \mm \ge &\int_\Omega \!\Big( |{\rm Hess} (u)|^2   + K|\nabla u|^2  \Big)\Psi_{M,\eta} \d \mm  \!\\
            &\quad-\int_\Omega \!  (\Delta u)^2\Psi_{M,\eta}+\la\nabla u,\nabla \Psi_{M,\eta}\ra\Delta u\d \mm.
			\end{aligned}
		\end{equation}
		Indeed \eqref{eq:bochner type} follows from \eqref{eq:bochner type modified}  integrating by parts. We also note that the integral on the left-hand side of \eqref{eq:bochner type modified} makes sense and is finite thanks to
		\begin{equation*}
			\nabla \Psi_{M,\eta} = \Psi'((|\nabla u | \wedge M)\eta^2) \Big( \eta^2 \mathds{1}_{\{|\nabla u| \leq M\}}\nabla |\nabla u| + (|\nabla u|\wedge M) \nabla \eta^2 \Big),
		\end{equation*}
		which implies
		\begin{equation}\label{eq:upper bound for nabla psi}
			|\nabla \Psi_{M,\eta}| |\nabla u| |\nabla |\nabla u||\le M \Lip (\Psi) (\eta^2|\nabla |\nabla u||^2+|\nabla \eta^2 ||\nabla u| |\nabla |\nabla u||)\in L^1(\mm). 
		\end{equation}
		
		\medskip 
		\noindent \textit{Step 2:} It is sufficient to show \eqref{eq:bochner type modified} for globally defined functions $u\in \W(\X)$ with $\Delta u \in L^2(\mm).$ Indeed by Lemma \ref{lem:cut-off lemma} there  exists $\bar u \in \W(\X)$ with $\Delta \bar u \in L^2(\mm)$ such that $\bar u=u$ in a neighborhood of $\supp(\eta).$ Moreover $\Psi_{M,\eta}=0$ outside the support of $\eta.$ Hence by locality \eqref{eq:bochner type modified} holds for $u$ if and only if it holds for $\bar u.$

		\medskip 
		\noindent \textit{Step 3:} It is sufficient to show \eqref{eq:bochner type modified} for functions $u\in  \test(\X).$  Indeed by Lemma \ref{lem:approximation lemma} there exists a sequence  $u_n \in \test(\X)$ such that  $\Delta u_n\to  \Delta u$ in $L^2(\mm)$, $u_n\to u$ in $H^{2,2}(\X)$  and  $|\nabla u_n|\to |\nabla u|$ in $\W(\X).$  
		We claim that \eqref{eq:bochner type modified} passes to the limit as $n\to +\infty.$ Up to passing to a subsequence we can assume {from the Generalised Dominated Convergence Theorem} that the functions $|\nabla u_n|$, $|\nabla |\nabla u_n||$, $|\hess (u_n)|$ and $\Delta u_n$ are all dominated {independently of $n$} by an $L^2(\mm)$-function and also converge pointwise $\mm$-a.e.\ to the corresponding expressions for $u$. Hence by Dominated Convergence Theorem it is sufficient to show pointwise $\mm$-a.e.\ converge of the integrands in \eqref{eq:bochner type modified}. For the terms without $\nabla \Psi_{M,\eta}$ this is immediate, so we focus on the others. Set $\Psi_{M,\eta}^n\coloneqq  \Psi((|\nabla u_n|\wedge M)\eta^2) .$  We observe that $|\nabla \Psi_{M,\eta}^n- \nabla \Psi_{M,\eta}|\to 0$ $\mm$-a.e., which follows by Lemma \ref{lem:psi' convergence} applied with $f_n=(|\nabla u_n|\wedge M)\eta^2$ and $f=(|\nabla u|\wedge M)\eta^2$. This settles the term with $\nabla \Psi_{M,\eta}$ on the right-hand side of \eqref{eq:bochner type modified}. For the left-hand side we write
		\begin{align*}
			|\la \nabla \Psi_{M,\eta}^n ,&\nabla |\nabla u_n| \ra\!-\!\la \nabla \Psi_{M,\eta} ,\nabla |\nabla u| \ra| \!\le\! |\nabla |\nabla u|||\nabla \Psi_{M,\eta}^n\!\!-\! \nabla \Psi_{M,\eta}|\!+\!| \nabla \Psi_{M,\eta}^n||\nabla |\nabla u|\!-\!\!\nabla |\nabla u_n||\!\to\! 0,
		\end{align*}
		$\alme$~as $n\to\infty$. {By recalling the aforementioned application of the Generalised Dominated Convergence Theorem and \eqref{eq:upper bound for nabla psi},} we have that \eqref{eq:bochner type modified} passes to the limit as $n\to\infty$.

		\medskip 
		\noindent \textit{Step 4:} 
		Inequality \eqref{eq:bochner type modified} indeed holds for all  $u\in \test(\X).$ This follows immediately from \eqref{eq:bochner} {of Theorem \ref{thm:improved Bochner}} by taking $\phi=\Psi_{M,\eta}$ integrating by parts the right-hand side and developing  the left-hand side (recall that $|\nabla u|\in L^\infty$ as $\test(\X)\subset \LIP(\X)$).
	\end{proof}
	
	The following is another version of the Bochner inequality that was proved in \cite{ivan1}.
	\begin{lemma}[{Bochner-type  inequality II, \cite[Lemma 4.8]{ivan1}}]\label{lem:bochner type 2}
		Let $\Xdm$ be $\rcd(K,\infty)$, $u \in \W(\X)$ with $\Delta u \in L^2(\mm)$,  $\psi \in \LIP\cap L^\infty(\rr)$ with $\psi'$ continuous up to a negligible set and 
		\begin{equation}\label{eq:psi' bound}
			|\psi '(t)|\le c (1+|t|)^{-1}, \quad \quad \text{ for a.e.\ $t \in\rr$,}
		\end{equation} for some constant $c>0$. Then, for every $\eta \in \LIP_{bs}(\X)$, $\eta \ge 0,$ it holds
		\begin{equation}\label{eq:starting point}
			\begin{split}
				&\int_\X{|{\bf H}u|^2 }
				\psi(|\nabla u|) \eta \d \mm \le 	\int_\X \la\nabla u\Delta u -|\nabla u|\,\nabla |\nabla u|, \nabla \eta\ra \psi(|\nabla u|)\d\mm\\
				&+ \!\!\int_\X \!\la\nabla u\Delta u \!-\!|\nabla u|\,\nabla |\nabla u|, \nabla |\nabla u|\ra \psi'(|\nabla u|)\eta \d\mm
				\!+\!\int_\X \! ((\Delta u)^2\!+\!K|\nabla u|^2)\psi(|\nabla u|) \eta\d\mm.
			\end{split}
		\end{equation}
	\end{lemma}

	\subsection{Eigenfunctions of the Laplacian}

	Consider $\Xdm$  an $\RCD(K,\infty)$ space.
	For any $\Omega \subset \X$ bounded, the embedding of $W^{1,2}_0(\Omega)$ in $L^2(\Omega)$ is compact (see \cite[Theorem 6.3-ii)]{GigliMondinoSavare13}) hence by spectral theory, the Dirichlet eigenvalues of the Laplacian form a discrete  sequence
	\[
	0= \lambda_1 \le  \lambda_2 \le \dots \le  \lambda_k \to +\infty,
	\]
	counted with multiplicity (see e.g.\ \cite[Lemma 2.8]{PFV25}). Moreover, the corresponding eigenfunctions
	\begin{equation}\label{eq:eig laplacian}
		\left\lbrace\begin{aligned}
			&\Delta \phi_k=-\lambda_k \phi_k \quad \text{in } \Omega, \\ 
			&\varphi_k \in W^{1,2}_0(\Omega), 
		\end{aligned}\right.
	\end{equation}
	with the renormalization $\|\phi_k\|_{L^2(\mm)}=1$, form an orthonormal base of $L^2(\Omega)$ and an orthogonal base of $W^{1,2}_0(\Omega)$;  
	in particular their linear span is dense in $\W_0(\Omega).$ For all $k\in \nn$ we define
	\begin{equation}\label{eq:Vk def}
		V_k := \spn \{\phi_1,\dots,\phi_k\}, 
	\end{equation}
	and we emphasise that $V_k$ depends on the choice of domain $\Omega$. An obvious but crucial observation is that $V_k$ is mapped to itself by the Laplacian operator.
	We finally observe that eigenfunctions satisfy by definition $\Delta \phi_k\in L^2(\Omega)$  and thus are ``smooth'' in the sense of  \eqref{eq:gradgrad local} (see Remark \ref{rmk:trhess}).

	\subsection{Minimization problems and Euler-Lagrange equations}\label{sec:minim}

	Here we collect existence, uniqueness and variational formulations  for the elliptic problems that we subsequently consider.

	For simplicity  throughout this section $\Xdm$ is an $\RCD(K,\infty)$ space. However all the results contained here  hold for any infinitesimally Hilbertian spaces with strong $p$-independence of weak upper gradients supporting a local Poincaré inequality (\textit{cf.}~ \cite[\S 2-3]{VIOLO2025}). 
	
	We start defining a general convex functional, for which we make the following definition. 
	
	\begin{definition}[Admissible $\Phi$]\label{def:admissible Phi}
		Let $p \in (1,\infty)$.	We say $\Phi: [0,\infty) \to \rr$ is \emph{$p$-admissible} if it is differentiable, convex, satisfies $\Phi'(0)=0$, and there exists a constant $C>0$ such that 
		\begin{equation}\label{eq:basic Phi bound}
			|\Phi(t)|\le C(1+|t|^p) \quad \text{for all $t\ge0$}.
		\end{equation}
	\end{definition}
	
	Inequality \eqref{eq:basic Phi bound} and the convexity condition imply the following derivative estimate
	\begin{equation}\label{eq:phi' bound}
		0 \leq \Phi'(t) \le 2^{p+2}C(1+|t|^{p-1}) \quad \text{for all $t\ge0$,}
	\end{equation}
	which straightforward verification is omitted  for succinctness. In particular $\Phi$ is monotone non-decreasing. For convenience of notation, we  set $\Psi(t)\coloneqq t^{-1}{\Phi'(t)}$, taken to be zero at $t=0.$
	 We omit the dependence on $\Phi$, which will be  clear from context.

	Let $\Omega \subset \X$ be open and bounded, $\Phi$ be a $p$-admissible function and $f\in L^{p'}(\Omega)$, where $p'=p/(p-1).$
	We define the functional $F_\Phi: W^{1,p}(\Omega) \to \rr$ as
	\begin{equation}\label{eq:F_Phi def}
		F_\Phi(u)\coloneqq \int_\Omega \Big( \Phi(|\nabla u|)+f u \Big) \d \mm.
	\end{equation}
	We omit the dependence {of the functional} $F_\Phi$ on $f$, which will be always clear from context.

	The next result establishes the equivalence between the minimization problem associated to the functional $F_\Phi$ of \eqref{eq:F_Phi def} and the equation \eqref{eq:intro eq}.  The proof is classical and  analogous to the one in the Euclidean setting (see \textit{e.g.}~\cite[Prop.\ 3.6]{VIOLO2025} or \cite[Theorem 2.5]{GR19}).

	\begin{prop}[Euler--Lagrange equations for $F_\Phi$]\label{prop:EL}
		Fix $p\in(1,\infty)$. Let $\Omega\subset \X$ be open and bounded  and  $\Phi$ be {$p$-}admissible in the sense of Definition \ref{def:admissible Phi}.
        
  { Then, $F_\Phi(u) \le  F_\Phi( u+\phi)$ for all $\phi \in W^{1,p}_0(\Omega)$ if and only if 
			\begin{equation}\label{eq:weak EL}
				-\int_\Omega \Psi(|\nabla u|)\la \nabla u, \nabla \phi \ra\d \mm=\int_\Omega f \phi \d \mm, \quad \forall \phi \in W^{1,p}_0(\Omega);
			\end{equation}  }
			 in which case  $\div(\Psi(|\nabla u|) \nabla u)=f.$
	\end{prop}
	The left-hand side of \eqref{eq:weak EL} is well-defined because $\Psi(|\nabla u|)|\nabla u|\in L^{p'}(\Omega)$ by \eqref{eq:phi' bound}.
    
    The following is a well-known  semicontinuity result which is obtained using  Mazur's Lemma. 
	
	\begin{lemma}[$W^{1,p}$ weak lower semicontinuity]\label{lem:lsc energy}
		Let  $\Phi:[0,\infty)\to [0,\infty)$ be a monotone non-decreasing convex function satisfying  $|\Phi(t)|\le C(1+|t|^p)$ for all $t\ge 0$.
		Suppose also that the sequence $\{f_n\}_n$ converges weakly in $W^{1,p}(\Omega)$ to some $f$. Then, 
		\begin{equation}\label{eq:lsc energy}
			\int_\Omega \Phi(|\nabla f|)\, \d\mm\le \liminf_n \int_\Omega  \Phi(|\nabla f_n|)\,\d \mm.
		\end{equation}
		If the convergence is strong in $W^{1,p}(\Omega)$ then equality holds 
		and we replace $\liminf_n$ with $\lim_n$.
	\end{lemma}
	
	The following existence result is a standard application of the direct method of the calculus of variations; we omit its proof for brevity.
	\begin{prop}[Existence and uniqueness of minimizers of $F_\Phi$]\label{prop:existence variational}
		Let $p\in(1,\infty).$ Suppose  that $\Omega\subset \X$ is bounded open and $\mm(\X\setminus\Omega)>0$. Let $g \in W^{1,p}(\Omega)$ and $\Phi$ be $p$-admissible $($\textit{cf.}~Definition \ref{def:admissible Phi}$)$ and satisfying, for some given constant $c>0$, 
		\begin{equation}\label{eq:coerciveness}
			\Phi(t)\ge c(|t|^p-1) \quad \forall t \in [0,\infty).
		\end{equation}
	Then, there exists a minimizer $u \in g+W^{1,p}_0(\Omega)$ which achieves {$F_\Phi(u) = \inf_{v\in g+W^{1,p}_0(\Omega)} F_\Phi(v).$} Moreover, if $\Phi$ is strictly convex then the minimizer is unique.
	\end{prop}

	We provide the connection between assumptions \eqref{eq:psi' condition} and \eqref{eq:psi growth} for $\Psi$ and $p$-admissible $\Phi$ . 
	
	\begin{lemma}\label{lem:psi and phi}
		Let $\Psi \in {\sf AC}_\loc(0,\infty)$ be positive satisfying both \eqref{eq:psi' condition}  and the $p$-growth condition \eqref{eq:psi growth} for $p \in(1,\infty).$ Define $\Phi(t)\coloneqq \int_0^t s\Psi(s)\d s.$ Then $\Phi$ is strictly convex, $p$-admissible in the sense of Definition \ref{def:admissible Phi}  and satisfies the coercivity condition \eqref{eq:coerciveness}.
	\end{lemma}
	\begin{proof}
		{By Lemma \ref{lem:psi basic}-i)}, $s\Psi(s)\to 0$ as $s\to 0^+$ and so it is bounded near zero, in particular $\Phi$ is well defined, differentiable in $[0,\infty)$ and $\Phi'(0)=0.$
		Moreover \eqref{eq:psi' condition} implies that $(t\Psi(t))'\ge (1+\lambda)\Psi(t)>0$, which shows that $t\Psi(t)$ is strictly increasing and so $\Phi$ is strictly convex.  
	\end{proof}

	\section{A priori Estimates}
	In this section we derive  \textit{a-priori}, \textit{uniform} second-order and gradient estimates for local solutions of $\div(\Psi(|\nabla u|)\nabla u)=f$. The term ``a-priori'' means that we assume that $u$ satisfies $\Delta u \in L^2(\mm)$. By ``uniform'' instead we mean that the estimates will depend  on   
	$\Psi$ only via the parameters appearing in the ellipticity and growth conditions \eqref{eq:psi' condition}, \eqref{eq:above and below} and \eqref{eq:psi growth}.
	
	In the sequel we shall frequently employ the following approximations of the gradient: 
	\begin{equation}\label{eq:cutoff approx of gradient sec 7}
		|\nabla u|_{M,\delta} := \sqrt{(|\nabla u| \wedge M)^2 + \delta}, \qquad |\nabla u|_\delta\coloneqq \sqrt{|\nabla u|^2+\delta},  
	\end{equation}
    {where $M>0,\delta>0$ are arbitrary constants.  }

	\subsection{$L^2$-Laplacian uniform estimate}

	The following proposition will be used  to prove Theorem \ref{thm:main quasilinear}, and in particular to obtain the desired local estimate on the Laplacian, which only involves the source term $f$. We emphasise that the result of Prop.\  \ref{prop:a priori quasilinear lapl estimates} does not rely on $u$ being the weak solution of a particular equation: it is true for a general $u \in W^{1,2}(\X)$ with $\Delta u \in L^2(\mm)$, with $\Psi$ as prescribed in the statement. 
	\begin{prop}[\textit{A priori} Laplacian estimates]\label{prop:a priori quasilinear lapl estimates}
		Let $\Xdm$ be an $\RCD(K,\infty)$ space and let $u\in\W(\X)$ with $\Delta u\in L^2 (\mm)$. Let $\Psi\in \LIP([0,\infty))$ be such that $\Psi'$ continuous up to a negligible set of points and  $\Psi'(t)=0$ for all $t\ge M$ for some constant $M>0$. {Suppose also that $\Psi$ satisfies the first in \eqref{eq:psi' condition}  with $\lambda\le 0 $} and $ c\leq \Psi\leq  c^{-1}$ for some $c\in(0,1)$. 
        
        Then for all $\phi \in \LIP_c(\X)$ it holds
		\begin{equation}\label{eq:quasilinear a-priori 6.3}
			\int\phi^2 (\Delta u)^2\d \mm \le C \int [\div(\Psi(|\nabla u|)\nabla u)]^2\phi^2+ |\nabla u|^2(\phi^2+|\nabla \phi|^2)  \d \mm,
		\end{equation}
		where $C>0$ is a constant depending only on  $c$, $\lambda$ and $K$.
	\end{prop}
	Recall from Lemma \ref{lem:gradgrad} that under the above assumptions $|\nabla u|\in \W(\X)$ and so the divergence operator appearing in \eqref{eq:quasilinear a-priori 6.3} is well defined (recall Prop.\  \ref{prop:leibniz rule div}).

	We note in passing that our choice of assumptions regarding the continuity of $\Psi'$ is motivated by the fact that we will  apply Prop.\  \ref{prop:a priori quasilinear lapl estimates} to the  {truncation $\Psi(t\wedge M)$}. We emphasise that the final estimate is independent of $M$.

	\begin{proof}

		\noindent 1. \textit{We first assume that $u \in \test(\X)$}:  Let $\phi \in \LIP_c(\X)$ be fixed as per the statement.  Since $|\nabla u|\in \W(\X)$, 
		by the Leibniz rule for the divergence in Prop.\  \ref{prop:leibniz rule div} we have
		\[
		\div(\Psi(|\nabla u|)\nabla u)=\Delta u \Psi(|\nabla u|)+\la \nabla \Psi(|\nabla u|) , \nabla u\ra\in L^2(\mm),
		\]
		where the $L^2(\mm)$-integrability follows  because $|\nabla u|\in L^\infty(\mm)$, $\Psi\in L^\infty([0,\infty))$ and $|\nabla \Psi(|\nabla u|)|\le \Lip(\Psi)|\nabla |\nabla u||\in L^2(\mm).$ Moreover $\Psi(|\nabla u|)|\nabla u|\in L^2(\mm)$. Hence by Remark \ref{rmk:sobolev test} we can use functions in $\W(\X)$ to test $\div(\Psi(|\nabla u|)\nabla u)$. {Set $\tilde c\coloneqq (1+\lambda)c$,  $\tilde \Psi:=\Psi-{\tilde c}\ge 0$ and note that by $\lambda \in (-1,0]$ it holds $\tilde c\in(0,c]$. Moreover from \eqref{eq:psi' condition} we deduce
			\begin{equation}\label{eq:psi trick 2}
				{t\Psi'(t)} \ge \lambda\Psi 
				= (1 +\lambda)\Psi-\Psi\ge c(1+\lambda)-\Psi=-\tilde \Psi , \quad \text{ for a.e.\  } t>0.
			\end{equation}
		} Choosing $\phi^2 \Delta u$ as test function we have
		\begin{equation*}
			\begin{split}
				\int& \div(\Psi(|\nabla u|)\nabla u) \phi^2 \Delta u\d \mm = -\int  \Psi(|\nabla u|) \la \nabla u, \phi^2 \nabla \Delta u + \Delta u \nabla \phi^2   \ra  \d\mm\\
				=& -\!\int \phi^2 \tilde \Psi(|\nabla u|) \la \nabla u, \nabla \Delta u\ra \d \mm- \!2\!\!\int \Psi(|\nabla u|) \phi \la \nabla u, \nabla \phi \ra \Delta u  \d \mm - \!{\tilde c}\!\! \int \phi^2  \la \nabla u, \nabla \Delta u\ra \d \mm ,
			\end{split}
		\end{equation*}
		where we note that $0 \leq \phi^2 \tilde \Psi(|\nabla u|) \in \W\cap L^\infty(\X)$. Hence, applying the Bochner inequality of Theorem \ref{thm:improved Bochner}  with $\phi= \phi^2 \tilde \Psi(|\nabla u|)$ to the first term on the right-hand side gives
		\begin{equation} \label{eq:orange 1}
			\begin{aligned}
				\int &\div(\Psi(|\nabla u|)\nabla u) \phi^2 \Delta u\d \mm \ge \int  \phi^2 \tilde \Psi(|\nabla u|) \Big( |\hess u|^2 + K |\nabla u|^2 \Big) \d \mm  \\
				& +  \!\int \!\bigg(\frac{\la \nabla (\phi^2\tilde \Psi(|\nabla u|)) ,\nabla |\nabla u|^2 \ra}2 -2\Psi(|\nabla u|) \phi \la \nabla u, \nabla \phi \ra \Delta u  \bigg)\d \mm - \!{\tilde c} \!\!\int \phi^2  \la \nabla u, \nabla \Delta u\ra \d \mm.
			\end{aligned}
		\end{equation}
		We now expand the term containing the gradient of $\phi^2\tilde \Psi(|\nabla u|)$:
		\begin{equation}\label{eq:orange 2}
			\begin{split}
				\frac{1}{2}\la \nabla (\phi^2\tilde \Psi(|\nabla u|)) ,\nabla |\nabla u|^2 \ra &= 2 \phi\tilde \Psi(|\nabla u|)|\nabla u|\la \nabla \phi, \nabla |\nabla u|\ra
				+  \phi^2\Psi'(|\nabla u|)|\nabla u||\nabla |\nabla u||^2\\
				\quad &\ge  2 \phi\tilde \Psi(|\nabla u|)|\nabla u|\la \nabla \phi, \nabla |\nabla u|\ra- \phi^2 \tilde \Psi(|\nabla u|)|\hess (u)|^2, 
			\end{split}
		\end{equation}
		where we used \eqref{eq:psi trick 2} and \eqref{eq:gradgrad} to obtain the final line.  Moreover, integration by parts yields 
		\begin{equation}\label{eq:orange 3}
			- {\tilde c}\int \phi^2  \la \nabla u, \nabla \Delta u\ra \d \mm= {\tilde c}\int \phi^2 (\Delta u)^2 \d\mm +  2{\tilde c}\int \phi \Delta u\la \nabla \phi , \nabla u\ra \d \mm.
		\end{equation}
		By combining \eqref{eq:orange 1},\eqref{eq:orange 2} and \eqref{eq:orange 3},  the  Hessian terms cancel out and we get 
		\begin{align*}
			{\tilde c} \int \phi^2 (\Delta u)^2 \d \mm\le & -K \int \phi^2 \tilde \Psi (|\nabla u|)|\nabla u|^2 \d \mm + \int \div(\Psi(|\nabla u|)\nabla u) \phi^2 \Delta u \d \mm \\ 
			&+ 2{c^{-1}} \int |\nabla \phi| |\phi| |\nabla u|\big(|\Delta u|+ |\nabla |\nabla u||\big) \d \mm,
		\end{align*}
		where we used that {$\tilde \Psi\le \Psi\le c^{-1}$ and $\tilde c \leq c\le c^{-1}.$}
		By bounding the term involving $|\nabla|\nabla u||$ by $|\hess(u)|$ as per \eqref{eq:gradgrad} and applying twice Young's inequality to the above, we deduce 
		\begin{align*}
			\frac{{\tilde c}}{2} \int \phi^2 (\Delta u)^2 \d \mm\le& (|K| + \delta^{-1}) {c^{-1}}\int  (\phi^2 + |\nabla \phi|^2) |\nabla u|^2 \d \mm \\ 
			&+ \frac{1}{2 {\tilde c}} \int [\div(\Psi(|\nabla u|)\nabla u)]^2 \phi^2 \d \mm +  \delta{c^{-1}}\int \phi^2 (|\hess (u)| +|\Delta u| )^2 \d \mm,
		\end{align*}
		for any $\delta >0.$ Using \eqref{eq:hessian laplacian estiamte} and choosing $\delta>0$ sufficiently small with respect to $c$, we absorb all the terms containing $\hess (u)$ and $\Delta u$ into the left-hand side. This completes the proof of \eqref{eq:quasilinear a-priori 6.3}. 
		
		\smallskip 
		
		\noindent 2. \textit{Assume now only $u\in\W(\X)$ with $\Delta u\in L^2 (\mm)$}:
		By Lemma \ref{lem:approximation lemma} there exists a sequence $u_n\in \test(\X)$ such that $u_n\to u$ in $\W(\X)$, $\Delta u_n\to \Delta u$ in $L^2(\mm)$ and $|\nabla u_n|\to |\nabla u|$ in $\W(\X)$. By the previous step we know that \eqref{eq:quasilinear a-priori 6.3} holds for $u_n.$ For convenience, we fix from now on Borel representatives of all these functions and of the corresponding ones for $u.$ Up to passing to a subsequence, we can also assume {by the Generalised Dominated Convergence Theorem} that the previous convergences hold $\mm$-a.e.\ and that the functions $|\Delta u_n|, |\nabla u_n|$ and $|\nabla |\nabla u_n||$ are all dominated $\mm$-a.e.~by a common $L^2(\mm)$-function $G\ge 0$. Hence, by passing to the limit in \eqref{eq:quasilinear a-priori 6.3}, we see that to obtain the result for $u$ it is enough to show:
		\begin{equation}\label{eq:convergence of div}
			\lim_{n\to\infty} \int [\div(\Psi(|\nabla u_n|)\nabla u_n)]^2 \phi^2 \d \mm = \int [\div(\Psi(|\nabla u|)\nabla u)]^2 \phi^2 \d \mm.
		\end{equation}
		We aim to apply the Dominated Convergence Theorem, for which we begin by showing the existence of a suitable dominating function.
		Note that, by the Leibniz rule of Prop.\  \ref{prop:leibniz rule div}, 
		\begin{equation}\label{eq:div expression}
			\div(\Psi(|\nabla u_n|)\nabla u_n)=\Delta u_n \Psi(|\nabla u_n|)+\la \nabla \Psi(|\nabla u_n|) , \nabla u_n\ra
		\end{equation}
		and that the same expression holds for $u$.
		From \eqref{eq:div expression} we immediately get
		$$	[\div(\Psi(|\nabla u_n|)\nabla u_n)]^2\phi^2 \le 2\|\Psi\|_\infty^2 G^2 \phi^2+ 2M^2\Lip(\Psi)^2G^2\phi^2 \in L^1(\mm),$$
		where we used that $\Psi'(t)=0$ for all $t\ge M.$ Thus we have a suitable dominating function.

		It remains to check pointwise $\mm$-a.e.\ convergence. The continuity of $\Psi$ implies that the sequence $\{\Psi(|\nabla u_n|)\Delta u_n\}_n$ converges pointwise $\mm$-a.e.~to  $\Psi(|\nabla u|)\Delta u $. It remains to check  
		\begin{equation}\label{eq:a.e. annoying}
			\la \nabla \Psi(|\nabla u_n|) , \nabla u_n\ra \to \la \nabla \Psi(|\nabla u|) , \nabla u\ra, \quad \mm\text{-a.e.,}
		\end{equation}
		and this follows as an immediate consequence of Lemma \ref{lem:psi' convergence} {and the triangle inequality}.
	\end{proof}

	\subsection{Second order Caccioppoli-type estimates}

	The following result gives a weighted estimate on $\hess(u)$ in terms of   $\div(\Psi(|\nabla u|)\nabla u)$. The free parameter $\beta \ge 0$ will be useful when applying a Moser-type iteration for the gradient in the next section. 

	\begin{prop}[{Weighted gradient Caccioppoli estimate}]\label{prop:key estimate}
		Let $\Xdm$ be an $\RCD(K,\infty)$ space and $u\in \W(\X)$ with $\Delta u \in L^2(\mm).$ Let $\Psi\in \LIP\cap L^\infty([0,\infty))$ be positive and such that $\Psi'$ is continuous up to a negligible set of points. Assume also that $\Psi$ satisfies \eqref{eq:psi' condition} for some $\Lambda,\lambda\in \rr$.
		Then for all   $\beta\ge 0$, {$M >0$}, $\delta>0$ and all functions $\eta \in \LIP_{bs}(\X)$ it holds
        
		\begin{equation}\label{eq:key with alpha}\begin{aligned}
				\int_\X  \Psi(|\nabla u|)^2 |\nabla u|_{M,\delta}^\beta  |{\rm Hess} (u)|^2 \eta^2 \d \mm \le C
				\int_\X  & (1+|\beta|^2)\div(\Psi(|\nabla u|)\nabla u)^2|\nabla u|_{M,\delta}^{\beta} \eta^2 \d \mm \\
				+C \int_\X & |\nabla u|^2\Psi(|\nabla u|)^2 |\nabla u|_{M,\delta}^\beta(|\nabla \eta|^2+\eta^2) \d \mm,\end{aligned}
		\end{equation}
		where $|\nabla u|_{M,\delta}$ was defined in \eqref{eq:cutoff approx of gradient sec 7} and $C>0$ is a constant depending only on $K,\Lambda$ and $\lambda.$ 
	\end{prop}
	\begin{proof}
    {The argument is similar to \cite[Lemma 4.9]{ivan1}.}
		We report  inequality \eqref{eq:starting point} in Lemma \ref{lem:bochner type 2}:
		\begin{equation}\label{eq:starting point bis}
			\begin{split}
				&\int_\X |{\rm Hess} (u)|^2 
				\psi(|\nabla u|) \eta^2 \d \mm \le 	\int_\X \la\nabla u\Delta u -|\nabla u|\,\nabla |\nabla u|, \nabla \eta^2\ra \psi(|\nabla u|)\d\mm\\
				&+\!\! \int_\X \! \la\nabla u\Delta u \!-\!|\nabla u|\,\nabla |\nabla u|, \nabla |\nabla u|\ra \psi'(|\nabla u|)\eta^2 \d\mm
				+\!\!\int_\X \! ((\Delta u)^2\!+\!K|\nabla u|^2)\psi(|\nabla u|) \eta^2\d\mm, 
			\end{split}
		\end{equation}
		which holds for all $\eta \in \LIP_{bs}(\X)$ and all $\psi \in \LIP\cap L^\infty(\rr)$ with $\psi'$ continuous up to a negligible set and such that $\psi'(t)\le c(1+|t|)^{-1}$ for all $t\ge 0$ and for some constant $c>0.$  We choose
		$\psi(t)\coloneqq \Psi(t)^2 ({(t\wedge M)^2+\delta})^\frac \beta2,$ which satisfies these constraints {due to \eqref{eq:psi' condition}}.
		Note that 
		\begin{equation}\label{eq:t psi deriv in proof with infinity laplacian}
			t\psi'(t)=\psi(t) \bigg(2p(t)+\mathds{1}_{\{t\le M\}}\frac{\, \beta t^2}{{t^2+\delta}}\bigg), \quad \text{a.e.\ $t\ge 0$,}
		\end{equation}
		where $p(t)\coloneqq \frac{t\Psi'(t)}{\Psi(t)}$ satisfies $\lambda \le p\le \Lambda$ {by the assumption \eqref{eq:psi' condition}}. The strategy is to get rid term by term of the Laplacian in \eqref{eq:starting point bis}, as it does not appear in the final estimate. To do so, set $f\coloneqq \div(\Psi(|\nabla u|)\nabla u)$ and note that, by the Leibniz rule of Prop.\  \ref{prop:leibniz rule div}, 
		\begin{equation}\label{eq:Delta=trick}
			\Delta u =  \frac{f}{\Psi(|\nabla u|)} - p(|\nabla u|) \frac{\Delta_\infty u}{|\nabla u|^2},
		\end{equation}
		where $\Delta_\infty u\coloneqq \la \nabla |\nabla u|,\nabla u\ra|\nabla u|$ and $\frac{\Delta_\infty u}{|\nabla u|}$ is assumed zero when $|\nabla u|=0.$ Plugging in \eqref{eq:Delta=trick}, the first term of the right-hand side of \eqref{eq:starting point bis} can be estimates as follows:
		\begin{align}\label{eq:splitting that causes C to depend on Psi Linfty norm}
			\la\nabla u\Delta u -&|\nabla u|\nabla |\nabla u|, \nabla \eta^2\ra \psi(|\nabla u|) \\
			&\le 
			\bigg(\frac{|f|}{\Psi(|\nabla u|)}+(|\Lambda|+1) |{\rm Hess} (u)|\bigg) |\nabla u| |\nabla \eta^2| \psi(|\nabla u|),
		\end{align}
		having used that $|\nabla |\nabla u||\le |{\rm Hess} (u)|$ {and direct computation to bound $\frac{|\Delta_\infty u|}{|\nabla u|^2} \leq |\hess(u)|$}. Similarly, for the first half of the second term in the right-hand side of \eqref{eq:starting point bis} we have \footnote{The functions $\psi'(|\nabla u|)$ and $p(|\nabla u|)$ for different representatives of $\Psi'$ and $|\nabla u|$  might change on positive measure sets. However on these sets $|\nabla |\nabla u||=0$ $\mm$-a.e.\ by locality and so, since $\psi'(|\nabla u|)$ appears only multiplying $|\nabla |\nabla u||$ or $|\Delta_\infty u|\le |\nabla |\nabla u|||\nabla u|^2 $, the quantities we write are  in fact  independent of the representatives.}
		\begin{align*}
			\la\nabla u\Delta u&, \nabla |\nabla u|\ra \psi'(|\nabla u|) \d\mm=\Delta u \frac{\Delta_\infty u}{|\nabla u|} {\psi'(|\nabla u|)}\\
			& =  \frac{f }{\Psi(|\nabla u|)} \frac{\Delta_\infty u}{|\nabla u|}  \psi'(|\nabla u|)- p(|\nabla u|) \frac{(\Delta_\infty u)^2}{|\nabla u|^3} \psi'(|\nabla u|)\\
			&\le  \bigg(\frac{\delta_1^{-1}f^2}{\Psi(|\nabla u|)^2} + \delta_1\frac{(\Delta_\infty u)^2}{|\nabla u|^4}\bigg)  {|\nabla u|\psi'(|\nabla u|)}- p(|\nabla u|) \frac{(\Delta_\infty u)^2}{|\nabla u|^3} \psi'(|\nabla u|)\\
			&\le  (2|\Lambda |+ \beta)\bigg(\frac{\delta_1^{-1}f^2}{\Psi(|\nabla u|)^2} + \delta_1\frac{(\Delta_\infty u)^2}{|\nabla u|^4}\bigg)  {\psi(|\nabla u|)}- p(|\nabla u|) \frac{(\Delta_\infty u)^2}{|\nabla u|^4} |\nabla u|\psi'(|\nabla u|),
		\end{align*}
		for all $\delta_1>0$, {where we used \eqref{eq:t psi deriv in proof with infinity laplacian} to write $t \psi'(t) \leq (2|\Lambda|+\beta)\psi(t)$ in the first term on the right-hand side}. {The second half of the aforementioned term  of \eqref{eq:starting point bis} equals}:
		\[
		\la-|\nabla u|\nabla |\nabla u|, \nabla |\nabla u|\ra \psi'(|\nabla u|)=-|\nabla u|\psi'(|\nabla u|)|\nabla |\nabla u||^2.
		\]
		Finally, for all $\delta_2>0$, {by squaring \eqref{eq:Delta=trick} and using Young's inequality}, 
		\begin{align*}
			(\Delta u)^2\psi(|\nabla u|)  &\le
			(1+\delta_2^{-1}) f^2 ({(}|\nabla u|\!\wedge \!M {)^2}+{\delta})^{\frac\beta2}  + (1+\delta_2)p(|\nabla u|)^2 \frac{(\Delta_\infty u)^2}{|\nabla u|^4}\psi(|\nabla u|) \\
			\le  (1&\!+\!\delta_2^{-1}\!) f^2 ({(}|\nabla u|\!\wedge \!M {)^2}\!\!+\!{\delta})^{\frac\beta2}  \!\!+\! p(|\nabla u|)^2 \frac{(\Delta_\infty u)^2}{|\nabla u|^4}\psi(|\nabla u|) \!+\!\delta_2 |\Lambda|^{{2}} |{\rm Hess}(u)|^{{2}}{\psi(|\nabla u|)}, 
		\end{align*}
		where we used the explicit definition of $\psi$ to write $\frac{\psi(t)}{\Psi(t)^2}\leq ((t\wedge M)^2 + 1)^{\frac{\beta}{2}}$. Plugging all the above estimates into \eqref{eq:starting point bis}, using the Young's inequality and choosing $\delta_1,\delta_2$ small enough to reabsorb the terms containing ${\rm Hess}(u)$ or $\Delta_\infty u$ into the left-hand side,  we reach 
		\begin{equation*}
			\begin{split}
				\int_\X |{\rm Hess} & (u)|^2 
				\psi(|\nabla u|) \eta^2 \d \mm \\ 
				\le &	C\int_\X  \bigg [(1+\beta^2) f^2\eta^2 + |\nabla u|^2 {\Psi(|\nabla u|)^2}(|\nabla \eta|^2+\eta^2)\bigg] ({(}|\nabla u|\!\wedge \!M {)^2}+ \delta)^\frac \beta2 \d \mm \\
				&+ \int_\X \underbrace{\bigg[ \frac{(\Delta_\infty u)^2}{|\nabla u|^4}(p^2-pg)(|\nabla u|) - g(|\nabla u|)|\nabla |\nabla u||^2  \bigg]}_{F\coloneqq }  \psi(|\nabla u|) \eta^2 \d \mm,
			\end{split}
		\end{equation*}
		where $g(t)\coloneqq {\frac{t \psi'(t)}{\psi(t)}} = 2p(t)+\mathds{1}_{\{t\le M\}}\frac{\, \beta t^2}{{t^2+\delta}}$ {as per \eqref{eq:t psi deriv in proof with infinity laplacian}}. Note that {\eqref{eq:psi' condition} implies } 
		\begin{equation}\label{eq:g>-2lamb}
			{\inf_{[0,\infty)}g  \ge 2 \inf_{[0,\infty)} p \geq 2\lambda {>-2}.   } 
		\end{equation}
		To conclude it remains to show that $F\le \theta |{\rm Hess}(u)|$ $\mm$-a.e.\ for some $\theta \in (0,1)$ depending only on $\lambda.$  Choosing $\theta$ close enough to one we can assume that $g+2\theta\ge 0.$  For convenience set $A\coloneq \frac{(\Delta_\infty u)^2}{|\nabla u|^4}$ and $B\coloneqq |\nabla |\nabla u||^2$ and ${C}\coloneqq |{\rm Hess}(u)|^2.$ By \eqref{eq:gradgrad} we have $0\le A\le B\le C$. Moreover $C\ge 2B-A$ (see \cite[Prop.\  4.4]{ivan1}). Hence it is enough to show, since $\beta\ge 0,$ that 
		\begin{align*}
			A(p^2-pg)-gB\le \theta (2B-A) &\iff      A(p^2-pg+\theta)-(g+2\theta)B\le 0.
		\end{align*}
		{Due to the lower bound \eqref{eq:g>-2lamb}, we have $\theta \in (-\lambda,1) \implies g+2\theta \geq 0$, whence}
		\begin{align*}
			&A(p^2-pg)-gB\le \theta (2B-A)   \iff      A(p^2-pg+\theta)-(g+2\theta)A\le 0\\
			&\iff      (p^2-pg-g-\theta \pm 1)\le 0  \iff (p+1)(p-1-g)+1-\theta\le 0,
		\end{align*}
		and the last inequality is true for {$\theta \in (-\lambda,1)$}, since by \eqref{eq:g>-2lamb}
		\[(p+1)(p-1-g)\le (p+1)(-p-1)= -(\lambda+1)^2<0.\qedhere\]
	\end{proof}
	
	{
		\begin{rem}\label{rmk:non autonom version 2}
			Prop.\  \ref{prop:key estimate} holds also replacing $\div(\Psi(|\nabla u|)\nabla u)$ with $\div(\aa(x)\Psi(|\nabla u|)\nabla u)$, for any $\aa\in \LIP(\X)$ with $A^{-1}\le \aa\le A $ for some constant $A\ge 1,$ up to take the constant $C$ in \eqref{eq:key with alpha}  depending on $\Lip(\aa)$ and $A$ as well. The argument is exactly the same only that in the right-hand side of \eqref{eq:Delta=trick} the extra term $-\frac{\la \nabla \aa , \nabla u \ra }{\aa}$ appears. Keeping track of this term in the estimates, noting that is bounded by $A\,\Lip(a)|\nabla u|$,  gives the result. 
			The same applies to Prop.\  \ref{cor:key estimate}, Prop.\  \ref{prop:key estimate hessian} and Prop.\  \ref{prop:lip a priori}, since their proof rely only on \eqref{eq:key with alpha}.
			\fr
		\end{rem}
	}

    {From the previous result in the case $\beta=0$ we have the following consequence.}

	\begin{prop}[Local second order estimate]\label{cor:key estimate}
		Assume that the conditions of Prop.\  \ref{prop:key estimate} hold and moreover that $\Xdm$ is an $\RCD(K,N)$ space with $N<\infty.$ Then, for all $B_R(x)\subset \X$ with $R\le 1,$ it holds 
		\begin{equation}\label{eq:key}
			\begin{split}
				\fint_{B_{R/2}(x)} &\Big( R^{-2}|V_u|^2 +  |\nabla  V_u|^2 \Big) \d \mm \le C
				\fint_{B_R(x)}   \div(V_u)^2 \d \mm 
				+ C \left(\fint_{B_R(x)}|V_u| \d \mm\right)^2,
			\end{split}
		\end{equation}
		where $V_u\coloneqq \Psi(\gradu) \nabla u$ and $C$ is a constant depending only on $K,N,\lambda$ and $\Lambda.$
	\end{prop}
	\begin{proof}
		{Fix $0 < R \leq 1$.}	{By scaling  we may assume $\mm(B_R(x))=1.$}
		 By the chain rule, 
		\begin{equation}\label{eq:grad computation psi}
			|\nabla V_u|\le |{\rm Hess} (u)|(\Psi'(|\nabla u|)|\nabla u|+\Psi(|\nabla u|))\le (\Lambda +1) \Psi(|\nabla u|)|{\rm Hess} (u)|,
		\end{equation}
		where we used \eqref{eq:psi' condition}. For all  $s,t \in (1/2,1)$ with $s<t$ choose $\eta \in \LIP_c(B_{tR}(x))$ with $\eta \equiv 1$ in $B_{sR}(x)$, $0\le \eta \le 1,$ and $\Lip(\eta)\le R^{-1}(t-s)^{-1} $. Applying \eqref{eq:key with alpha} with $\beta=0$ and this choice of $\eta$ to control the 
		Hessian term, we get 
		\begin{equation}\label{eq:pre estimate for plap}
			\begin{split}
				\int_{B_{sR}(x)}   |V_u|^2 \eta^2 \d \mm &\le C
				\int_{B_R(x)}   \div(V_u)^2 \d \mm+ \frac{C}{R^2(t-s)^2}\int_{B_{tR}(x)} |V_u|^{2}\d \mm,
			\end{split}
		\end{equation}
		where $C$ depends only on $K,\lambda,\Lambda$. The desired inequality now follows from \eqref{eq:pre estimate for plap} using the same interpolation-iteration argument as in \cite[Corollary 4.6]{ivan1}. As this iteration method is well-known (see \textit{e.g.}~\cite[Lemma 3.1, Chapter V]{GiaquintaBook} or \cite[pag.\ 594]{CianchiMazya}) but technical, we omit the details. 
	\end{proof}
	
	{
	\begin{prop}[Local non-degenerate Hessian estimate]\label{prop:key estimate hessian}
		Under the assumptions and notations of Prop.\  \ref{cor:key estimate}, assuming also that that $\Psi\ge c>0$,  it holds 
		\begin{equation}\label{eq:key hessian}
				c^2\fint_{B_{R/4}(x)} |\hess (u)|^{2} \d \mm \le C \fint_{B_R(x)} f^2 \d \mm + C \left(\fint_{B_R(x)} \Psi(|\nabla u|)|\nabla u| \d \mm \right)^2
		\end{equation}
		where $f\coloneqq \div(\Psi(|\nabla u|)\nabla u)$ and  $C$ is a constant depending only on $N,K,\Lambda,\lambda, p,\nu $.
	\end{prop}
	\begin{proof}
   Immediate by combining \eqref{eq:key with alpha} with $\beta=0$ and \eqref{eq:key}, choosing  $\eta\in \LIP_c(B_R(x))$ so that $\eta \equiv 1$ on $B_{R/4}(x)$  and $\Lip(\eta)\le 8R^{-1}.$
	\end{proof}
	}
	\subsection{Gradient estimates}

	In what follows, we employ a Moser iteration technique on the gradient to obtain local Lipschitz estimates on the solution of \eqref{eq:intro eq} with $\Psi$ satisfying \eqref{eq:psi' condition} and \eqref{eq:lower bound in lip}. We note that iteration strategies have been used to obtain gradient estimates for non-linear elliptic equations in a number of classical works (\textit{cf.}~\textit{e.g.}~\cite[\S 3 of Ch.4]{UralLadybook}, \cite[Chapter 8]{giustibook}, \cite{lieberman1991natural}) and also in recent developments (see \cite{DuzarMing,DMnonuniformly,mingioneSurvey,minDesurvey} and references therein).

	\begin{prop}[$L^\infty$-gradient estimate]\label{prop:lip a priori}
		Fix $p\in(1,\infty).$ Let $\Psi \in \LIP([0,\infty)$ be positive, bounded, satisfy \eqref{eq:psi' condition} with constants $\lambda,\Lambda$ and
		\begin{equation}\label{eq:lower bound in lip}
			\Psi(t)\ge \nu t^{p-2}, \quad \text{for all $t\ge 1$}
		\end{equation}
		for some $\nu>0$.
		Let $\Xdm$ be an $\rcd(K,N)$ space with $N<\infty$ and let $q>{\max\{N, 2\}}.$
		Suppose that $u\in \dom(\Delta)$ satisfies $\Delta u \in L^2(\mm)$ and  it holds
		\begin{equation}\label{eq:lip a priori ass}
			\big |\div(\Psi(|\nabla u|)\nabla u)|= f,\quad \mea\text{-a.e.\ in $B_{R}(x)$},
		\end{equation}
		for some $B_{R}(x)\subset \X$, with $R\le 1$ and $\left(\fint_{B_{R}(x)} |f|^{q}\right)^{1/q} \d \mm \le C_0$. 
		Then  
		\begin{equation}\label{eq:apriori lip}
			\| \nabla u \|_{L^\infty(B_{R/4}(x))}\le C \bigg(1+\Psi(1)+\fint_{B_{R}(x)} \Psi(|\nabla u|) |\nabla u| \, \d \mm  \bigg)^\frac{1}{p-1},
		\end{equation}
		where $C$ is a constant depending only on $N,K,C_0,\lambda,q,p,\nu$ and $\Lambda.$
	\end{prop}
	
	The proof of this result is very similar to the one of \cite[Prop.\  4.10]{ivan1}, which focused on the particular case of $\Psi(t)\coloneqq (t^2+\delta)^\frac{p-2}{2}$ for  $p\in (1,3+\frac{2}{N-2})$ and $\delta>0$. Note that {this particular choice of} $\Psi$ satisfies \eqref{eq:psi' condition} with $\lambda=\Lambda=p-2$ but is neither bounded nor Lipschitz for $p>2$.  Except for this, the assumptions in \cite[Prop.\  4.10]{ivan1}  are the same as here.

	\begin{proof}[Proof of Proposition \ref{prop:lip a priori}]
    {Up to increasing $N$ we can assume that $N>2.$}
		Fix $B\coloneqq B_R(x)\subset \X$, $R\le 1$ and $u$ as in the hypotheses. By the scaling of \eqref{eq:apriori lip} we can assume  $\mea(B_{R}(x))=1.$   We will denote by $C\ge 1$ a constant whose value may change from line to line but  depending only on $N,K,\lambda ,\Lambda, C_0,q,\nu,p$. The rest of the proof is divided into several steps. 
		
		\smallskip 
		
		\noindent 1. \textit{Caccioppoli-type inequality}: For any  
		$\gamma\ge 0$,  define the function 
		\begin{equation}\label{eq:def of v in moser proof}
			v\coloneqq \Psi(|\nabla u|_1)|\nabla u|_1 (|\nabla u|_{M,1})^{\gamma}\in \W_\loc(\X), 
		\end{equation}
		{where we recall the notations $|\cdot|_1$ and $|\cdot|_{M,1}$ introduced in \eqref{eq:cutoff approx of gradient sec 7}.} {Direct computation} using the chain rule (recall the bound \eqref{eq:gradgrad} on $\nabla|\nabla u|$ from Lemma \ref{lem:gradgrad}) yields 
		\begin{align*}
			|\nabla v|\le(|\Lambda|+|\lambda|+1+\gamma) |\hess (u)|\Psi(|\nabla u|_1) |\nabla u|_{M,1}^{\gamma}, \quad \mea\text{-a.e.} 
		\end{align*}
		Let $\eta \in \LIP_c(B)$ be such that $0\le \eta\le 1$. It is immediate to verify that $\Psi^1(t)\coloneqq \Psi(\sqrt{t^2+1})$ still satisfies the hypotheses (see Lemma \ref{lem:regularity of truncated psi}). Hence we can apply Prop.\  \ref{prop:key estimate} to  $\Psi^1$ (with the constants $\beta$ and $\delta$ therein set to $2\gamma$ and $1$, respectively)  to control the Hessian term: 
		\begin{equation}\label{eq:start moser}
			\begin{split}
				\int_\X |\nabla v|^2\eta^2 \d \mm &\le C  (1+\gamma^4)\int_\X \Big(  f^2 \eta^2|\nabla u|_{M,1}^{2\gamma}+ (|\nabla \eta|^2+\eta^2) v^{2}   \Big)   \d \mm\\
				&\le C  (1+\gamma^4)\int_\X \bigg( f^2  \eta^2|\nabla u|_{1}^{2\gamma \frac{p-1}{\gamma+p-1}}|\nabla u|_{M,{1}}^{2\gamma \frac{\gamma}{\gamma+p-1}}+ (|\nabla \eta|^2+\eta^2) v^{2} \bigg)   \d \mm\\
				&\le   C (1+\gamma^4)\int_\X \Big( f^2 \eta^2 (v^2)^{\frac{\gamma}{\gamma+p-1}}+ (|\nabla \eta|^2+\eta^2) v^{2} \Big) \d \mm, 
			\end{split}
		\end{equation}
		where, in the first inequality, we used \eqref{eq:lip a priori ass} {and $|\nabla u| \leq |\nabla u|_1$}, and in the last one we used \eqref{eq:lower bound in lip}. {Note} \eqref{eq:start moser} is identical to \cite[eq.\ (4.28)]{ivan1}, where $\frac{\gamma}{\gamma+p-1}\in [0,1)$ is denoted there $\lambda_\beta$.

		\smallskip

		\noindent 2. \textit{Reverse Sobolev inequality}: 
		Arguing exactly as per \cite[Prop.\  4.10]{ivan1}, from \eqref{eq:start moser} we get 
		\begin{equation}\label{eq:final moser}
			\| \eta v\|_{L^{2^*}(\mm)}^2  
			\le C  (1+\gamma^4)^{\alpha}\max \Big\{ \left( R^2\Lip(\eta)^2+1\right)\|v\|_{L^2(\supp(\eta))}^2\,,1 \Big\},
		\end{equation}
		 where $2^*\coloneqq \frac{2N}{N-2}$  and $\alpha\coloneqq 1+\frac{N}{q-N}$. {For the benefit of the reader, we briefly outline the main idea. To show \eqref{eq:final moser} we apply H\"older's inequality to the term in \eqref{eq:start moser} containing $f$:
			\begin{equation}\label{eq:holder f moser}
				\int_\X  f^2 \eta^2 (v^2)^{\frac{\gamma}{\gamma+p-1}}\d \mm 
				\le  \|f\|_{L^q(B)}^{{2}} \|\eta v\|_{L^{\frac{2\gamma}{\gamma +p-1}\cdot \frac{2q}{q-2}}(B)}^\frac{2\gamma}{\gamma +p-1}
				\le \|f\|_{L^q(B)}^{{2}} \|\eta v\|_{L^{\frac{2q}{q-2}}(B)}^\frac{2\gamma}{\gamma +p-1},
			\end{equation} 
			where in the first inequality we used $|\eta| \le 1$, {$q>2$}, and $\frac{\gamma}{\gamma +p-1}\in (0,1] $ and in the last one Jensen's inequality.
			Plugging  \eqref{eq:holder f moser} into \eqref{eq:start moser},  by interpolation, we absorb the $L^{\frac{2q}{q-2}}$-norm of $\eta v$ into the  left-hand side of \eqref{eq:start moser} after an  application the Sobolev inequality. Here we use that $q>N$ {implies $\frac{2q}{q-2} < 2^*$}. See the argument in \cite[eq.~(4.30) to eq.~(4.33)]{ivan1} for further details.
		}

		\smallskip

		\noindent 3. \textit{Recursion estimate}: Having the bound \eqref{eq:final moser} we are now in a position to perform the Moser iteration. For each $k \in \nn\cup\{0\}$ we choose a function $\eta_k$ as before and satisfying 
		\[
		\eta_k \equiv 1 \text{ in $B_{  \frac R4+\frac{1}{2^{k+1}}\frac R4 }(x) =: B_k$}, \quad  \supp\eta_k\subset \bar B_{k-1} , \quad \Lip(\eta_k)\le  R^{-1}{2^{k+3}}.
		\]
		Set $\gamma_0\coloneqq p-1$  and  $\gamma_k\coloneqq  (\frac{N}{N-2})\gamma_k >\gamma_{k-1}$ for all $k\in \nn$.  
		Next, we define 
		\begin{equation}\label{eq:ak def moser proof}
			{\bf a}_k\coloneqq \left(\int_{B_k}  \Psi(|\nabla u|_1)^2|\nabla u|_1^2|\nabla u|_{M,1}^{2(\gamma_{k}-p +1)} \, \d \mm   \right)^\frac {1}{\gamma_k}, \quad k\in \nn\cup \{0\}, 
		\end{equation}
		and choose $\eta = \eta_k$ and $\gamma=\gamma_{k-1}-(p-1)$ in the inequality \eqref{eq:final moser}. 
		Iterating  \eqref{eq:final moser} exactly as in \cite[Prop.\  4.10]{ivan1} we reach
		\begin{equation}
			{\bf a}_k\le C ({\bf a}_0+1)
		\end{equation}
		where $C$  still depends only on $N,K,\lambda,\Lambda,C_0,q,\nu,p$ (here we used again \eqref{eq:lower bound in lip}). 
		\smallskip 
		
		\noindent 4. \textit{Deducing the $L^\infty$-bound}: 
		From the definition of $\mathbf{a}_k$, using \eqref{eq:lower bound in lip} and the bounds $|\nabla u| \wedge M\le |\nabla u|_{M,1} \le |\nabla u|_1$   we obtain for all $k$
		\begin{equation}\label{eq:almost the end}
			\begin{split}
				\nu^\frac{2}{\gamma_k}\bigg(\int_{B_{R/4}(x)}  &(|\nabla u|\wedge M)^{2\gamma_{k}} \, \d \mm   \bigg)^\frac 1{\gamma_k} \le {\bf a}_{k}\le C( {{\bf a}_{0}}+1)\\
				&
				\le C \|\Psi(|\nabla u|)|\nabla u|\|_{L^2(B_{{R/2}}(x))}^\frac{2}{p-1} +C(1+\Psi(1)^\frac{2}{p-1}),
			\end{split}
		\end{equation}
		where we used that $\sqrt{t^2+1}\Psi(\sqrt{t^2+1})\le 2^{\Lambda +1}(t\Psi(t)+\Psi(1))$ (see Lemma \ref{lem:psi basic}). 
		To estimate the term on the right-hand side we use Prop.\  \ref{cor:key estimate}, which gives
		\begin{equation}\label{eq:from 2(p-1) to p-1}
			\|\Psi(|\nabla u|)|\nabla u|\|_{L^2(B_{{R/2}}(x))}^2 
			\le C'
			\|f\|_{L^2(B_R(x)}^2 \d \mm 
			+ C'\|\Psi(|\nabla u|)|\nabla u|\|_{L^1(B_{{R}}(x))}^2 ,
		\end{equation}
		for some $C'$ depending only on $K,N,\lambda,\Lambda$.  Letting $k\to \infty$ and then $M\to \infty$ we conclude.
	\end{proof}

	\section{Non-degenerate Quasilinear Equations}\label{sec:the quasilinear case}
	
	{The ultimate goal of this section is to obtain $L^2$-Laplacian estimates for equation \eqref{eq:intro eq} where the function $\Psi$ satisfies the ellipticity condition \eqref{eq:psi' condition} and the non-degeneracy bound
		\begin{equation}\label{eq:upper lower c}
		c\le \Psi \le c^{-1}, \quad c\in(0,1).
		\end{equation}
 In particular we prove Theorem \ref{thm:main quasilinear}. With $c$ we will  denote the constant in \eqref{eq:upper lower c}. }

{To help the reader navigate this section we summarize below its content:
\begin{itemize}
    \item In \S \ref{subsec:general exist approx quasilinear} we prove an existence result, Prop.\  \ref{prop:existence minty}, for a truncated/cut-off problem, where we   replace $\Psi(|\nabla u|)$ by $	\Psi_{M,\eta}(|\nabla u|)$ (see \eqref{eq:PsiM def in 6.2}). 
    \item In \S \ref{subsec:galerkin} we employ a Galerkin  scheme for the truncated/cut-off problem, Prop.\  \ref{prop:galerkin scheme}. 
    \item In \S \ref{subsec:proof of theorem 2} we send $\eta\to 1$ and $M\to \infty$ and we obtain regularity for the original operator $\div(\Psi(\gradu)\nabla u)$, thus proving Theorem \ref{thm:main quasilinear}.
\end{itemize}
}

	\subsection{Existence for the truncated/cut-off problem}\label{subsec:general exist approx quasilinear}

	In what follows, we establish the existence of a weak solution $u$ with boundary data $g\in \W(\Omega)$ to the equation
	\begin{equation}\label{eq:approx problem general discussion start of S6}
		\div(\Psi_{M,\eta}(|\nabla u|) \nabla u ) = f, \quad \text{in $\Omega$,}
	\end{equation}
	in duality with a chosen closed subspace $V$ of $W^{1,2}_0(\Omega)$, where the shorthand $\Psi_{M,\eta}$ denotes 
	\begin{equation}\label{eq:PsiM def in 6.2}
		\Psi_{M,\eta}(|\nabla u|) := \Psi((|\nabla u| \wedge M)\eta^2), 
	\end{equation}
	and is used throughout \S \ref{sec:the quasilinear case} to simplify notation. { To solve \eqref{eq:approx problem general discussion start of S6},  $\eta$ can be any $L^0$-function,  later we will need to choose $\eta\in \LIP_c(\Omega)$.} The existence result   is given in Prop.\  \ref{prop:existence minty} below and will be used in subsequent section to obtain the Galerkin approximations $u_k$ in duality with the spaces $V=V_k$. We will then deduce the convergence $u_k \rightharpoonup u$ for $u$ a solution of equation \eqref{eq:approx problem general discussion start of S6} and, by obtaining $k$-independent {Laplacian-}estimates (see Prop.\  \ref{prop:smooth laplacian estimate}), also that $\Delta u \in L^2_\loc$. These estimates however depend on $M$ and the cut-off function $\eta$ in \eqref{eq:PsiM def in 6.2}. As a result, in order to let $M\to\infty$ and $\eta\to 1$, we will  rely instead on the estimate previously obtained in  Prop.\  \ref{prop:a priori quasilinear lapl estimates}.

	We explain briefly why we proceed in this manner. {The advantage of considering the Galerkin approximations $u_k$ is that they are { $W^{2,2}_\loc$} by definition and so we can justify taking { second} derivatives.} Thus, we  deduce the second-order estimate for  $u_k$, and then pass it to the limit for the original solution $u$ of \eqref{eq:intro eq}. {Unfortunately we cannot obtain these estimates working directly with $\Psi(|\nabla u|)$  as the coefficient in \eqref{eq:approx problem general discussion start of S6} and we need to introduced the $M$-truncation and the cut-off $\eta\in\LIP_c(\Omega)$}. {Indeed, both of these modifications are needed when applying the Bochner inequality of Prop.\  \ref{prop:bochner type}; the former is needed to make sense of products such as $|\nabla u| \langle \nabla \Psi_{M,\eta}, \nabla|\nabla u|\rangle$ in \eqref{eq:bochner type}, while the }latter is used primarily so that we may justify integrating by parts in various manipulations.

	\begin{prop}\label{prop:existence minty}
		Let $\Xdm$ be an $\RCD(K,\infty)$ space and $\Omega\subset \X$ be open bounded, such that $\mm(\X\setminus \Omega)>0.$ Let $\Psi\in C([0,\infty))$ satisfy  {\eqref{eq:upper lower c}} for some $c\in(0,1)$ and be such that $t \mapsto t \Psi(t)$ is monotone non-decreasing.
		Let also $\eta \in L^0(\Omega)$,  $g\in \W(\Omega)$, $M>0$. Then 
		\begin{enumerate}[label=\roman*)]
			\item   for all $v_1,v_2 \in \W_0(\Omega),$ it holds
			\begin{equation}\label{eq:monotonicity}
				\int_\Omega \la \Psi_{M,\eta}(|\nabla v_2+g|) \nabla (v_2+g)-\Psi_{M,\eta}(|\nabla v_1+g|) \nabla (v_1+g),\nabla (v_2-v_1)\ra\ge 0,
			\end{equation}
			\item for all $f\in L^2(\Omega)$ and all closed subspace $V\subset \W_0(\Omega)$   there exists $u\in g+V$ so that
			\begin{equation}\label{eq:weak equation nonlinear}
				\int_\Omega \Psi_{M,\eta}(|\nabla u|)\la \nabla u,\nabla w\ra =\int_\Omega f w \d \mm, \quad \forall w \in V
			\end{equation}
			and (any such) $u$ satisfies
			\begin{equation}\label{eq:weak equation nonlinear W12estimate}
				\|u\|_{\W(\Omega)}\le C (\|g\|_{\W(\Omega)}+\|f\|_{L^2(\Omega)}),
			\end{equation}
			where $C>0$ is a constant depending only on $c$ and $\Omega$. Moreover if $t\mapsto t\Psi(t)$ is strictly increasing then such $u$ is unique.
		\end{enumerate}
	\end{prop}
	\begin{proof}
		Our strategy is to apply the Minty--Browder Theorem (see \textit{e.g.}~\cite[Theorem 5.16]{BrezisBook}) to the operator $A:V\to V^*$, defined by 
		\[
		Av(w)\coloneqq \int_\Omega \Psi_{M,\eta}(|\nabla (v+g)|)\la \nabla (v+g),\nabla w\ra\d \mm,
		\]
		where $V$ is endowed with the $\|\cdot\|_{\W(\Omega)}$ norm.
		Since $|\Psi|$ is bounded, it is clear that  $Av\in V^*$, hence the map is well posed.  We need to prove that $A$ is continuous, monotone and coercive; we stress that, while $Av$ is a linear map for each fixed $v \in V$, the map $A$ {itself} is not  linear. 
		
		\smallskip 
		
		\noindent 1. \textit{$A$ is continuous}: We argue sequentially. Take any sequence $\{v_n\}_n \subset V$ such that $v_n\to v\in V$  strongly in $\W(\Omega)$. Observe that, given any $w \in V$, there holds 
		\begin{equation*}
			\begin{aligned}
				|Av_n(w) - Av(w)| = \bigg| \int_\Omega \langle \Psi_{M,\eta}(|\nabla (v_n+g)|)\nabla (v_n+g) - \Psi_{M,\eta}(|\nabla (v+g)|) \nabla (v+g),\nabla w \ra \d \mm \bigg|, 
			\end{aligned}
		\end{equation*}
		whence, by H\"older's inequality and taking the supremum over all $w \in V$ with unit norm, 
		\begin{align*}
			\|Av_n-Av\|_{V^*} &\le \Big\|  \left|\Psi_{M,\eta}(|\nabla (v_n+g)|) \nabla (v_n+g) - \Psi_{M,\eta}(|\nabla (v+g)|) \nabla (v+g)\right| \Big\|_{L^2(\Omega)} \to 0,
		\end{align*}
		which follows from the Dominated Convergence Theorem {and the continuity of $\Psi$}. 
		
		\smallskip 
		
		\noindent 2. \textit{$A$ is coercive}: By direct computation, using {\eqref{eq:upper lower c}} and the Poincar\'e inequality {of Prop.\  \ref{prop:poincare}}, we obtain, for all $v \in V$,
		\begin{equation}\label{eq:Avv>}
			\begin{split}
				\frac{Av(v)}{\|v\|_{\W(\Omega)}} &\ge  \frac{{c}\Vert\nabla v\Vert_{L^2(\Omega)}^2- {c^{-1}} \|\nabla v\|_{L^2(\Omega)}\|\nabla g\|_{L^2(\Omega)}}{\|v\|_{W^{1,2}(\Omega)}} \\ 
				&\geq {c}C_P^{-1} \Vert v \Vert_{W^{1,2}(\Omega)} -  {c^{-1}} \Vert g \Vert_{W^{1,2}(\Omega)}, 
			\end{split}
		\end{equation}
		{with $C_P$ the Poincar\'e constant}. Hence $ \frac{Av(v)}{\|v\|_{\W(\Omega)}} \to\infty$ as $\|v\|_{\W(\Omega)}\to \infty$,   as required. 
		
		\smallskip 
		
		\noindent 3. \textit{$A$ is monotone}: It remains to show monotonicity, that is  
		\begin{equation}\label{eq:monotonicity in proof}
			(Av_1-Av_2)(v_1-v_2)\ge 0, \quad \forall v_1,v_2\in V;
		\end{equation}
		to be interpreted as $Av_1(v_1-v_2) - Av_2(v_1-v_2) \geq 0$ as real numbers. Note that this would also prove item i) {of the proposition}. To show \eqref{eq:monotonicity in proof}, we define 
		\begin{equation}\label{eq:zeta function}
			\zeta_\lambda(t)\coloneqq t\Psi((t\wedge M)\lambda).
		\end{equation}
		It holds that $t\mapsto \zeta_\lambda(t)$ is non-decreasing for all $\lambda \ge 0$,$M>0$ and strictly increasing if $t\Psi(t)$ is.

		For ease of notation, we now denote $\Psi_j := \Psi((|\nabla(v_j+g)|\wedge M)\eta^2)$ and $w_j := v_j + g$, $j=1,2$, in the lines of computation that follow. Observe that direct computation yields 
		\begin{equation*}
			\begin{aligned}
				\langle \Psi_1 \nabla w_1 -& \Psi_2 \nabla w_2 , \nabla(v_1-v_2) \rangle = \Psi_1 |\nabla w_1|^2  + \Psi_2 |\nabla w_2|^2   -(\Psi_1 + \Psi_2) \langle \nabla w_1 , \nabla w_2  \rangle. 
			\end{aligned}
		\end{equation*}
		It therefore follows that there holds, in the $\mm$-a.e.~sense, 
		\begin{equation}\label{eq:after multiply by eta2}
			\begin{aligned}
				\langle \Psi_1  \nabla w_1 - \Psi_2 \nabla w_2 , \nabla(v_1-v_2) \rangle   &\ge  \Psi_1 |\nabla w_1|^2  + \Psi_2 |\nabla w_2|^2   -(\Psi_1 + \Psi_2) |\nabla w_1||\nabla w_2| \\
				&= (\Psi_1  |\nabla w_1| - \Psi_2 |\nabla w_2|) (|\nabla w_1| - |\nabla w_2|) \\ 
				&=(\zeta_{\eta^2}(|\nabla w_1|) - \zeta_{\eta^2}(|\nabla w_2|) (|\nabla w_1| - |\nabla w_2|) \geq 0,
			\end{aligned}
		\end{equation}
		where we used that  $\zeta_\lambda$, \textit{cf.}~\eqref{eq:zeta function}, is monotone for all $\lambda\ge 0$. Integrating \eqref{eq:after multiply by eta2} gives \eqref{eq:monotonicity in proof}.

		\smallskip 
		
		\noindent 4. \textit{$A$ is strictly monotone when $t\mapsto t \Psi(t)$ is strictly increasing}: Suppose that $ (Av_1-Av_2)(v_1-v_2)=0$ for some $v_1,v_2\in V.$ 
		Then the two inequalities in \eqref{eq:after multiply by eta2} must be  equalities $\mm$-a.e. The first one,\textit{ i.e.}\ Cauchy–Schwarz, since $\Psi_1>0,\Psi_2>0$ implies that
		\[
		\la \nabla (v_1+g) ,\nabla (v_2+g)\ra=|\nabla v_1+g||\nabla v_2+g|, \quad \mm\text{-a.e.,}
		\]
		while the second  combined with the strict monotonicity of $t\mapsto \zeta_\lambda(t)$, for all $\lambda \in \rr$, gives
		\[
		|\nabla (v_1+g)|=|\nabla (v_2+g)|\quad \mm\text{-a.e.}
		\]
		Hence $|\nabla (v_1-v_2)|=0$ $\mm$-a.e.\ in $\Omega$ and by the Poincaré inequality \eqref{eq:poincare} we deduce $v_1=v_2.$
		
		\smallskip 
		
		\noindent 5. {The assumptions of the Minty--Browder Theorem are therefore satisfied, and it follows that for all $f \in L^2(\Omega) \subset V^*$ there exists $v \in V$ solving $Av = f$. In turn, by letting $u=v+g$, we obtain the sought solution of \eqref{eq:weak equation nonlinear}.} Finally estimate \eqref{eq:weak equation nonlinear W12estimate} follows immediately from \eqref{eq:Avv>} applied to $v=u-g$ and standard algebraic manipulations. 
	\end{proof}

	\subsection{Galerkin approximation}\label{subsec:galerkin}

	In this portion of the manuscript, we perform a Galerkin approximation of the quasilinear problem \eqref{eq:approx problem general discussion start of S6} as discussed at the start of \S \ref{subsec:general exist approx quasilinear}.   We shall first derive a Laplacian-estimate \eqref{eq:second order est for galerkin} in Prop.\  \ref{prop:smooth laplacian estimate} below, which is independent of the $k$-order subspace $V_k$---recall \eqref{eq:Vk def}---on which we perform the Galerkin approximation. This will allow us to deduce second order regularity of the limit solution.

	\begin{prop}[Laplacian-type estimate]\label{prop:smooth laplacian estimate}
		Let $\Xdm$ be a {locally compact} $\RCD(K,\infty)$ space, $\Omega\subset \X$ be open and bounded and $u\in \W(\Omega)$ with $\Delta u\in \W(\Omega)$. Let $M>0$. Let $\Psi\in \LIP([0,\infty))$  be bounded, such that
		$\Psi(t)> \Psi(0)>0$ for all $t>0$  and 
		\begin{equation}\label{eq:psi trick in 6.2}
			\frac{t\Psi'(t)}{\Psi(t)-\Psi(0)}\ge -1, \quad \text{ for a.e.\ } t >0.
		\end{equation}
		Suppose that for some $\eta \in \LIP_{c}(\Omega)$ and some $g \in \W(\Omega)$ with $\Delta g\in \W(\Omega)$ it holds $\Delta (u-g)\in \W_0(\Omega)$ and, with $\Psi_{M,\eta}$ as in \eqref{eq:PsiM def in 6.2},
		\begin{equation}\label{eq:psi nabla u term}
			-\int_\Omega \Psi_{M,\eta}(|\nabla u|) \la \nabla u, \nabla \Delta (u-g)\ra \d \mm=\int f \Delta (u-g) \d \mm.
		\end{equation} 
		Then 
		\begin{equation}\label{eq:second order est for galerkin}
			\int_\Omega  |\Delta u|^2\d \mm \le C\int_\Omega \Big( |\nabla u|^2+f^2+|\nabla \Delta g|^2+|\Delta g|^2 + g^2 \Big)  \d \mm,
		\end{equation}
		where $C>0$ is a constant depending only on $\Psi$, $M$,  $K$ and $\eta.$ 
	\end{prop}

	\begin{rem}\label{rem:condition vs strict monotonicity}
		Condition \eqref{eq:psi trick in 6.2} implies that $t \mapsto t \Psi(t)$ is non-decreasing. More strongly,
		\eqref{eq:psi trick in 6.2}  implies that  the first in \eqref{eq:psi' condition} holds with $\lambda=-1+\frac{\Psi(0)}{\|\Psi\|_\infty}.$ 
		\fr 
	\end{rem}

	Before giving the proof of this result, we explain why it is needed for our strategy, and why we cannot simply use the uniform estimate in Prop.\  \ref{prop:a priori quasilinear lapl estimates} instead of \eqref{eq:second order est for galerkin}. At first glance, it seems that \eqref{eq:second order est for galerkin} is not needed, since the estimate of Prop.\  \ref{prop:a priori quasilinear lapl estimates} \emph{appears} stronger. {Indeed, if one formally substitutes for $f$ in place of $\div(\Psi_{M,\eta}(|\nabla u|) \nabla u)$ in Prop.\  \ref{prop:a priori quasilinear lapl estimates} (say that $\eta=1$ in $\supp(\phi)$)} one obtains a sharper estimate than \eqref{eq:second order est for galerkin}.  However, the aforementioned formal substitution cannot be rigorously justified, as the equation $\div(\Psi_{M,\eta}(|\nabla u|)\nabla u) = f$ for the Galerkin approximation {does not hold pointwise, but} only  in duality with the subspace $V_k$. { As a matter of fact Prop.\  \ref{prop:smooth laplacian estimate} gives  a Laplacian estimate precisely  assuming only that $\div(\Psi_{M,\eta}(|\nabla u|)\nabla u) = f$ holds  in duality with $u-g$, see \eqref{eq:psi nabla u term}. The drawback is that the constant depends on $\eta$ and $M.$} Only after passing to the limit as $k\to\infty$ and obtaining a {sufficiently regular ``true''} solution, in duality with  $W^{1,2}_0(\Omega)$, we will be able to invoke the estimate of Prop.\  \ref{prop:a priori quasilinear lapl estimates}.

	\begin{proof}[Proof of Proposition \ref{prop:smooth laplacian estimate}]
		Set $\tilde \Psi=\Psi-\Psi(0)\ge 0$.
		Since $\tilde \Psi(0)=0$ ---this is precisely why we employ $\tilde\Psi$ instead of $\Psi$ directly--- the assumptions of the  Bochner-type inequality \eqref{eq:bochner type} in Prop.\  \ref{prop:bochner type} with $u,\eta$ and $\tilde \Psi$ are satisfied. Hence
		\begin{equation}\label{eq:bochner psi tilde}
			-\int_\Omega |\nabla u|\la \nabla \tilde \Psi_{M,\eta} ,\nabla |\nabla u| \ra \d \mm \ge \int_\Omega \Big( |{\rm Hess} (u)|^2  +\la \nabla u, \nabla \Delta u\ra + K|\nabla u|^2 \Big) \tilde \Psi_{M,\eta}  \d \mm. 
		\end{equation}
		{Setting $\Psi'_{M,\eta}\coloneqq \Psi'((|\nabla u | \wedge M)\eta^2)$}  we  now compute
		\begin{equation}\label{eq:deriv formula truncation}
			\nabla \tilde \Psi_{M,\eta} = \nabla\Psi_{M,\eta} = { \Psi'_{M,\eta}} \Big( \eta^2 \mathds{1}_{\{|\nabla u| \leq M\}}\nabla |\nabla u| + (|\nabla u|\wedge M) 2 \eta \nabla \eta \Big).
		\end{equation}
		Substituting \eqref{eq:psi nabla u term} into the right side of \eqref{eq:bochner psi tilde} and developing the left side using \eqref{eq:deriv formula truncation}, we get 
		\begin{equation*}
			\begin{aligned}
				&-\int_\Omega |\nabla u| {  \Psi'_{M,\eta}} \la \eta^2 \mathds{1}_{\{|\nabla u| \leq M \}} \nabla|\nabla u| + (|\nabla u|\wedge M)2 \eta \nabla \eta ,\nabla |\nabla u| \ra \d \mm \\ 
				\ge &  \int_\Omega \Big(    |{\rm Hess} (u)|^2 +   \la\nabla u,\nabla \Delta g\ra  +    K|\nabla u|^2 \Big) \tilde \Psi_{M,\eta} -f \Delta (u-g) - \Psi(0)  \la \nabla u, \nabla \Delta (u-g)\ra \d \mm. 
			\end{aligned}
		\end{equation*}
		We note here that the term on the left-hand side of  \eqref{eq:second order est for galerkin} will appear  integrating by parts in the final term on the right-hand side; we will do this later in the argument.	 Rearranging, we get 
		\begin{align*}
			\int_\Omega \Big( &{ \Psi'_{M,\eta}} |\nabla u|\eta^2 \mathds{1}_{\{|\nabla u| \leq M\}} |\nabla |\nabla u||^2 +  \tilde \Psi_{N,\eta} |{\rm Hess} (u)|^2 \Big) \d \mm \\ 
			\le &-\int_\Omega \Big(  \langle \nabla u , \nabla \Delta g \rangle    + K|\nabla u|^2 \Big)\tilde \Psi_{N,\eta} \d \mm + \int_\Omega f \Delta(u-g) \d \mm  \\
			&+ \!\Psi(0) \!\int_\Omega \!\!\langle \nabla u , \nabla \Delta(u\!-\!g) \rangle \d \mm \!-\! 2\!\!\int_\Omega \!\! {|\nabla u|}   (|\nabla u | \wedge M) {  \Psi'_{M,\eta}}\eta \langle \nabla \eta , \nabla |\nabla u| \rangle  \d \mm, 
		\end{align*}
		and moreover, using \eqref{eq:psi trick in 6.2} and \eqref{eq:gradgrad}
		we get 
		\[
		 { \Psi'_{M,\eta}} |\nabla u|\eta^2 \mathds{1}_{\{|\nabla u| \leq M\}} |\nabla |\nabla u||^2 +  \tilde \Psi_{M,\eta} |{\rm Hess} (u)|^2 \ge 0\qquad \mm\text{-a.e.}, 
		\]
		and we deduce 
		\begin{align*}
			-\Psi(0)\!\! \int_\Omega \langle \nabla u , \nabla \Delta(u-g) \rangle \d \mm \le &-\!\int_\Omega \!\Big(  \langle \nabla u , \nabla \Delta g \rangle   \! + \! K|\nabla u|^2 \Big)\tilde \Psi_{N,\eta} \d \mm \!+\! \int_\Omega f \Delta(u-g) \d \mm  \\
			& - 2\int_\Omega {|\nabla u|}(|\nabla u |\wedge M)  {\Psi'_{M,\eta}}  \eta\langle \nabla \eta , \nabla |\nabla u| \rangle  \d \mm.
		\end{align*}
		The final term on the right-hand side is bounded by $2M \Lip(\Psi) \int_\Omega \eta |\nabla\eta| |\hess(u)| |\nabla u| \d \mm$, again using \eqref{eq:gradgrad}.  Integrating by parts on the left-hand side, thanks to the assumption $\Delta(u-g) \in W^{1,2}_0(\Omega)$, and using Young's inequality yields, for all choices of $\delta >0$, 
		\begin{align*}
			\frac{\Psi(0)}{4}\!\int_\Omega \!\!  |\Delta u|^2 \d \mm  \leq & \frac{\Psi(0)}{2} \Vert \Delta g \Vert_{L^2(\Omega)}^2 + \Vert {\Psi}\Vert_{L^\infty}\int_\Omega\!\! \Big( |\nabla u| |\nabla \Delta g| + |K| |\nabla u|^2 \Big) \d \mm \!-\! \!\int_\Omega\!\! f \Delta g \d \mm \\ 
			+&\frac{2}{\Psi(0)} \Vert f \Vert^2_{L^2(\Omega)} \!+\! \frac{\delta}{2}  \Vert \eta |\hess(u)| \Vert^2_{L^2(\Omega)} \!+\! \frac{1}{2\delta}M^2\Lip(\Psi)^2\Vert \nabla \eta \Vert_{L^\infty}^2 \Vert  \nabla u \Vert^2_{L^2(\Omega)}. 
		\end{align*}		
Finally, using  {the Calder\'on--Zygmund inequality} {from Corollary \ref{cor:hessian laplacian estimate}} and choosing $\delta$ sufficiently small we can absorb the term containing $\hess (u)$ into the left-hand side and conclude.
	\end{proof}

	{We are ready to perform the Galerkin scheme. We will first use Prop.\  \ref{prop:existence minty} to obtain existence of the approximating solutions in duality with the subspace $V_k$ 	 (of eigenfunctions of the Laplacian) and then exploit the $k$-independent estimate \eqref{eq:second order est for galerkin} to obtain regularity of the limit. For the latter regularity we will need strong assumption on the boundary data, \textit{i.e.}\ $\Delta g\in \W(\Omega)$, which will be removed at later stage  by  approximating $g.$ }

	\begin{prop}[Galerkin scheme]\label{prop:galerkin scheme}
		Let $\Xdm$ be a {locally compact} $\RCD(K,\infty)$ space and $\Omega\subset \X$ be bounded with $\mm(\X\setminus \Omega)>0.$  Let $\Psi\in C([0,\infty))$  satisfy {\eqref{eq:upper lower c}} and be such that $t \mapsto t \Psi(t)$ is monotone non-decreasing. Fix also $\eta \in \LIP_c(\Omega)$, $g\in \W(\Omega)$ and $f\in L^2(\Omega)$. Then: 
		\begin{enumerate}[label=\roman*)]
			\item \underline{Solution to the reduced problem:} For all $k\in \nn$ there  exists $u_k\in g+V_k$ satisfying
			\begin{equation}\label{eq:galerkin equation}
				-\int_\Omega \Psi_{M,\eta}(|\nabla u_k|)\la \nabla u_k, \nabla w\ra\d\mm=\int_\Omega f w\d \mm, \quad \forall \, w\in V_k.
			\end{equation}
			\item \underline{Convergence to a solution:} Up to passing to a subsequence,  $u_k\rightharpoonup u$ weakly  in $W^{1,2}(\Omega)$ to some  $u\in g+W^{1,2}_0(\Omega)$ such that 
			\begin{equation}\label{eq:galerkin limit}
				\div(\Psi_{M,\eta}(|\nabla u|)\nabla u)=f.
			\end{equation}
			\item \underline{Regularity:} Assume furthermore that $\Psi\in \LIP\cap L^\infty([0,\infty))$,
			$\Psi(t)> \Psi(0)>0$ for all $t>0$  and
			\begin{equation}\label{eq:psi trick}
				\frac{t\Psi'(t)}{\Psi(t)-\Psi(0)}\ge -1, \quad \text{ for a.e.\ } t >0.
			\end{equation}
			Suppose also that $\Delta g\in \W(\Omega)$. Then $\Delta u \in L^2(\Omega).$
		\end{enumerate}
	\end{prop}
	\begin{proof} We deal with each part separately.
    
		\noindent  \textit{Proof of i)}: This is a direct consequence of Prop.\  \ref{prop:existence minty}, which also gives
		\begin{equation}\label{eq:uk bound indep of eta}
			\| \nabla u_k \|_{L^2(\Omega)}^2\le C_{\Omega,\Psi}(\|f\|_{L^2(\Omega)}^2+ \|  g \|_{\W(\Omega)}^2) \quad \text{for all $k\in \nn.$}
		\end{equation}

		\noindent  \textit{Proof of ii)}: From \eqref{eq:uk bound indep of eta} we have that  $u_k$ converges weakly in $W^{1,2}(\Omega)$ to some $u \in g+\W_0(\Omega)$ up to passing to a subsequence. {By the boundedness of $\Psi$}, up to passing to a further subsequence, the product $\Psi_{M,\eta}(|\nabla u_k|) \nabla u_k$  converges weakly in $L^2(T\X)\restr{\Omega}$ to some $\chi\in L^2(T\X)\restr\Omega$. We claim that
		\begin{equation}\label{eq:claim chi}
			\div \chi= \div (\Psi_{M,\eta}(|\nabla u|)\nabla u).
		\end{equation}
		First, we note that this would be sufficient to conclude the proof of the proposition. Indeed, by passing to the limit in \eqref{eq:galerkin equation}, assuming \eqref{eq:claim chi}, we obtain that
		\[
		-\int_\Omega \Psi_{M,\eta}(|\nabla u|)\la \nabla u, \nabla w\ra\d\mm=\int_\Omega f w\d \mm, \quad \forall \, w\in V_{k'},
		\]
		for all $k' \in \nn$ and thus by density also for all  $w \in W^{1,2}_0(\Omega)$, which implies \eqref{eq:galerkin limit}. 
		
		It remains to prove the claim \eqref{eq:claim chi}.  {We argue by Minty-Browder monotonicity method.} Fix $w \in V_{k'}$. Set $v_k\coloneqq u_k-g\in V_k$ and $v\coloneqq u-g\in \W_0(\Omega)$. In particular $v_k \rightharpoonup v$ weakly in $W^{1,2}(\Omega)$.  By monotonicity---recall \eqref{eq:monotonicity}---we have 
		\begin{equation}\label{eq:monotonicity trick}
			\begin{aligned}  \int_\Omega \! \! \Psi_{M,\eta}(|\nabla (v_k\!+\!g)|)\la \nabla ( v_k\!+\!g ), \nabla (v_k\!-\!w)\ra\d\mm \! \ge  \! \int_\Omega \!\! \Psi_{M,\eta}(|\nabla (w\!+\!g)|)\la \nabla (w\!+\!g), \nabla (v_k\!-\!w)\ra\d\mm.
			\end{aligned}
		\end{equation}
		Using $v_k$ as test function in \eqref{eq:galerkin equation}, we get 
		\begin{align*}
			0&=\int_\Omega f v_k \d \mm + \int_\Omega \Psi_{M,\eta}(|\nabla u_k|)\la \nabla u_k,\nabla (v_k\pm w)\ra\d \mm\\
			&=\int_\Omega f v_k \d \mm + \int_\Omega\Big( \Psi_{M,\eta}(|\nabla (v_k+g)|)\la \nabla (v_k+g),\nabla (v_k-w)\ra + \Psi_{M,\eta}(|\nabla u_k|)\la \nabla u_k,\nabla w\ra\Big) \d \mm\\
			&\ge  \int_\Omega f v_k \d \mm + \int_\Omega \Big( \Psi_{M,\eta}(|\nabla (w+g)|)\la \nabla (w+g),\nabla (v_k-w)\ra + \Psi_{M,\eta}(|\nabla u_k|)\la \nabla u_k,\nabla w\ra \Big) \d \mm, 
		\end{align*}
		where we used \eqref{eq:monotonicity trick} to obtain the final inequality. Since $\Psi_{M,\eta} \in L^\infty$, by letting $k\to \infty$, 
		\begin{align}\label{eq:after use monotonicity trick}
			0\ge  \int_\Omega f v \d \mm + \int_\Omega \Psi_{M,\eta}(|\nabla (w+g)|)\la \nabla (w+g),\nabla (v-w)\ra + \int_\Omega \la \chi,\nabla w\ra\d \mm.
		\end{align}
		By density, the above actually holds for all $w \in \W_0(\Omega).$ 
		Meanwhile, we also have the equality 
		\begin{equation}\label{eq:key X trick}
			\int fv\d \mm=-\int\la\chi,\nabla v\ra \d \mm,
		\end{equation}
        where $v=u-g\in \W_0(\Omega)$ was defined before \eqref{eq:monotonicity trick}.
		Indeed, by \eqref{eq:galerkin equation} and since $\Psi_{M,\eta}(|\nabla u_k|) \nabla u_k$  converges weakly in $L^2(T\X)\restr{\Omega}$ to $\chi$, we deduce that  \eqref{eq:key X trick} holds with $v=v' \in V_{k'}$ for all $k'\in \nn$, from which \eqref{eq:key X trick} follows by density. Substituting \eqref{eq:key X trick} into  \eqref{eq:after use monotonicity trick}, we get 
		\begin{equation}\label{eq:final monotonicity}
			0\ge \int_\Omega \la \Psi_{M,\eta}(|\nabla (w+g)|)\nabla (w+g)-\chi, \nabla (v-w)\ra\d \mm, \quad \forall\, w \in \W_0(\Omega).
		\end{equation}
		 Taking $w=v+t h =u-g+th$, with $h \in \W_0(\Omega)$  arbitrarily chosen, we obtain 
		\[
		0\geq - t\int_\Omega \la \Psi_{M,\eta}(|\nabla (u+th)|)\nabla (u+th)-\chi, \nabla h\ra\d \mm, \quad \forall\, h \in \W_0(\Omega).
		\]
		Dividing by $t$, letting  $t\to 0^+$ and  $t\to 0^-$, by the Dominated Convergence Theorem,  we get
		\[
		0=\int_\Omega \la \Psi_{M,\eta}(|\nabla u|)\nabla u-\chi, \nabla h\ra\d \mm, \quad \forall\, h \in \W_0(\Omega), 
		\]
		which shows \eqref{eq:claim chi} and concludes the proof of ii). 
		
		\smallskip 
		
		\noindent  \textit{Proof of iii)}: The final part of the result follows immediately by applying Prop.\  \ref{prop:smooth laplacian estimate} with $u=u_k.$ Indeed, $u_k-g\in V_k$, whence it follows that $\Delta (u_k-g)\in V_k\subset \W_0(\Omega)$. We then obtain the requirement \eqref{eq:psi nabla u term} in Prop.\  \ref{prop:smooth laplacian estimate} by testing \eqref{eq:galerkin equation} with $w=\Delta (u_k-g)$. Finally, we apply Prop.\  \ref{prop:smooth laplacian estimate} combined with estimate \eqref{eq:uk bound indep of eta} to obtain 
		\begin{equation*}
			\int_\Omega  |\Delta u_k|^2\d \mm \le C\int_\Omega \Big( |\nabla g|^2+g^2+f^2+|\nabla \Delta g|^2+|\Delta g|^2  \Big) \d \mm,
		\end{equation*}
		{where the constant $C$ is independent of $k$}, from which $\Delta u \in L^2(\Omega)$ follows. 
	\end{proof}

	\subsection{Stability with respect to $M$, $\eta$, $g$ and proof of Theorem \ref{thm:main quasilinear}}\label{subsec:proof of theorem 2}

	{
	In what follows, we  let $\eta \to 1$ and $M\to\infty$ in the coefficient $\Psi_{M,\eta}$, \textit{cf.}~\eqref{eq:PsiM def in 6.2}, and finally prove Theorem \ref{thm:main quasilinear}. To do so will rely on the Laplacian estimate given in Prop.\  \ref{prop:a priori quasilinear lapl estimates}. We will also need to approximate the boundary data $g$ with a sequence $\{g_k\}_k$ such that $\Delta g_k \in W^{1,2}(\Omega)$ (\textit{cf.}\  Prop.\  \ref{prop:galerkin scheme}-iii)). In the process we will need the following technical convergence result. The proof  is very similar to Prop.\  \ref{prop:galerkin scheme}-ii) via the same monotonicity argument, and we shall omit it for brevity.	}
	{	\begin{prop}[Stability with respect to $\eta,g,\Psi$]\label{prop:quasilinear stabilities} 
Let $\Xdm$ be an $\RCD(K,\infty)$ space and $\Omega\subset \X$ be bounded with $\mm(\X\setminus \Omega)>0.$ Let $\Psi\in C([0,\infty))$ satisfy \eqref{eq:upper lower c} with $t \mapsto t \Psi(t)$ monotone non-decreasing. Fix $\eta \in L^0(\Omega)$,  $g\in \W(\Omega)$ and $f\in L^2(\Omega)$. 

Let $u_k\in \W(\Omega)$ satisfy $u_k\in g_k+\W_0(\Omega)$ and 
\begin{equation*}
\div(\psi_k(|\nabla u_k|)\nabla u_k)=f
\end{equation*}
for all $k\in \nn$, where one of the following holds:
\begin{enumerate}[label=\roman*)]
\item $\psi_k=\Psi_{M,\eta_k}$ with $\eta_k \in L^0(\Omega)$ converging to $\eta$ pointwise $\mm$-a.e., and $g_k\equiv g$;
\item $\psi_k=\Psi_{M,\eta}$ and $g_k \to g$ in $\W(\Omega)$;
\item $\psi_k=\Psi^k(t)$ with $\Psi^k \in C([0,\infty)])$  satisfying the same assumptions as $\Psi$  taking the same constant $c\in(0,1)$ in \eqref{eq:upper lower c}, $\Psi^k(t)\to \Psi(t)$ pointwise for all $t>0$, and $g_k\equiv g$.
\end{enumerate}
Then, up to subsequence, $u_k$ converges weakly in $W^{1,2}(\Omega)$ to $u\in g+\W_0(\Omega)$ satisfying
\begin{equation*}
\div(\Psi_{M,\eta}(|\nabla u|)\nabla u)=f\quad \text{(or $\div(\Psi(|\nabla u|)\nabla u)=f$ in case iii)}.
\end{equation*}
\end{prop}}
	We are finally ready to prove Theorem \ref{thm:main quasilinear}.

	\begin{proof}[Proof of Theorem \ref{thm:main quasilinear}]
		The proof is divided into seven steps.
		
		\smallskip

		\noindent 1. \textit{Setup for the proof}: Fix any  $\Omega'\subset \subset \Omega$. {By the local compactness of $\X$}, we can find $\phi,\eta \in \LIP_{c}(\Omega)$ such that $\phi \equiv 1$ in $\Omega'$ and  $\eta \equiv 1$ in $\supp \phi$.

		We will first prove the theorem under the following  additional assumptions:
		\begin{itemize}
			\item[A)]  $\Delta g \in \W(\Omega);$
			\item[B)]$\Psi'(t)=0$ for all $t\ge M$ for some constant $M>0$ (in particular $\Psi(t)=\Psi(t\wedge M))$;
			\item[C)]  {$\Psi\in \LIP([0,\infty))$ and}  $\Psi(t)> \Psi(0)>0$ for all $t>0$  and
			\begin{equation}\label{eq:psi ass in proof}
				\frac{t\Psi'(t)}{\Psi(t)-\Psi(0)}\ge -1, \quad \text{ for a.e.\ } t >0, 
			\end{equation}
			\item[D)]  it holds $g=\bar g\restr \Omega$ for some $\bar g\in \W(\X)$. 
		\end{itemize}
		and then relax assumptions A) to {D) in steps 4--7 of the proof}.

		\smallskip

		\noindent 2. \textit{Galerkin approximations and $M$-independent estimate}: {From the first in \eqref{eq:psi' condition} and since $\Psi$ is positive we have that $t\mapsto t\Psi(t)$ is strictly increasing}. Hence, by applying  Prop.\  \ref{prop:existence minty} with $\Psi,\eta,g,f$,  there exists a unique $u \in g+ W^{1,2}_0(\Omega)$ weak solution in duality with $W^{1,2}_0(\Omega)$ of 
		\begin{equation}\label{eq:galerkin limit in proof}
			\div(\Psi_{M,\eta}(|\nabla u|)\nabla u)=f. 
		\end{equation}
		Meanwhile, by  applying Prop.\  \ref{prop:galerkin scheme} with $\Psi,\eta,g,f$,  we find functions $u_k\in g+V_k$ solving \eqref{eq:galerkin equation} and such that $u_k \rightharpoonup \tilde u\in g+\W_0(\Omega)$ weakly in $\W(\Omega)$, for some $\tilde u$ solving \eqref{eq:galerkin limit in proof}, which by uniqueness must coincide with $u$. Next, by applying part iii) of Prop.\  \ref{prop:galerkin scheme}, thanks to the current assumption that $\Delta g \in \W(\Omega),$ we deduce that $\Delta u \in L^2(\Omega)$.
		In turn all the assumptions of Prop.\  \ref{prop:a priori quasilinear lapl estimates} are satisfied thanks to assumption B).         Hence, the estimate \eqref{eq:quasilinear a-priori 6.3} holds with $\varphi$ chosen as a test function (selected in Step 1 of the proof), and we get the $M$-independent bound 
		\begin{equation}\label{eq:quasilinear a-priori}
			\int_\Omega \phi^2 (\Delta u)^2\d \mm \le C \int_\Omega \Big( [\div(\Psi(|\nabla u|)\nabla u)]^2\phi^2+ |\nabla u|^2(\phi^2+|\nabla \phi|^2) \Big) \d \mm,
		\end{equation}
		where $C>0$ is a constant depending only on  $c$, $\lambda$ and $K$; in particular, it is independent of $M,\eta$.  We stress that  we did  not use equation \eqref{eq:galerkin limit in proof} so far.  Exploiting now \eqref{eq:galerkin limit in proof} and locality (recall that $\eta=1$ in $\supp(\phi)$) we deduce that $\div(\Psi(|\nabla u|)\nabla u)=f$ $\mm$-a.e.\ in $\supp(\phi)$; we emphasise that is also due to assumption B) which ensures that $\Psi_{M,\eta}(|\nabla u|) = \Psi(|\nabla u|)$ on inside the set $\supp (\varphi)$. Recalling also estimate \eqref{eq:weak equation nonlinear W12estimate} for $\|u\|_{\W(\Omega)}$  we conclude that $u$ satisfies {the $L^2$-estimate on $\Delta u$ } in \eqref{eq:quasilinear estimate} as desired. 	However, we emphasise that (at this point) $u$ is a solution of the approximate problem \eqref{eq:galerkin limit in proof} and not  the original equation \eqref{eq:quasilinear equation thm6}. We must now let $\eta \to 1$ and remove the assumptions A), B), C), D) all the while preserving the second order estimate \eqref{eq:quasilinear estimate}.

		\smallskip

		\noindent 3. \textit{Letting $\eta\to 1$}: We let  $\eta_n\in \LIP_{c}(\Omega)$ be such that $\eta_n \equiv 1$ in $\supp \phi$ and $\eta_n\to 1$ pointwise in $\Omega.$ Let $u_n\in g+\W_0(\Omega)$ be the unique function solving \eqref{eq:galerkin limit in proof} with $\eta=\eta_n$. By Step 2, the estimate \eqref{eq:quasilinear estimate} holds with $u=u_n$. Moreover, by part i) of Prop.\  \ref{prop:quasilinear stabilities}, up to passing to a subsequence, the sequence $\{u_n\}_n$ converges to $u\in g+ \W_0(\Omega)$ weakly in $W^{1,2}(\Omega)$, which is a weak solution of \eqref{eq:quasilinear equation thm6}{; recall that we have (for the moment) assumed $\Psi(t)=\Psi(t \wedge M)$ by the additional assumption B)}.  However, this solution is unique by Prop.\  \ref{prop:existence minty}. Finally \eqref{eq:quasilinear estimate} holds for $u$ by lower semicontinuity.  This shows the results under assumptions A), B), C), D). {In the remainder of the proof, we remove these additional assumptions.}

		\smallskip 
		
		\noindent 4. \textit{Removing assumption $A)$}: Consider now any $g \in \W(\Omega)$. By D) there exists $g\in \W(\X)$ such that $g=\bar g\restr \Omega$. Then, we can find $g_n\in \W(\X)$ such that $g_n\to g$ and $\Delta g_n \in \W(\X)$; \textit{e.g.}~we can take $g_n\coloneqq h_{t_n}\bar g$, where $t_n\to 0^+$ and $ h_{t_n}\bar g$ is the evolution of $\bar g$ via the heat flow (see \textit{e.g.}~\cite[Chapter 5]{GP20}). In particular $\tilde g_n\coloneqq g_n\restr\Omega\to g$ in $\W(\Omega)$ and $\Delta \tilde g_n \in \W(\Omega)$. By applying Step 3 with $\tilde g_n$ we obtain $u_n \in\tilde g_n + \W_0(\Omega)$ weak solution of 
		\begin{equation*}
			\div(\Psi(|\nabla u_n|)\nabla u_n) = f, 
		\end{equation*}
		in duality with $W^{1,2}_0(\Omega)$, and satisfying \eqref{eq:quasilinear estimate} with $u=u_n$. We can then apply Prop.\  \ref{prop:quasilinear stabilities} in case iii) to obtain that the sequence $\{u_n\}_n$, up to passing to a further subsequence, converge to $u$ the solution of \eqref{eq:quasilinear equation thm6} in $g+\W_0(\Omega)$. Again we  obtain \eqref{eq:quasilinear estimate} also for $u$ by the one from $u_n$ using the weak lower semiconituity of the norm{; we remark that the right-hand side of \eqref{eq:quasilinear estimate} only involves $f,g$ and does not involve the norms of $u_n$, so that we do not require the strong convergence of the sequence $\{u_n\}_n$ to deduce the final estimate for $\Delta u$}.
		
		\smallskip 
		\noindent 5. \textit{Removing assumption $B)$, \textit{i.e.}\  letting $M\to\infty$}: From here onwards, let $\Psi$ be as in the hypotheses of the theorem and satisfy assumption C), but not B). Define ${\Psi_M} (t)\coloneqq \Psi(t\wedge M)$. Clearly, $\Psi_M$ still satisfies the assumptions of the theorem and {satisfies \eqref{eq:psi' condition} and \eqref{eq:above and below}  with the same parameters $c$ and $\lambda$ as $\Psi$ (see Lemma \ref{lem:regularity of truncated psi})}, and  C); indeed, $\Psi_M'(t)=\Psi'(t)$ for $t < M$ and $\Psi_M'(t)=0$ for $t \geq M$, {with continuous derivative up to a negligible set of points}. Moreover, $\Psi_M(t)\to \Psi(t)$ as $M\to\infty$.  Hence, labelling $u_M \in g+W^{1,2}_0(\Omega)$ the solution of $$\div(\Psi_M(|\nabla u_M|)\nabla u_M) = f,$$ 
		an application of part iii) of Prop.\  \ref{prop:quasilinear stabilities} implies that the sequence $\{u_M\}_M$ converges weakly in $W^{1,2}(\Omega)$, up to passing to a further subsequence, to $u \in g+\W_0(\Omega)$ which solves \eqref{eq:quasilinear equation thm6}. Finally, estimate \eqref{eq:quasilinear estimate} holds, {with a constant sill depending only on $\lambda,c$ and $K$,} for each $u_M$ by Step 4 of the proof and because $\Psi_M$ satisfies B). As before, we deduce that $u$ satisfies \eqref{eq:quasilinear estimate} as well.  {As per Step 4 of the proof, only $f,g$ appear on the right-hand side of \eqref{eq:quasilinear estimate}, whence the strong convergence of $\{u_M\}_M$ in $W^{1,2}$ is not required to obtain the final estimate \eqref{eq:quasilinear estimate}.}
		\smallskip 
		
		\noindent 6. \textit{Removing assumption $C)$}:  From here onwards, we only assume that D) holds. Define, for all $n\in \nn$, 
		\[
		\Psi_n(t)\coloneqq \min\bigg\{\frac{{c(1+\lambda)}}{2}+nt\,,\,\Psi(t)\bigg\}, \quad \forall\,\, t\ge 0,
		\]
		{where $\lambda$ is the constant in assumption \eqref{eq:psi' condition} for $\Psi$, which we can assume $\lambda \le 0.$}
		The assumption $\Psi\ge c$ implies that $\Psi_n(t)>\Psi_n(0)=c(1+\lambda)/2$ for all $t>0.$ {{Around zero  $\Psi_n$ is a linear function} while}  from \eqref{eq:psi' condition} we have $|\Psi'(t)|\le t^{-1}(|\lambda|+|\Lambda|)$, hence  $\Psi_n\in \LIP([0,\infty))$.
		Note also that $\Psi'_n$ is continuous up to a negligible set of points, as it coincides a.e.\ either with $\Psi'$ or $n.$ {From this we also check that  $\Psi_n$ satisfies \eqref{eq:psi' condition} with the same $\lambda$ as $\Psi$.} Finally a basic computation as in \eqref{eq:psi trick 2} shows that $\Psi_n$ satisfies \eqref{eq:psi ass in proof} { for a.e.\ $t>0$ such that  $\Psi'_n(t)=\Psi'(t)$, while \eqref{eq:psi ass in proof} is trivially true when $\Psi'_n=n\ge0$.} Thus $\Psi_n$ satisfies the hypotheses of the theorem and  $C)$. Finally, we  have $\Psi_n(t)\to \Psi(t)$ for all $t> 0$. We can conclude as in the previous step using iii) of Prop.\  \ref{prop:quasilinear stabilities}.

		\smallskip 
		
		\noindent 7. \textit{Removing assumption $D)$}: The existence and uniqueness of $u$ solving \eqref{eq:quasilinear equation thm6} follows from Prop.\  \ref{prop:existence minty}, for which we do not need assumption D). For $\Omega'$ as in the assumptions, by local compactness there exists $\Omega''\subset\subset \X$ open such that $\Omega'\subset \subset \Omega''$. We now apply the result of Step 6 to $\Omega''$ with $\Psi$, $f\restr{\Omega''}$ and $g=u\restr{\Omega''}$. Indeed $u\restr{\Omega''}$ can be extended to some $\bar g \in \W(\X)$, as we can simply take $\bar g=\phi (u\restr{\Omega''})$ for some $\phi \in \LIP_{c}(\Omega)$ with $\phi \equiv 1$ in $\Omega''$. We obtain a (unique) solution $u'$ in $\Omega''$ which satisfies $\Delta u'\in L^2_\loc(\Omega'')$ and such that \eqref{eq:quasilinear estimate} holds with   $u'$ and $\Omega'$. However also $u\restr{\Omega''}$ solves the same equation by definition, thus by uniqueness $u\restr{\Omega''}=u'$. Hence we obtain that \eqref{eq:quasilinear estimate} holds for $u$ in $\Omega'$. By the arbitrariness of $\Omega'$ this concludes the proof.
	\end{proof}

	{
		\begin{rem}\label{rmk:non autonomous summary thm2}
			We observe that Theorem \ref{thm:main quasilinear} automatically self-improves to operators of the form
			$\div(\aa\, \Psi(|\nabla u|)\nabla u),$ where $\aa\in \LIP(\Omega)$  is such that $A^{-1}\le \aa\le A$ for some $A\ge 1.$ This is an immediate consequence of  the Leibniz rule for the divergence.
	\end{rem}}

	\section{Quasilinear Equations with p-growth}\label{sec:p harmonic}

	{
	In this section we consider degenerate/singular quasilinear elliptic operators satisfying \eqref{eq:psi' condition} and the $p$-growth condition \eqref{eq:psi growth}, \textit{e.g.}\ the $p$-Laplacian. In particular we prove Theorem \ref{thm:p-delta regularity}. The idea is to {argue through} Theorem \ref{thm:main quasilinear} by {obtaining the relevant estimates} applied to the regularized-truncated function $\Psi\big(\sqrt{(t\wedge M)^2+\delta}\, \big)$ {and successively letting $M\to\infty$ and $\delta \to 0$ (in this specific order)}. This procedure is rather delicate as the truncation changes the growth conditions of the operator from $p$ to $2.$ 
	
	We will rely  on the following two crucial converge{nce} results{: Prop.\  \ref{prop:approx of delta by M-delta} (concerned with the limit $M\to\infty$) and Prop.\  \ref{prop:approx W1p Fphi by Fphidelta} (concerned with the limit $\delta \to 0$)}.
	{Broadly speaking,} Prop.\  \ref{prop:approx of delta by M-delta} says that the solution to the equation \eqref{eq:quasilinear equation delta} with $\Psi(t)$ replaced by $\Psi(t\wedge M)$ converges strongly in the Sobolev sense to the solution of \eqref{eq:quasilinear equation delta} as $M\to +\infty$. This result is a delicate  interplay between exponent $p$ and $2$. More precisely $\Psi(t\wedge M)$ does not have $p$-growth as $\Psi$, but instead has $2$-growth (in the sense of \eqref{eq:psi growth}), indeed it is bounded above and below by positive constants  for $t\ge 1$. Hence, simple variational arguments can not yield the sought convergence of solutions and we {must} argue via convexity and monotonicity. Note that the statement is formulated in the equivalent terms of minimizers of the corresponding functionals $F_\Phi,F_{\Phi_M}$ defined in \eqref{eq:F_Phi def}.}
	\begin{prop}[Convergence of $M$-truncations]\label{prop:approx of delta by M-delta}
		Let $\Psi\in {\sf AC}_\loc(0,\infty)$ satisfy \eqref{eq:psi' condition} and \eqref{eq:psi growth} for $p\in(1,\infty)$. Suppose also that $\Psi\ge c>0$ in the case $p\ge 2.$  Let $\Xdm$ be an $\RCD(K,\infty)$ space and let $\Omega\subset \X$ be open, bounded, with  $\mm(\X\setminus\Omega)>0$ and such that
		$W^{1,p}_0(\Omega)=\widehat W^{1,p}_0(\Omega).$
		Let $f\in L^2\cap L^{p'}(\Omega)$, where $p'\coloneqq \frac{p}{p-1}$, and $g\in W^{1,p \vee 2}(\Omega)$. For all $M>0$ set $\Psi_M(t)\coloneqq \Psi(t\wedge M)$, $\Phi_M(t)\coloneqq \int_0^t s\Psi_M(s)\d s$ and $\Phi(t)\coloneqq \int_0^t s\Psi(s)\d s$.  Let $u_M\in g+W^{1,2}_0(\Omega)$, $u\in g+W^{1,p}_0(\Omega)$ be respectively  the unique minimizers of $F_{\Phi_M}$ and $F_{\Phi}$ in their respective domains. Then $$u_M\to u \quad \text{strongly in }W^{1,{q}}(\Omega), \qquad q\coloneqq \min \{p,2\}.$$ 
		Moreover, it holds
		\begin{equation}\label{eq:uniform bound for uM}
			\| \,\Psi_M(|\nabla u_M|)\, |\nabla u_M| \,\|_{L^1(\Omega)}\le C \int_\Omega \Big( |f|^{p'\vee 2}+|g|^{ p}+|\nabla g|^{p\vee 2}+1 \Big) \d \mm,
		\end{equation}
		where $C$ is  a constant depending only on $\Psi,p$ and $\Omega.$ 
	\end{prop}
	
	{Estimate \eqref{eq:uniform bound for uM} should be interpreted as a $M$-truncated uniform $L^{p-1}$-estimates on $|\nabla u_M|$. Indeed the $p$-growth assumption \eqref{eq:psi growth} gives that $\Psi_M(t)t\geq (t\wedge M)^{p-2}t $ for $t\ge 1.$ Note also that $L^1$-control on  $\Psi_M(|\nabla u_M|)\, |\nabla u_M|$ is the one needed for the uniform estimate in Prop.\  \ref{cor:key estimate}.}
	We will  employ Prop.\  \ref{prop:approx of delta by M-delta}  to    take the limit $M\to \infty$ in $\Psi(\sqrt{(|\nabla u |\wedge M)^2 + \delta})$,  before investigating the limit $\delta \to 0$.  Indeed  $\Psi(\sqrt{t^2+\delta})$ is bounded below by a positive constant in the case $p\ge 2$, from \eqref{eq:psi growth}. {We note that it is harmless for \eqref{eq:uniform bound for uM} to depend on $\Psi$ rather than only $\lambda,\Lambda $, since sending $\delta \to 0$, we will rely on the uniform estimate given by  Prop.\  \ref{cor:key estimate}. }

	The second {convergence} result (Prop.\  \ref{prop:approx W1p Fphi by Fphidelta}) {is now concerned with the $\delta \to 0$ limit of }  $\Psi(\sqrt{t^2+\delta})$. Again,  for brevity, the statement is in terms of minimizers.
	\begin{prop}[Convergence of $\delta$-regularization]\label{prop:approx W1p Fphi by Fphidelta}
		Let $\Psi\in {\sf AC}_\loc(0,\infty)$ satisfy \eqref{eq:psi' condition} and \eqref{eq:psi growth} for some $p\in(1,\infty)$.   Let $\Xdm$ be an $\RCD(K,\infty)$ space and let $\Omega\subset \X$ be open, bounded with $\mm(\X\setminus\Omega)>0$.  Set $\Phi(t)\coloneqq \int_0^ts\Psi(s) \d s $ and 
		$	\Phi^\delta(t)\coloneqq \Phi(\sqrt{t^2+\delta}).$
		Let $g\in W^{1,p}(\Omega)$ and let $u_\delta, u\in g+W^{1,p}_0(\Omega)$ be respectively the unique minimizers of $F_{\Phi^\delta}$ and $F_\Phi$ in $g+W^{1,p}_0(\Omega)$.
		Then $u_\delta\to u$ {strongly} in $W^{1,p}(\Omega)$ as $\delta \to 0^+.$
	\end{prop}
	
	{In  the proof of Theorem \ref{thm:p-delta regularity}, $f\in L^2(\Omega)$ may not imply $f\in L^{p'}(\Omega)$ for $p<2,$ which obstructs variational arguments. We  address this via the following result.  }
	\begin{prop}\label{prop:continuity wrt f}
		Let $\Psi\in {\sf AC}_\loc(0,\infty)$ satisfy \eqref{eq:psi' condition} and \eqref{eq:psi growth} for some  $p\in(1,\infty)$. {Let $\Xdm$ be an $\RCD(K,\infty)$ space, let $\Omega\subset \X$ be open bounded} and such that $\mm(\X\setminus\Omega)>0$ and  let $f_n \in L^2(\Omega)$ converge strongly in $L^2(\Omega)$ to some $f \in L^2(\Omega)$. Let also $g\in W^{1,p}(\Omega)$ and suppose  that $u$ and $u_n$ are respectively solutions of 
		\begin{equation}\label{eq:two equaions f,f_n}
			\begin{cases}
				\div(\Psi(|\nabla u|)\nabla u) u=f, & \text{in $\Omega$},\\
				u\in g+W^{1,p}_0(\Omega),
			\end{cases}
			\quad \quad 
			\begin{cases}
				\div(\Psi(|\nabla u_n|)\nabla u_n)=f_n, & \text{in $\Omega$},\\
				u_n\in g+W^{1,p}_0(\Omega).
			\end{cases}
		\end{equation}
		Then   $|\nabla (u_n-u)|\to 0$ in $L^{p-1}(\Omega).$ In particular, the sequence $\{|\nabla u_n|\}_{n}$ is bounded  in $L^{p-1}(\Omega).$
	\end{prop}

	{For clarity of presentation,} the proofs of Prop.\  \ref{prop:approx of delta by M-delta}, Prop.\  \ref{prop:approx W1p Fphi by Fphidelta} and Prop.\  \ref{prop:continuity wrt f} are postponed to \S \ref{sec:proof of propositions}. We {proceed by proving Theorem \ref{thm:p-delta regularity}, using all three of the aforementioned propositions}.
	\begin{proof}[Proof of Theorem \ref{thm:p-delta regularity}]
		The proof is divided into several steps in order of increasing generality. We first prove the result under the following additional assumptions:
		\begin{enumerate}[label=\Alph*), ref=\Alph*)]
			\item \label{it: W=hat W} $W^{1,p}_0(\Omega)=\widehat W^{1,p}_0(\Omega)$ (recall \eqref{eq:hat w1p});
			\item \label{it: g extends} there exists $\bar g \in \W(\X)$ such that $u=\bar g\restr \Omega$ (in particular  $u\in W^{1,\max\{2,p\}}(\Omega)$).
			\item \label{it: Psi good} $\Psi\ge c>0$ if $p\ge 2$, moreover for all $M>0$ it holds that  $C_M^{-1}\le \Psi(t\ww M)\le C_M$ for all $t>0$ and some constant $C_M>1.$ 
			\item \label{it: f in dual} $f\in L^2\cap L^{p'}(\Omega)$, where $p'\coloneqq \frac{p}{p-1}.$
		\end{enumerate}

		\smallskip 
		
		\noindent 1. \textit{Proof under assumptions \ref{it: W=hat W}, \ref{it: g extends}, \ref{it: Psi good}, \ref{it: f in dual}} : 
		Recall first that $\Psi_M(t)\coloneqq \Psi(t\wedge M)$ also satisfies \eqref{eq:psi' condition} with the same $\lambda,\Lambda$ as $\Psi$ (see Lemma \ref{lem:regularity of truncated psi}). Hence by \ref{it: Psi good} the function $\Psi_M$ satisfies the assumptions of Theorem \ref{thm:main quasilinear}. Therefore, {by applying Theorem \ref{thm:main quasilinear}}, for all numbers $M>0$  there exists a unique weak solution $u_M\in u+\W_0(\Omega)${---we recall that in the statement of the theorem we assume the existence of $u$ solution of \eqref{eq:quasilinear equation delta}---}  of 
		\begin{equation}\label{eq:p-delta poisson thm}
			\div\big( \Psi_M(|\nabla u_M|)\,\nabla u_M\big)=f,
		\end{equation}
		with $\Delta u_M \in L^2_\loc(\Omega)$. Applying  Prop.\  \ref{cor:key estimate}  to $ u_M$ (up to extending $u_M$ to $\X$ as in Lemma \ref{lem:cut-off lemma}) we obtain that for all $B_R(x)\subset \subset \Omega$ with $R\le 1,$ it holds
		\begin{equation}\label{eq:key applied uM}
			\begin{split}
				\fint_{B_{R/2}(x)} \Big( &R^{-2}|V_M|^2 +  |\nabla V_M|^2 \Big) \d \mm \le C
				\fint_{B_R(x)}   f^2 \d \mm 
				+ C \left(\fint_{B_R(x)}|V_M| \d \mm\right)^2,
			\end{split}
		\end{equation}
		where $V_M\coloneqq \Psi_M(|\nabla u_{{M}}|)\nabla u_M$ and $C$ is a constant depending only on $K,N,\Lambda,\lambda$. Set  $h_M\coloneqq\Psi_M(|\nabla u_M|)|\nabla u_M|\in L^2(\Omega)$.
		Thanks to assumptions \ref{it: W=hat W}, \ref{it: g extends}, \ref{it: Psi good}, \ref{it: f in dual} we apply   Prop.\  \ref{prop:approx of delta by M-delta} and get 
		\[
		\| h_M \,\|_{L^1(\Omega)}\le C \int_\Omega \Big( |f|^{p'\vee 2}+|u|^{p}+|\nabla u|^p+1 \Big) \d \mm; 
		\]
		$C$ is a constant depending only on $\Psi,p,\Omega,$ but not $M.$ Therefore, {the above and \eqref{eq:key applied uM} imply } 
		\[
		\begin{split}
			\int_{B_{R/2}(x)} \Big( &R^{-2}h_M^2 +  |\nabla h_M|^2 \Big) \d \mm \le C(p,{\Vert f \Vert_{L^2}},{\Vert u \Vert_{W^{1,p}}},R,x,\Psi,\Omega)
		\end{split}
		\]
		{where the constant $C(p,{\Vert f \Vert_{L^2}},{\Vert u \Vert_{W^{1,p}}},R,x,\Psi,\Omega)>0$  is a constant independent of $M$}.    This shows that the sequence of functions $\{h_M\}_M$ is  uniformly bounded in $\W_\loc(\Omega)$ independently of $M$. { Prop.\  \ref{prop:approx of delta by M-delta} also gives that  $|\nabla u_M-\nabla u|\to 0$ in $L^q(\Omega),$ where $q\coloneqq \min \{2,p\}.$ Hence we can easily verify that, up to passing to a subsequence,
			\begin{equation}\label{eq:strong gradient convergence in proof}
				h_M=\Psi_M(|\nabla u|)\,|\nabla u_M|\to \Psi(|\nabla u|)\,|\nabla u| , \quad \alme \text{ in $\Omega$.}
			\end{equation}
			By the Rellich Theorem (see  \cite[Thm.\ 6.3-ii)]{GigliMondinoSavare13} or \cite[\S 8]{HajlaszKoskela00}), up to  a further  subsequence,  $\{h_M\}_M$ converges {strongly} in $L^2_\loc(\Omega)$ to some $h\in L^2_\loc(\Omega)$, and \eqref{eq:strong gradient convergence in proof} implies $h=\Psi(|\nabla u|)\,|\nabla u|$. Therefore
			\begin{equation}\label{eq:convergence of integral gradM}
				\int_{B_R(x)}\Psi_M(|\nabla u_{{M}}|)|\nabla u_M| \d \mm \to \int_{B_R(x)}\Psi(|\nabla u|)\,|\nabla u| \d \mm, \quad \text{as $M\to +\infty.$}
			\end{equation}
			  By lower semicontinuity of the energy (see \cite[Lemma 2.18]{ivan1})  combined with \eqref{eq:key applied uM}, \eqref{eq:strong gradient convergence in proof} and \eqref{eq:convergence of integral gradM} we obtain both that 
			$\Psi(|\nabla u|)\nabla u \in W^{1,2}_{C,\loc}(T\X;\Omega)$ and  \eqref{eq:key applied udelta}.} {Finally \eqref{eq:gradient estimate for udelta} follows immediately from Prop.\  \ref{prop:lip a priori} applied to $u_M$ (note that $\Psi_M$ satisfies the assumptions), the fact that  $|\nabla u_M-\nabla u|\to 0$ $\mm$-a.e.\ in $\Omega$ and again \eqref{eq:convergence of integral gradM}}. Note also that by \eqref{eq:psi growth} it holds $\Psi(1)\le \nu.$

		\smallskip 
		
		\noindent 2. \textit{Removing assumptions \ref{it: W=hat W} and \ref{it: g extends}}: We now consider any $\Omega$ and  $u\in W^{1,p}(\Omega).$ {The existence and uniqueness of the solution $u$ still holds by Prop.\  \ref{lem:psi and phi} and the other  results in \S \ref{sec:minim}.}
		Since the required regularity estimates are local, we can replace $\Omega$ with any $B_R(x)\subset\subset\Omega$ and $u$  with $u \restr {B_R(x)}.$ By Lemma \ref{lem:good balls}, up to enlarging $R$ {if necessary}, we can assume that $W^{1,p}_0(B_R(x))=\widehat W^{1,p}_0(B_R(x)).$ 
		Moreover there exists $\bar g\in W^{1,p}(\X)$ such that $\bar g\restr {B_R(x)}=u \restr {B_R(x)}$, \textit{e.g.}\ multiplying $u$ by a  cut-off function (\textit{cf.}~the last step of the proof of Theorem \ref{thm:main quasilinear}). From now on we denote $\Omega=B_R(x).$
		By density we can find $\bar g_n \in \LIP_{bs}(\X)\subset \W\cap W^{1,p}(\X)$ such that $g_n\to \bar g $ in $W^{1,p}(\X)$. Set $g_n\coloneqq \bar g_n\restr\Omega$ and let $u_n\in g_n+W^{1,p}_0(\Omega)$ be the solution to \eqref{eq:quasilinear equation delta}. By Step 1 the statement holds for $u_n$ and $g_n$.  Hence it is sufficient to show that  $u_n\to u $ strongly in $W^{1,p}(\Omega)$, since the regularity estimates pass to the limit from $u_n$  to $u$ as in  the previous step.   
		This convergence is  standard, but we provide a short argument. Observe that $\{|{\nabla }u_n|\}_n$ is uniformly bounded in ${L^{p}}(\Omega)$, {since each $u_n$ is a minimizer of the functional $F_\Phi$ in $g_n + W^{1,p}_0(\Omega)$ and $\Phi$ is coercive; \textit{cf.}~\eqref{eq:F_Phi def} for the definition of $F_\Phi$. {Furthermore, $\{g_n\}_n$ is uniformly bounded in $W^{1,p}(\Omega)$ as it converges therein}. Then,  by }the Poincaré inequality, {we deduce that ${\{u_n\}_n}$ is bounded in $W^{1,p}(\Omega)$. {Setting $v_n\coloneqq u_n+(u-g_n)\in u+W^{1,p}_0(\Omega)$}, we have that $v_n$ converges weakly up to a subsequence to some $v\in {g+W^{1,p}_0(\Omega)}.$ }Hence $F_\Phi(v)\le \liminf_n F_\Phi(u_n)$ by Lemma \ref{lem:lsc energy}. On the other hand by minimality $F_{\Phi}(v_n)\ge F_\phi(u).$ However $\|v_n-u_n\|_{W^{1,p}(\Omega)}=\|u-g_n\|_{W^{1,p}(\Omega)}\to 0$ and thus $\liminf_n F_{\Phi}(v_n)=\liminf_n F_{\Phi}(u_n)$ {(\textit{cf.}\ Lemma \ref{lem:lsc energy})}. We deduce that $F_\Phi(v) \le F_\Phi(u)$ which gives that $v=u$ and $\lim_n F_{\Phi}(u_n)=F_\Phi(u)$.  From Lemma \ref{lem:effective monot}  with $\Psi(t)$, which  satisfies $\Psi(t)\ge \nu t^{p-2}$ for $t\ge 1$,  arguing as in \eqref{eq:monotonicity quant trick p>2} we deduce {$v_n\to v$ in $W^{1,p}(\Omega)$ } and so {$u_n\to u$} in $W^{1,p}(\Omega)$.
		
		\smallskip
		
		\noindent 3. \textit{Removing assumption \ref{it: Psi good}}: Consider now a general $\Psi$ as in the assumptions {of Theorem \ref{thm:p-delta regularity}} and for all $\delta\in(0,1)$ define $\Psi^\delta(t)\coloneqq \Psi (\sqrt{t^2+\delta}\,).$ From Lemma \ref{lem:regularity of truncated psi} we have that $\Psi^\delta$ satisfies \eqref{eq:psi' condition} with the same $\lambda,\Lambda.$ Moreover by \eqref{eq:psi growth} for $\Psi$ we have, for {all $p \in (1,\infty)$} and all $t\ge 1$, 
		\begin{align*}
			&C_p^{-1}\nu^{-1} t^{p-2}\le  \nu^{-1} (t^2+\delta)^\frac{p-2}{2}\le \Psi^\delta(t)\le \nu (t^2+\delta)^\frac{p-2}{2}\le C_p\nu t^{p-2},
		\end{align*}
		where $C_p>1$ depends only on $p$, having used that $\frac{1}{2}\le \frac{t^2}{t^2+\delta}\le 1$ {for $t \geq 1>\delta$}. Therefore  $\Psi^\delta$  also satisfies \eqref{eq:psi growth} with constant $C_p\nu$. Moreover  $\Psi^\delta(t\wedge M)$ clearly belongs to $\LIP([0,\infty)$ and it is globally bounded above and below by positive constants. Finally for $p\ge 2$  we have that $\Psi^\delta\ge c>0$ for some constant $c>0;$  indeed  by \eqref{eq:psi growth} it holds $\Psi^\delta(t)\ge \nu\delta^\frac{p-2}{2}$ for all $t\ge 1$ and $\Psi$ is positive and continuous in $(0,\infty)$. This shows that $\Psi^\delta$ satisfies {hypothesis} \ref{it: Psi good}.
		
		Let $u_\delta$ be the solution of \eqref{eq:quasilinear equation delta} with $\Psi=\Psi^\delta.$ By the previous step we have that the statements holds for $\Psi^\delta$ and $u_\delta$ with the same $\lambda,\Lambda.$ To conclude it  is sufficient to show that  $u_\delta\to u $ strongly in $W^{1,p}(\Omega)$, since the needed regularity estimates  pass to the limit from $u_\delta$  to $u$ as we did above. This convergence is precisely the content of Prop.\  \ref{prop:approx W1p Fphi by Fphidelta}.
		
		\smallskip
		
		\noindent 4. \textit{Removing assumption \ref{it: f in dual}}: Consider now a general $f\in L^2(\Omega).$ Take $f_n\coloneqq (-n)\vee f \wedge n\in L^{p'}\cap L^2(\Omega)$. In particular $f_n\to f$ in $L^2(\Omega).$ By virtue of Prop.\  \ref{prop:existence variational}, for each $n\in\mathbb{N}$ there exists a unique $u_n\in u+W^{1,p}_0(\Omega)$ such that $\div( \Psi(|\nabla u_n|)\, \nabla u_n)=f_n$ in $\Omega$. Moreover, by Prop.\  \ref{prop:continuity wrt f}, we have that $|\nabla u_n-\nabla u|\to0$ in $L^{p-1}(\Omega).$ From the previous step  the {conclusion of the theorem} holds for $u_n$. Hence, sending $n\to +\infty$ we deduce the result for $u$ passing to the limits in the regularity estimates as we did {in Step 1 of the proof}. {  {From our definition of $f_n$} we get $\fint_{B_{R}(x)} |f_n|^{q} \d \mm \le C_0$, hence  \eqref{eq:gradient estimate for udelta} for $u_n$ holds with the same constant $\tilde C$ for all $n.$ {By passing again to the limit $n\to\infty$, we recover \eqref{eq:gradient estimate for udelta} for $u$ the solution of \eqref{eq:quasilinear equation delta}.}}
	\end{proof}
	
	{
		\begin{rem}\label{rmk:non autonomous summary}
			The argument for the proof of Theorem \ref{thm:p-delta regularity} can be easily adapted to the case $\div(\aa \Psi(|\nabla u|)\nabla u)$ where $\aa\in \LIP(\Omega)$  is bounded above and below by positive constants. Indeed Theorem \ref{thm:main quasilinear} also holds in this case, as observed in Remark \ref{rmk:non autonomous summary thm2}. Moreover all the ``first-order'' results such as existence, uniqueness and approximation arguments (\textit{e.g.}\ the ones in \S \ref{sec:minim} or Propositions \ref{prop:approx of delta by M-delta}-\ref{prop:approx W1p Fphi by Fphidelta} work the same way (in fact for this the Lipschitz regularity of $\aa$ is not needed). Finally, it was observed that all the second order-uniform estimates extend  to this case (see Remark \ref{rmk:non autonom version 2}). Thus the  argument carries over with the relevant modifications.  \fr
	\end{rem}}

{
\begin{rem}\label{rem:nondeg detailed}
If we assume that $\Psi\ge c>0$  in Theorem \ref{thm:p-delta regularity}, for $p\ge 2$, then we  deduce that $u \in H^{2,2}_\loc(\Omega)$ if $p\ge 2.$ The proof is identical and we omit it for brevity. The key  is the uniform $L^2$-Hessian bound given by  Prop.\  \ref{prop:key estimate hessian}, which then passes to the limit exactly as per \eqref{eq:key}. In particular estimate \eqref{eq:key hessian} would also hold for $u.$  If $p< 2 $, we cannot assume $\Psi$ is bounded below in view of \eqref{eq:psi growth}. Instead the standard non-degeneracy assumption for $p\in(1,2)$ would be
\begin{equation}\label{eq:non deg p<2}
    c(t^2+1)^\frac{p-2}{2}\le \Psi(t)   \le  c^{-1}(t^2+1)^\frac{p-2}{2}
\end{equation}
and the expected result is $u\in W^{2,p}$ (see \textit{e.g.}\ \cite[Chapter 8]{giustibook}). However, the space $W^{2,p}(\X)$, for $p<2$,  is not even clearly defined in RCD setting unless one also assumes the gradient to be $L^2$ (see comments at the beginning of \cite[\S 3.3.3]{Gigli14})---we refrain from treating this case.  \fr
\end{rem}}

{
With Theorem \ref{thm:p-delta regularity} at hand we can now prove the Cheng-Yau type inequality for $p$-harmonic functions. { The argument relies heavily on the one presented in \cite{WZ11}---for brevity, many details are omitted and we only provide an account of where the strategy must be adjusted.}
\begin{proof}[Proof of Corollary \ref{cor:pharm}]
    Set $B\coloneqq B_R(x).$ We can assume $\mm(\Omega\setminus B)>0.$ Theorem \ref{thm:main plap} gives already that $u\in \LIP_\loc(B)$. Hence we focus on \eqref{eq:CY}. Since we know also $\gradu^{p-1}\in \W_\loc(B),$ we can argue as in  \cite{WZ11}. We thus outline only the main points.
Set $w\coloneqq -(p-1)\log(u)$ and ${\bf g}\coloneqq |\nabla w|^2$. {Since $u \in \LIP_\loc(B)$ is positive on $B$, we deduce} $w\in \LIP_\loc(B)$. {By the chain rule, $|\nabla u|^p = (|\nabla u|^{p-1})^{\frac{p}{p-1}} \in \W_\loc(B)$,  since $|\nabla u|^{p-1} \in W^{1,2}_\loc$. Then,} by the Leibniz rule ${\bf g}^{p/2}\in L^\infty_\loc\cap \W_\loc(B)$. Moreover by the chain rule $w$  solves $\Delta_p w={\bf g}^{p/2}$ in $B$. It is sufficient to deduce the integral inequality in \cite[eq.\ (2.7) p.\ 766]{WZ11} involving ${\bf g}$, from which \eqref{eq:CY} follows verbatim as in \cite{WZ11}. A technical difficulty is that to obtain \cite[eq.\ (2.7)]{WZ11}  it is used therein that $u\in C^{1,\alpha}$ and $C^\infty$ away from the critical set {$\{|\nabla u | = 0\}$}, while here this is not possible as $|\nabla u|$ is not even continuous. Instead we argue via the regularized {{$\Delta_{p,\delta}$}} operator.
For all $\delta>0$ consider the (unique) solution to
\[
\begin{cases}
    \Delta_{p,\delta} \,v_\delta={\bf g}^{p/2}, & \text{in $B$,}\\
     v_\delta \in w +W^{1,p}_0(B).
\end{cases}
\]
By Prop.\  \ref{prop:approx W1p Fphi by Fphidelta} we have that $v_\delta \to w$ in $W^{1,p}(B)$. Moreover, since ${\bf g}\in L^\infty_\loc(B)$, by Theorem \ref{thm:main quasilinear}  $\Delta v_\delta \in L^2_\loc(B)$ and { by Theorem \ref{thm:p-delta regularity}} $v_\delta \in \LIP_\loc(B)$ with $\sup_{\delta>0}\||\nabla v_\delta|\|_{L^\infty(K)}<\infty$ for all compact sets $K\subset B.$ In particular  $f_\delta\coloneqq|\nabla v_\delta|^2+\delta\in L^\infty_\loc\cap \W_\loc(B) $ and
\begin{equation}\label{eq:fdelta to g}
    f_\delta \overset{L^q(K)}{\longrightarrow} {\bf g} \quad \text{ as $\delta \to 0^+$, \, for all compact sets $K\subset B$  and all $q<\infty$.}
\end{equation}
Below $\delta$ is fixed and we write $f,v$ in place of $f_\delta,v_\delta$ { for brevity}.
For all $\phi \in L^\infty\cap \W(\X)$ with compact support in $B$, it holds
\begin{equation}\label{eq:trolololol}
    \begin{aligned}
            \mathcal{L}_f(\phi) &\coloneqq - \int \Big( \,\big\langle\, \nabla \phi,  f^{\frac{p}{2}-1} \nabla f + (p-2) \nabla v f^{\frac{p}{2}-2}\la \nabla v, \nabla f \ra \,\big\rangle  + 2\la \nabla v, \nabla \Delta_{p,\delta}v\ra \phi \Big) \d \mm\\
                      &\ge \int  2\phi f^{\frac{p}{2}-1}\big ( |{\bf H} v|^2 + K|\nabla v|^2 \big) + \phi \la \nabla  f^{\frac{p}{2}-1},  \nabla f \ra  -  2\div(\nabla v\phi) \Delta_{p,\delta}  v\d \mm\\
                      &\quad - 2\int  f^{p/2-1} \div(\nabla v\phi)\Delta v +
                       (p-2)f^{p/2-2} \la \nabla v, \nabla f \ra \div(\nabla v\phi)     \d \mm\\
                       &=\int  2\phi f^{p/2-1}\big ( |{\bf H} v|^2 + K|\nabla v|^2 \big) + \big(\frac p2-1\big)|\nabla f|^2 f^{\frac{p}{2}-2} \phi  \d \mm,
    \end{aligned}
\end{equation}
where in first inequality we used the Bochner inequality of the second part of Theorem \ref{thm:improved Bochner} { with test function $2\varphi f^{\frac{p}{2}-1}$} {(extending $v$ using Lemma \ref{lem:cut-off lemma})}, { the relation $\nabla f= \nabla |\nabla v|^2$, and we integrated by parts the final term in $\mathcal{L}(f)(\varphi)$,} and in the last line we used that $\Delta_{p,\delta }v {=\div(f^{\frac{p}{2}-1}\nabla v)} = f^{\frac{p}{2}-1} \Delta v + \frac{(p-2)}{2} f^{\frac{p}{2}-2} \la \nabla v, \nabla f \ra  $. Inequality \eqref{eq:trolololol} replaces \cite[Lemma 2.1]{WZ11}. Computing in coordinates (see \cite[p.\ 763]{WZ11} and {\textit{cf.}\ \cite[Lemma 4.6]{GigliViolo23}}) gives
\begin{equation}\label{eq:hessian jensen}
    |{\bf H} v|^2\ge \frac14 f^{-1}|\nabla f|^2+ \frac{f^{2-p}(\Delta_{p,\delta} v)^2}{N-1}-c_{p,N} f^{-\frac{p}{2}}|\nabla v||\nabla f||\Delta_{p,\delta}  v|,
\end{equation}
where $c_{p,N}>0$ is a constant depending only on $N$ and $p.$ Plugging \eqref{eq:hessian jensen} into \eqref{eq:trolololol}
\begin{equation}
\begin{aligned}
    & -\mathcal L_f(\phi)+\int \frac{2f^{1-p/2}{\bf g}^{p}}{N-1}\phi\d\mm\le C_{N,p}\int  \left(K^-f^{p/2} +  f^{-1/2}|\nabla f|{\bf g}^{p/2}\right) \phi \d \mm.
    \end{aligned}
\end{equation}
Choosing $\phi=\eta^2f^b$ for any $b>p/2$,  rearranging and using  the Young's inequality as in \cite[pag.\ 765]{WZ11} gives   (reintroducing the notation $f_\delta=f$) 
\begin{align*}
    \int b|\nabla &f_\delta^\frac{p}{2}|^2  f_\delta^{b-p/2}\eta^2 \d \mm + \int {f_\delta^{1+b-p/2}{\bf g}^{p}}\eta^2\d\mm\\
    &\le C_{N,p}\int  \left(K^-f_\delta^{p/2+b}+f_\delta^{1/2+b}|\nabla {\bf g}^{p/2}| \right)  \eta^2 +\frac1bf_\delta^{1+b-p/2} {\bf g}^{p}\eta^2+ \frac1bf_\delta^{p/2+b}|\nabla \eta|^2 \d \mm.
\end{align*}
By \eqref{eq:fdelta to g} all terms  pass to the limit { as $\delta \to 0^+$}, except the one with $|\nabla f_\delta|$ which is lower semicontinuous being a constant multiple of $\int |\nabla f_\delta^{p/4+b/2}|\eta^2\ \d \mm$. Hence we obtain the same inequality replacing $f_\delta$ by ${\bf g}$, which after  reabsorbing $|\nabla {\bf g}^{p/2}|$ using the Young's inequality gives \cite[eq.\ (2.7)]{WZ11} as desired.
\end{proof}}

{
For completeness we prove below the applications listed at the end of the introduction.

\begin{proof}[Proof of Corollary \ref{cor:sobolev}]
    We can assume that $\|u\|_{L^{p^*}(\mm)}=1$. By minimality $u$ satisfies the Euler-Lagrange equation {associated to $\mathscr{F}[v] = A \||\nabla v|\|_{L^p(\mm)}^p+B\|v\|_{L^p(\mm)}^p - \Vert v \Vert_{L^{p^*}(\mm)}^p$, namely}:
    \[
    -A \Delta_p u= |u|^{p-2}u ( |u|^{p^*-p} -B).
    \]
   Since the right-hand side is of the form $g |u|^{p-2}u$ with $g=|u|^{p^*-p} -B \in L^{N/p}_\loc(\mm)$, standard first order iteration arguments  that readily adapt to this setting show that $u\in L^\infty_\loc(\mm)$ (see \textit{e.g.}\ \cite[Theorem E.0.20]{peral}). We conclude by applying Theorem \ref{thm:main plap}.
\end{proof}

\begin{proof}[Proof of Corollary \ref{cor:minimal surf}]
    Since $u$ is assumed locally Lipschitz,  in every ball $B\subset\subset \Omega$ it solves $\div(\Psi_M(|\nabla u|)\nabla u)=0$, where $\Psi_M(t)\coloneqq ((t\wedge M)^2+1)^{-\frac12}$, for $M$ big enough {depending on $\Vert \nabla u \Vert_{L^\infty(B)}$}. Moreover, $\Psi_M$ satisfies \eqref{eq:psi' condition} with $\lambda=\frac{-M^2}{M^2+1}>-1$ and $\Lambda=0.$ Moreover $c<\Psi_M\le c^{-1}$ for some $c>0$ {again depending on $\Vert \nabla u \Vert_{L^\infty(B)}$}. We conclude by applying Theorem \ref{thm:main quasilinear}.
\end{proof}

\begin{proof}[Proof of Corollary \ref{cor:eigein}]
   Since $u$ is locally bounded in $\Omega$ (see \textit{e.g.}\ \cite[Theorem A.4]{ivan1}) we can directly apply Theorem \ref{thm:main plap}.
\end{proof}

}

\section{Proofs of Propositions  \ref{prop:approx of delta by M-delta}, \ref{prop:approx W1p Fphi by Fphidelta} and \ref{prop:continuity wrt f}} \label{sec:proof of propositions}

	\subsection{Approximation of $F_\Phi$ via $M$-truncation and $\delta$-regularization}\label{app:variational}
	In this section we prove Prop.\  \ref{prop:approx of delta by M-delta} and Prop.\  \ref{prop:approx W1p Fphi by Fphidelta}, sated in \S \ref{sec:p harmonic} and used to prove Theorem \ref{thm:p-delta regularity}.

	\begin{proof}[Proof of Proposition \ref{prop:approx W1p Fphi by Fphidelta}]
		Recall $\Phi$ is given by $\Phi(t) = \int_0^t s \Psi(s) \d s$, and $\Phi^\delta(t)\coloneqq \Phi(\sqrt{t^2+\delta})$.
		
		The minimizers $u,u_\delta$ exist unique by Prop.\  \ref{prop:existence variational}; indeed $\Phi,\Phi^\delta$ are {$p$-}admissible in the sense of Definition \ref{def:admissible Phi}, { strictly convex and satisfies \eqref{eq:coerciveness} with a uniform constant independent of $\delta$, thanks to Lemma \ref{lem:psi and phi}}.  The rest of the proof is divided into three steps. 
		
		\smallskip 
		
		\noindent 1. \textit{Uniform bounds in $W^{1,p}$}: Standard estimates using that $u_\delta$ is a minimizer, the {uniform coercivity of $\Psi_\delta$} and the Poincaré inequality ({applied to $u_\delta \!-\! g \in W^{1,p}_0(\Omega)$}) show
		\begin{equation}\label{eq:to get the W1p bound on udelta}
			\begin{aligned}
				\int_\Omega |\nabla u_\delta|^p\d \mm &\leq C \Big( F_{\Phi^\delta}(u_\delta) + \Vert f \Vert_{L^{p'}(\Omega)}^{p'} + \Vert g \Vert_{W^{1,p}(\Omega)}^p \Big), 
			\end{aligned} 
		\end{equation}
		where $C>0$ is a constant independent of $\delta.$
		Meanwhile, by minimality and monotonicity, 
		\begin{equation}\label{eq:chain of ineqs Phidelta Phi}
			F_{\Phi^\delta}(u_\delta)\le F_{\Phi^\delta}(u_1) \le F_{\Phi^1} (u_1)<\infty,
		\end{equation}
		which, combined with \eqref{eq:to get the W1p bound on udelta} and the Poincaré inequality (Prop.\  \ref{prop:poincare}), implies  the sequence $\{u_\delta\}_\delta$ is uniformly bounded in $W^{1,p}(\Omega).$ Hence up to passing to a subsequence which we do not relabel, we have that $u_\delta\rightharpoonup v$ weakly in $W^{1,p}(\Omega)$ for some limit function $v\in g+W^{1,p}_0(\Omega).$  
		It remains to show $v=u$ and $u_\delta \to u$ strongly in $W^{1,p}(\Omega)$. 
		
		\smallskip 
		
		\noindent 2. \textit{Showing $v=u$}: As $u_\delta \rightharpoonup v$ in $W^{1,p}(\Omega)$, the weak lower semicontinuity provided by Lemma \ref{lem:lsc energy} implies $ F_\Phi(v)\le \liminf_{\delta\to 0} F_\Phi(u_\delta).$ Therefore, arguing analogously to \eqref{eq:chain of ineqs Phidelta Phi},
		\begin{equation}\label{eq:comparing Fv with Fu}
			\begin{aligned}
				F_\Phi(v)&\le \liminf_{\delta\to 0} F_\Phi(u_\delta)\le  \liminf_{\delta\to 0} F_{\Phi^\delta}(u_\delta) \le \limsup_{\delta\to 0} F_{\Phi^\delta}(u_\delta)\le  \limsup_{\delta\to 0} F_{\Phi^\delta}(u) =F_\Phi(u),
			\end{aligned}
		\end{equation}
		where in the last step we used the Monotone Convergence Theorem. 
		Hence, by the uniqueness of the minimizer for $F_\Phi$, we deduce that $u=v.$ 
		
		\smallskip 
		
		\noindent 3. \textit{Strong convergence of $\{\nabla u_\delta\}_\delta$ in $L^{p}$}: The conclusion of the previous step implies that all the inequalities in \eqref{eq:comparing Fv with Fu} are equalities, hence $\lim_{\delta \to 0} F_{\Phi^\delta}(u_\delta)=F_\Phi(u).$ Applying Lemma \ref{lem:effective monot}  we get 
		\begin{equation}\label{eq:monotonicity quant trick p>2}
			\begin{aligned}
				\frac{{1+\lambda}}9  &\int_\Omega   \Big( \inf_{\big[\frac{|\nabla u_\delta|\vee |\nabla u|}3,|\nabla u_\delta|\vee |\nabla u|\big]} \Psi\Big)  |\nabla u_\delta-\nabla u|^2\d \mm \\
				&\le \int_\Omega \Big( \Phi(|\nabla u_\delta|)- \Phi(|\nabla u|)-\Psi(|\nabla u|)\la \nabla u,\nabla (u_{\delta}-u)\ra \Big)\d \mm\\
				&=\int_\Omega \Big( \Phi_{\delta}(|\nabla u_\delta|)- \Phi(|\nabla u|)+f(u_\delta-u) \Big) \d \mm\\
				&=F_{ \Phi}(u_\delta)-F_{ \Phi}(u)\le F_{\Phi^\delta}(u_\delta)-F_\Phi(u)\to 0, 
			\end{aligned}
		\end{equation}
		{using the Euler--Lagrange equation for $u$ (\textit{cf.}~Prop.\  \ref{prop:EL}) in the third line.}
		From  Lemma \ref{lem:lp lemma}, since $|\nabla u_\delta-\nabla u|\le 2(|\nabla u_\delta|\vee |\nabla u|)$, {we get}  $|\nabla u_\delta-\nabla u|\to 0$ in $L^p(\Omega)$; {therein, we choose the sets $A_n = \Omega$ for all $n$}. Strong convergence in $W^{1,p}(\Omega)$ follows by the Poincaré inequality.
	\end{proof}

	\begin{proof}[Proof of Prop.\  \ref{prop:approx of delta by M-delta}]
		In the sequel $\Phi_M,\Psi_M$ are as in the statement and $\phi_M$ is as in Lemma \ref{lem:prop phiM}.
		The minimizers $u_M,u$ exist unique by Prop.\  \ref{prop:existence variational}, thanks to \ref{it:PhiM conv}, \ref{it:PhiM adm} and \ref{it:PhiM coercive} in Lemma \ref{lem:prop phiM}.

		\smallskip 
		
		\noindent 1. \textit{Uniform $W^{1,p}$-bounds for $u_M$:} By minimality, for all $M$ there holds 
		\begin{equation}\label{eq:unif est F_Phi_M}
			F_{\Phi_M}(u_M)\le F_{\Phi_M}(g)\le \int_\Omega \Big( C(|\nabla g|^p+|\nabla g|^2 {+1}) + fg \Big) \d \mm<\infty,
		\end{equation}
		where we used \ref{it:PhiM bound} in Lemma \ref{lem:prop phiM}.
		Hence by arguing as per Step 1 of the proof of Prop.\  \ref{prop:approx W1p Fphi by Fphidelta} and using the uniform coercivity estimate of $\Phi_M$ given in \ref{it:PhiM coercive} of Lemma \ref{lem:prop phiM}, we obtain that the sequence $\{u_M\}_M$ is uniformly bounded in $W^{1,q}(\Omega)$ {(recall $q = 2 \wedge p$)} and thus converges weakly in $W^{1,q}(\Omega)$, up to a subsequence which we do not relabel, to some $v\in g+W^{1,q}_0(\Omega).$

		\smallskip 
		
		\noindent 2. \textit{Weak convergence:} {Our objective in this step is to show first $v \in g+W^{1,p}_0(\Omega)$, and second 
			\begin{equation}\label{eq:conv of inf for PhiM}
				F_\Phi(v)=F_\Phi(u)= \lim_{M\to +\infty} F_{\Phi_M}(u_M).
		\end{equation}}
		This would show that $u=v$, by uniqueness of the minimizer {in the class $g+W^{1,p}_0(\Omega)$---we emphasise that, at this stage, we only know that $v \in g+W^{1,q}_0(\Omega)$}. From \ref{it:PhiM monotone} for all $m\le M$ we have
		\begin{align*}
			\int_\Omega \Big( \phi_m(|\nabla v|)+fv \Big) \d \mm &\le \liminf_M \,\int_\Omega \Big( \phi_m(|\nabla u_M|)+fu_M \Big) \d \mm \\
			&\le \liminf_M\,\int_\Omega \Big( \phi_M(|\nabla u_M|)+fu_M \Big) \d \mm \le \liminf_M \,F_{\Phi_M}(u_M),
		\end{align*}
		where in the first inequality we used Lemma \ref{lem:lsc energy} and the convexity of $\phi_M.$ Sending $m\to +\infty$ and applying the Monotone Convergence Theorem (recall that $\phi_M\uparrow \Phi$) we obtain that $F_\Phi(v)\le  \liminf_M F_{\Phi_M}(u_M)$. In particular $|\nabla v|\in L^p(\Omega)$ {due to the coercivity of $\Phi$} (see Lemma \ref{lem:psi and phi}). Note that we do not know yet that $v \in L^p(\Omega)$. From $|\nabla v| \in L^p(\Omega)$ we have $-k\vee v\wedge k \in W^{1,p}(\Omega)$ for all $k\in \nn$. However, since $v-g\in W^{1,2}_0(\Omega)\subset \widehat W^{1,2}_0(\Omega)$ we have that $v-g=0$ $\mm$-a.e.\ in $\X\setminus \Omega$.  Therefore $(-k)\vee (v-g)\wedge k\in \widehat W^{1,p}_0(\Omega)=W^{1,p}_0(\Omega)$ and  by the Poincaré inequality \eqref{eq:poincare} we deduce that $(v-g)\in L^p(\Omega)$ and in particular $v\in L^p(\Omega)$ and so $v\in g+W^{1,p}_0(\Omega)$ as well. 
		Moreover for all $f \in \LIP_{c}(\Omega)$ we have by minimality
		$$\limsup_M \,F_{\Phi_M}(u_M)\le \limsup_M\, F_{\Phi_M}(f+g)=F_\Phi(f+g),$$
		where in the last step we used the Dominated Convergence Theorem using \ref{it:PhiM bound} in Lemma \ref{lem:prop phiM}. By arbitrariness of $f$ we obtain 
		$\limsup_M \,F_{\Phi_M}(u_M)\le F_\Phi(u)\le F_\Phi(v)$, which shows \eqref{eq:conv of inf for PhiM}.

		\smallskip 
		
		\noindent 3. \textit{Strong convergence:} {Our objective in this step is to show $u_M \to u$ strongly in $W^{1,q}(\Omega)$. The previous steps showed that $u_M \rightharpoonup u$ weakly in $W^{1,q}(\Omega)$ with  $u \in g+W^{1,p}_0(\Omega)$. Due to the density of $\LIP_{c}(\Omega)$ in $W^{1,p}_0(\Omega)$, there exists a sequence $\{w_n\}_n \subset g + \LIP_{c}(\Omega)$ converging strongly to $u$ in $W^{1,p}(\Omega)$. 	For each fixed $n$, we have $F_{\Phi_{M}}(w_n) \to F_\Phi(w_n)$ as $M\to \infty$ by the Dominated Convergence Theorem as above. 
			Furthermore, we have $F_{\Phi}(w_n)\to F_{\Phi}(u)$ as $n\to\infty$ by strong convergence. By diagonal argument we can extract a subsequence $\{M_n\}_n$ such that 
			$ F_{\Phi_{M_n}}(w_n) \to  F_{\Phi}(u)$ as $n\to +\infty.$}
		Applying Lemma \ref{lem:effective monot} to $\Psi_M(t)\coloneqq \Psi(t\wedge M_n)$ (which satisfies again \eqref{eq:the usual psi'} for some $\lambda\le 0$), {and therein setting $v=\nabla w_n$ and $w=\nabla u_{M_n}$},  we obtain 
		\begin{equation*}
			\begin{aligned}
				\frac{1+\lambda}{9} &\int_\Omega   \Big( \inf_{\big[\frac{|\nabla u_{M_n}|\vee |\nabla w_n|}3,|\nabla u_{M_n}|\vee |\nabla w_n|\big]} \Psi_M\Big)   |\nabla u_{M_n}-\nabla w_n|^2\d \mm \\
				&\le \int_\Omega \Big( \Phi_{M}(|\nabla w_n|)-\Phi_{M_n}(|\nabla u_{M_n}|)-\Psi_{M_n}(|\nabla u_{M_n}|)\la \nabla u_{M_n},\nabla (w_n -u_{M_n})\ra \Big) \d \mm\\
				&=\int_\Omega \Big( \Phi_{M} (|\nabla w_n|)-\Phi_{M_n}(|\nabla u_{M_n}|)+f(w_n-u_{M_n}) \Big) \d \mm\\
				&=F_{\Phi_{M_n}}(w_n)-F_{\Phi_{M_n}}(u_{M_n})\to F_{\Phi}(u)-F_\Phi(u)=0,
			\end{aligned}
		\end{equation*}
		{where we used the Euler--Lagrange equation for $u_{M_n}$ (\textit{cf.}~Prop.\  \ref{prop:EL}) in the third line.} By Lemma \ref{lem:lp lemma}, using that $\Psi\ge c>0$ for $p\ge 2$ and  the triangle inequality, we get  $|\nabla u_{M_n}-\nabla u|\to 0$ in $L^q(\Omega)$.
		The strong $W^{1,q}(\Omega)$ convergence then follows by the Poincaré inequality.

		\smallskip
		
		\noindent 4. \textit{Uniform gradient estimate}: 
		{Finally}, we prove the estimate \eqref{eq:uniform bound for uM}. Recall from \eqref{eq:unif est F_Phi_M} 
		\[
		\int_\Omega \Phi_M(|\nabla u_M|)   \d \mm \le \int_\Omega \Big( C(|\nabla g|^p+|\nabla g|^2 {+1} ) + f(g-u_M) \Big) \d \mm.
		\]
		In turn, using Young's inequality, we obtain from the previous inequality that, for all $\eps>0$, 
		\begin{align*}
			\int_\Omega \Big( \eps^{-1}f^{2\vee p'}\! + &\eps(u_M\!-\!g)^q\!+\! C(|\nabla g|^p+|\nabla g|^2 {+1}) \Big) \d \mm \!   \geq \int_\Omega \Phi_M(|\nabla u_M|) \d \mm  \\ 
			&\ge \frac{c}{2} \int_\Omega \Psi_M(|\nabla u_M|)|\nabla u_M|^2 \d \mm  +   \frac{c}{2}\int_\Omega (|\nabla u_M|^q\!-\!1) \d \mm, 
		\end{align*}
		where we used \ref{it:PhiM coercive} and \ref{it:PhiM super coercive} of Lemma \ref{lem:prop phiM} to obtain the final line, and where $c>0$ is a constant depending only on $\Psi$, but not $M.$ Applying Poincaré's inequality on the left side to $u_M-g$ and absorbing the term with $|\nabla u_M|$ into the right side for $\eps$ small (depending on $p$ and $\delta$) we obtain
		\begin{equation}\label{eq:with square bound in appendix prop}
			\int_\Omega \Psi_M(|\nabla u_M|)|\nabla u_M|^2 \d \mm \le  \tilde C\int_\Omega \Big( |\nabla g|^2+|\nabla g|^p+g^2+f^{2\vee p'}+1 \Big) \d \mm ,
		\end{equation}
		with $\tilde C$ depending only on $\Psi,\Omega$. We get \eqref{eq:uniform bound for uM} since 
		$\Psi(t\wedge M)t\le\Psi_M(t)t^2 +\max_{[0,1]}\Psi$, $\forall t\ge 0$. 
	\end{proof}

	\subsection{Continuity of $\div(\Psi(\gradu)\nabla u)$ operator in $L^{p-1}$}\label{app:continuity of p Laplace wrt f}

	Here we prove Prop.\  \ref{prop:continuity wrt f}, stated in \S \ref{sec:p harmonic} and used to prove Theorem \ref{thm:p-delta regularity}. {The argument is very similar to \cite[Prop.\  3.4]{ivan1}.}

	\begin{proof}[Proof of Proposition  \ref{prop:continuity wrt f}]
		It is sufficient to show that there exists $q>1$ such that
		\begin{equation}\label{eq:gradient to zero in meas}
			\begin{split}
				&|\nabla u_n-\nabla u|\to 0 \quad \text{ in $\mea$-measure}, \qquad \sup_n \||\nabla u_n|^{p-1}\|_{L^q(\mm)}  <+\infty.
			\end{split}
		\end{equation}
		From \eqref{eq:gradient to zero in meas} we conclude by standard arguments (see \textit{e.g.}~\cite[Lemma 8.2]{HajlaszKoskela00}). Below we show \eqref{eq:gradient to zero in meas}. By $C$ we denote a constant depending only on $\Omega,\Psi$ and $p$, possibly changing from line to line. 
		
		\smallskip 
		
		\noindent 1. \textit{Truncated estimate on $|\nabla u_n|^p$}: Our goal in this step is to obtain the estimate 
		\begin{equation}\label{eq:integral gradient bound u}
			\int_{\{|u_n-g|\le k\}} |\nabla u_n|^p\d \mm \le C\Vert |\nabla g|+1\Vert_{L^p(\Omega)}^p  + Ck^{2- {p\wedge 2}} \|f\|_{L^2(\Omega)}^{2\wedge(\frac{p}{p-1})}, 
		\end{equation}
		for all $k>0$. The underlying idea is to test with $u_n-g$ in the equation $\div(\Psi(|\nabla u_n|\nabla u_n)=f_n$. However, since we do not know \emph{a priori} that $f_n \in L^{p'}(\Omega)$ for $p \in (1,2)$, this cannot be done directly. In order to overcome this, for all $k>0$, we define the {Stampacchia-type} truncation $F_k(t)\coloneqq (-k)\vee t \wedge k$ and insert $F_k(u_n-g) \in W^{1,p}_0\cap L^\infty(\Omega)$ into the weak formulation of $\div(\Psi(|\nabla u_n|\nabla u_n)=f_n$ (\textit{cf.}~Remark \ref{rmk:sobolev test}), noting that, by locality, 
		\begin{equation}\label{eq:gradient of truncation}
			\nabla F_k(u_n-g) = (\nabla u_n - \nabla g)\mathds{1}_{\{|u_n-g|\leq k\}}.
		\end{equation} 
		{From the growth assumption \eqref{eq:psi growth} and the non-negativity of $\Psi$ we deduce that  } $t^2\Psi(t)\ge c(t^p-1)$ for some $c>0$. { Again from \eqref{eq:psi growth} and Lemma \ref{lem:psi basic}-i) we deduce $t\Psi(t)\le \tilde c(t^{p-1}+1)$  for some $\tilde c>0$.}  Using both of these bounds and  testing with $F_k(u_n-g)$  we obtain
		\begin{equation*}
			\begin{split}
				\int_{\{|u_n-g|\le k\}}& (c|\nabla u_n|^p-1)\d \mm\le \tilde c\int_{\{|u_n-g|\le k\}} \!\!\!\!\!\!\!\!\!\ {(|\nabla u_n|^{p-1}+1)}|\nabla g| \d\mm + \int_\Omega |F_k(u_n\!-\!g)| |f_n|  \d \mm\\
				&\le \frac c2\!\!\int_{\{|u_n-g|\le k\}} \!\!\!\!\!\!\!\!\!\! {(|\nabla u_n|^{p}+1)} \d \mm\!+\! C_p \Vert \nabla g\Vert_{L^p(\Omega)}^p \!+\! \|F_k(u_n\!-\!g) \|_{L^2(\Omega)} \|f_n\|_{L^2(\Omega)},
			\end{split}
		\end{equation*}
		{where $C_p>0$ depends only on $p,c$ and $\tilde c.$}  Hence
		\begin{equation}\label{eq:ffft}
			\begin{split}
				\int_{\{|u_n-g|\le k\}} |\nabla u_n|^p\d \mm&\le C\bigg( \Vert |\nabla g|+1\Vert_{L^p(\Omega)}^p +k^{1-\frac {p\wedge 2}2} \|F_k(u_n-g) \|_{L^p(\Omega)}^\frac {p\wedge 2}2  \|f_n\|_{L^2(\Omega)}\bigg),
			\end{split}
		\end{equation}
        where we used the H\"older inequality for $p\ge 2$ and   $|F_k(u_n-g)|^2 \leq k^{2-p}|F_k(u_n-g)|^p$ if $p\le 2$. 
		Since $f_n \to f$ strongly in $L^2(\Omega)$, the final term on the right-hand side is bounded independently of $n$ by $C\Vert f \Vert_{L^2(\Omega)}$.	We apply the Poincaré inequality, recalling that $F_k(u_n - g) \in W^{1,p}_0(\Omega)$ and the relation \eqref{eq:gradient of truncation}, and then apply Young's inequality to get 
		\begin{align*}
			k^{1-\frac {p\wedge 2}2} \|  F_k(u_n-g) \|_{L^p(\Omega)}^\frac {p\wedge 2}2 \|f\|_{L^2(\Omega)}\le \delta \int_{\{|u_n-g|\le k\}} |\nabla (u_n-g)|^p\d \mm+ C\delta^{-1}k^{2- {p\wedge 2}} \|f\|_{L^2(\Omega)}^{2\wedge(\frac{p}{p-1})}, 
		\end{align*}
		where $\delta>0$ is arbitrary. By taking $\delta$ sufficiently small and using the triangle inequality, we absorb the term with $|\nabla u_n|$ into the left-hand side of \eqref{eq:ffft} and obtain \eqref{eq:integral gradient bound u}. 
		
		\smallskip 
		
		\noindent 2. \textit{Estimate on $\mm(\{|\nabla u_n|>k\})$}: By \eqref{eq:integral gradient bound u} and Poincaré inequality applied to $F_k(u_n-g)$, 
		\begin{equation}\label{eq:integral lp bounds}
			\int_\Omega | F_k(u_n-g)|^p\d \mm \le  C\Vert |\nabla g|  {+1}\Vert_{L^p(\Omega)}^p  + Ck^{2-p\wedge 2}\|f\|_{L^2(\Omega)}^{2\wedge(\frac{p}{p-1})}.
		\end{equation}
		Combining  \eqref{eq:integral lp bounds} and the Markov inequality, noting that $|F_k| \leq k$, we have 
		\begin{equation}\label{eq:markov unif bounds}
			\begin{aligned}
				\mm(|u_n-g| \geq k)&= \mm(|F_k(u_n-g)| = k) = \mm(|F_k(u_n-g)| \geq k) \\ 
				&\leq Ck^{-p}\Vert \nabla g + 1 \Vert^p_{L^p(\Omega)} + Ck^{(2-p)-p\wedge 2}\|f\|_{L^2(\Omega)}^{2\wedge(\frac{p}{p-1})}.
			\end{aligned}
		\end{equation}
		Combining  \eqref{eq:integral gradient bound u} with \eqref{eq:markov unif bounds} and using again the Markov inequality we get for every $k>0$ 
		\begin{equation}\label{eq:measure of gradun}
			\begin{aligned}
				\mea(\{|\nabla u_n|>k \})&\le  \mea(\{|\nabla u_n| \mathds{1}_{\{|u_n-g| \leq k\}}>k \}) +\mea(\{|u_n-g| \geq k\}) \\ 
				&\le C\frac{1}{k^p}\Vert |\nabla g| {+1} \Vert_{L^p(\Omega)}^p  +C \frac 1{k^{p-2+p\wedge 2}} \|f\|_{L^2(\Omega)}^{2\wedge(\frac{p}{p-1})}.
			\end{aligned}
		\end{equation}

		\smallskip 
		
		\noindent 3. \textit{Conclusion}: This shows the second part of \eqref{eq:gradient to zero in meas}; indeed by Cavalieri's formula, 
		\begin{align*}
			\int_{\Omega} {(|\nabla u_n|^{p-1})^q}\d \mm &=\int_0^\infty qt^{q-1}\mm(\{|\nabla u_n|^{p-1}>t\})\d t \\
			&\le
			q\mm(\Omega)+qC\int_1^\infty  t^{q-1}\big({t^{-\frac{p}{p-1}} + t^{-\frac{(p-2)+p\wedge 2}{p-1}}    }\big)\d t,
		\end{align*}
where $C$ is independent of $n$, and the last term is finite	if {$q<\min\big\{\frac{(p-2)+p\wedge 2}{p-1},\frac{p}{p-1}\big\}$} (note that $\frac{(p-2)+p\wedge 2}{p-1}>1$).

		For the first part of \eqref{eq:gradient to zero in meas} we use  $F_k(u_n - u) \in W^{1,p}_0(\Omega)$, with $k\in\nn$ arbitrary, as test function in both the equations in \eqref{eq:two equaions f,f_n} and subtract the two identities  to get, {for all fixed $k$}, 
		\[
		\int_{\{|u-u_n|<k\}} \la\Psi(|\nabla u_n|)\nabla u_n-\Psi(|\nabla u|)\nabla u , \nabla (u_n-u)\ra\le k \|f-f_n\|_{L^1(\Omega)}\to 0 \quad {\text{as } n \to \infty}.
		\]
		Applying the first monotonicity inequality of Lemma \ref{lem:quant mon Psi}, we obtain
		\begin{equation}
			{\lim_{n\to\infty}} \int_{A_n^k} \Psi\Big(\frac{|\nabla u_n|\vee |\nabla u|}{2} \Big) |\nabla u-\nabla u_n|^2\d \mm = 0,
		\end{equation}
		where $A_n^k\coloneqq \{|u-u_n|<k,\, |\nabla u_n|<k,\}.$   {We actually obtain that the limit holds integrating in $\{|u-u_n|<k\}$, but it is convenient to restrict to $A_n^k$. In this way $\||\nabla u_n|\|_{L^p(A_n^k)}$ is  uniformly bounded in $n$ and we can thus apply }
		the $L^p$-convergence Lemma \ref{lem:lp lemma} to obtain that $\|\nabla (u- u_n)\|_{L^p(A_n^k)}\to 0.$ By \eqref{eq:markov unif bounds} and \eqref{eq:measure of gradun} for all $\eps>0$ it holds $\mm(A_n^k)\le \eps$ for all $k$ big enough, uniformly in $n$. This implies $|\nabla (u- u_n)|\to 0$ in measure and proves \eqref{eq:gradient to zero in meas}. 			
	\end{proof}

	\appendix

	\section{Properties of the functions \texorpdfstring{$\Psi $ and $ \Phi$}{Phi and Psi}}\label{sec:appendix}

	{In the following statement we collect useful elementary properties of the function $\Psi$.}
	\begin{lemma}\label{lem:psi basic}
		Let $\Psi \in {\sf AC}_\loc(0,\infty)$ be positive, satisfying \eqref{eq:psi' condition} for constants $\lambda,\Lambda$. Then:
		\begin{enumerate}[label=\roman*)]
			\item $\Psi(1) \min  \{t^{\Lambda},t^\lambda\}\le \Psi(t)\le \Psi(1) \max  \{t^{\Lambda},t^\lambda\}$, for all $t>0,$
			\item  $ \Psi(t)\le (t/s)^\Lambda  \Psi(s)$ for all $t\ge s>0$,
			\item $\sqrt{t^2+1}\Psi(\sqrt{t^2+1})\le 2^{\Lambda +1} \big (t\Psi(t)+\Psi(1)\big)$
		\end{enumerate}
	\end{lemma}
	\begin{proof}
		i) and ii) follow easily integrating $\frac{\Psi'}{\Psi}$ and using \eqref{eq:psi' condition}. For iii) note that by ii) we have 
		\[
		\sqrt{t^2+1}\Psi(\sqrt{t^2+1}) \le \Big(\frac{\sqrt{t^2+1}}{t}\Big)^{\Lambda+1} t\Psi(t) \le 2^{\Lambda+1} t\Psi(t), \qquad \forall t \geq 1, 
		\]
		having used that $\sqrt{t^2+1}\le 2t$ for all $t\ge 1.$ By i) and since $t\Psi(t)$ is non-decreasing we have $\sqrt{t^2+1}\Psi(\sqrt{t^2+1})\le \sqrt 2\Psi(\sqrt 2)\le(\sqrt 2)^{\Lambda +1}\Psi(1)$ for all $t \in (0,1)$. 
	\end{proof}

	{We now prove that condition \eqref{eq:psi' condition} is preserved by taking $M$-truncation or $\delta$-regularization.}
	\begin{lemma}\label{lem:regularity of truncated psi}
		Let $\Psi \in {\sf AC}_\loc(0,\infty)$ satisfy \eqref{eq:psi' condition} with parameters $\lambda\le 0$ and $\Lambda$. Then for all constants $M>0,\delta>0$ both the functions 
		$
			\Psi_M(t)\coloneqq \Psi\big(t\wedge M)$ and $ \Psi^{\delta}(t)\coloneqq \Psi\big(\sqrt{t^2+\delta}\,\big)$
		belong to $\AC_\loc(0,\infty)$ and satisfy again \eqref{eq:psi' condition} with the same parameters $\lambda,\Lambda.$
	\end{lemma}
	\begin{proof}
		Clearly again $\Psi_M,\Psi^{\delta}\in {\sf AC}_\loc(0,\infty)$. The fact that $\Psi_M$ satisfies again \eqref{eq:psi' condition} with  $\lambda,\Lambda$ follows because $\Psi_M'=\mathds{1}_{\{t\le M\}} \Psi'$ and $\lambda\le 0$.  On the other hand
		\[
		t(\Psi^{\delta})'(t)= \frac{t^2}{\sqrt{t^2+\delta}}  \Psi'\big(\sqrt{t^2+\delta}\,\big)
		{\implies  \frac{t(\Psi^{\delta})'(t)}{\Psi^\delta(t)} = \frac{t^2}{t^2+\delta} \cdot \frac{(\sqrt{t^2+\delta})  \Psi'(\sqrt{t^2+\delta})}{\Psi(\sqrt{t^2+\delta})}  }     \]
		hence by \eqref{eq:psi' condition} for $\Psi$, we get 
		$\lambda\le \lambda\Big( \frac{t^2}{t^2+\delta} \Big)  \leq  \frac{t(\Psi^{\delta})'(t)}{\Psi^\delta(t)} \leq \Lambda \Big( \frac{t^2}{t^2+\delta} \Big) \leq \Lambda$, 
		where we used that $\lambda\le 0$ and $\frac{t^2}{t^2+\delta} \leq 1$. This shows that $\Psi^{\delta}$ satisfies again \eqref{eq:psi' condition} with  $\lambda,\Lambda.$ 
	\end{proof}

	{The following technical lemma   plays a role in the proof of Prop.\  \ref{prop:approx of delta by M-delta}.}
	\begin{lemma}[Properties of $\Phi_M$]\label{lem:prop phiM} Let $\Psi \in {\sf AC}_\loc(0,\infty)$ be positive, satisfying \eqref{eq:psi' condition} and the $p$-growth condition \eqref{eq:psi growth} for  $p\in(1,\infty)$.  For all $M\ge 1$ define the functions 
		\begin{equation}\label{eq:PhiM def integral}
			\Phi_M(t)\coloneqq \int_0^t s \Psi(s\wedge M)\d s;  \quad \phi_M(t)\coloneqq \int_0^t (s\wedge M) \Psi(s\wedge M)\d s. 
		\end{equation}
		Then, for constants $c,C>0$ depending only on $\Psi$ and $p$, it holds$:$
		\begin{enumerate}[label=\roman*)]
			\item\label{it:PhiM conv} $\Phi_M$ (resp.\ $\phi_M$) is monotone and strictly convex (resp.\ convex) for all $M>0$,
			\item\label{it:PhiM adm} $\Phi_M$ is 2-admissible in the sense of Definition \ref{def:admissible Phi},
			\item\label{it:PhiM coercive} $ c\,(t^{2\wedge p} -1) \le \Phi_M(t)$ for all $ t\ge0,M\ge 1,$
			\item\label{it:PhiM bound} $\Phi_M(t)\le C(t^p+t^2+1)$ holds for all $ t\ge0$  and all $M\ge 1,$
			\item\label{it:PhiM monotone} for all $0<m\le M$ it holds
			$\phi_m(t)\le \phi_{M}(t)\le \Phi_M(t),$ {for all $t\ge 0$,} 
			\item\label{it:PhiM super coercive} $c\,t^2\, \Psi(t\wedge M)\le \Phi_M(t)$ holds for all $t\ge 0$,
			\item\label{it:PhiM limit} $\phi_M(t)\uparrow \Phi(t)$ pointwise as $M\to +\infty$, {where $\Phi(t) \coloneqq \int_0^t s \Psi(s) \d s$}.
		\end{enumerate}
	\end{lemma}
	\begin{proof}
		{ We first note that by  Lemma \ref{lem:psi basic}-i) we have that $t\Psi(t) $  is bounded in $(0,C]$ for all $C>0$, hence $\Phi_M,\phi_M$ are well defined, finite and continuous in $[0,\infty)$.}
		
		\noindent 1. \textit{Proof of  \ref{it:PhiM conv}}: By \eqref{eq:psi' condition} it follows that $(t\Psi(t))'\ge (1+\lambda)\Psi(t)>0$, hence $t\Psi(t)$ is strictly increasing. {A similar computation} shows that $s \Psi(s\wedge M)$ is strictly increasing and that $(s\ww M) \Psi(s\wedge M)$ is increasing, which {respectively imply the convexity of $\Phi_M$ and $\varphi_M$}.

		\noindent 2. \textit{Proof of    \ref{it:PhiM adm}}: By Lemma \ref{lem:psi basic} we have $s\Psi(s)\le \Psi(1) \max\{s^{\lambda+1},s^{\Lambda+1}\}$ for all $t>0.$  Hence $s\Psi(s\wedge M)\le C_M(s+1)$ for all $t>0$ for some constant $C_M,$   which {(by Young's inequality)} gives $\Phi_M(t)\le  C_M(t^2+1)$ up to increasing $C_M,$ as desired.

		\noindent 3. \textit{Proof of    \ref{it:PhiM coercive} and \ref{it:PhiM bound}}: {Since $\Psi$ is positive, it is clear that the lower bound of part iii) is satisfied over the interval $t \in [0,1]$ for some $c$. Using the continuity of {$\Phi_M$}, the same is true for part iv). Thus, we may restrict our attention to } $t\ge 1$. By the $p$-growth condition \eqref{eq:psi growth}, 
		\[
		\Phi_M(t)\!\ge\! \int_1^{t} \!\!\!s \Psi(s\ww M)\d s \ge  \!\nu^{-1}\!\!\int_1^{t}\!\! s (t\ww M)^{p-2} \d s \!\ge\! \nu^{-1}\min\{1,(t\ww M)^{p-2}\} \!\int_1^t \!\! s \d s \ge \frac{\nu^{-1}}2{(t^{p\ww 2}\!-\!1)},
		\]
		{and, similarly,} $\Phi_M(t)\le  c_\Psi+\nu \max\{1,(t\ww M)^{p-2}\}\frac{t^2}2\le  c_\Psi+\nu t^{p\vee 2},$ where $c_\Psi\coloneqq \int_0^1 s \Psi(s)\d s<\infty$, { since $s\Psi(s)$ is bounded in $(0,1)$}.

		\noindent 4. \textit{Proof of \ref{it:PhiM monotone}}: Immediate because $(s\ww M)\Psi(s\wedge M)$ is increasing in $M$ and since $\Psi\ge 0.$

		\noindent 5. \textit{Proof of \ref{it:PhiM super coercive}}:  Recall that $ s\Psi(s\wedge M)$  is non-decreasing {and $\Psi$ is positive}, hence
		\[
		\Phi_M(t) \ge \int_{t/2}^t s \Psi\left(t\wedge M\right)\d s\ge \frac{t^2}{4} \Psi\left(\frac{t}{2}\wedge M\right) \ge 2^{-\Lambda}\frac{t^2}{4} \Psi(t\wedge M), 
		\]
		where we used {part ii) of} Lemma \ref{lem:psi basic} applied to $\Psi(s\wedge M)$.

		\noindent 6. \textit{Proof of \ref{it:PhiM limit}}: {This is clear from the Monotone Convergence Theorem}.
	\end{proof}

	\begin{lemma}[Effective convexity]\label{lem:quant mon Psi}
		Let $\Psi:[0,\infty)\to [0,\infty)$ be in $\AC_\loc(0,\infty)$ and such that
		\begin{equation}\label{eq:the usual psi'}
			{t\Psi'(t)}\ge  \lambda{\Psi(t)}, \quad \text{ for a.e.\ $t\ge 0$}
		\end{equation}
		for some constant $-1<\lambda\le 0$. Then, by setting $\Phi(T)\coloneqq \int_0^T r \Psi(r) \d r$, it holds 
		\begin{align}
			&(t\Psi(t)-s\Psi(s))(t-s)\ge \frac{1+\lambda}4 |t-s|^2\Psi   \bigg(\frac{t\vee s}2\bigg), \quad \text{ for all $s,t\ge 0$,} \label{eq:quant mon Psi} \\ 
            & \Phi(T)-\Phi(S)-\Phi'(S)(T-S)\ge \frac{1+\lambda}{9} |S-T|^2 \big( \inf_{\big(\frac{S\vee T}3,S\vee T\big)}\Psi\big) , \quad \text{ for all $S,T\ge 0$.} \label{eq:quant mon Psi 2}
		\end{align}		
	\end{lemma}
	\begin{proof}
		{
			We begin with \eqref{eq:quant mon Psi}. Note  that  $r\Psi(r)$ is non-decreasing by  assumption \eqref{eq:the usual psi'}. If $s=t=0$ there is nothing to prove, while if $s=0$ the statement follows immediately because by monotonicity {of $r \Psi(r)$ we have} $\Psi(t)\ge \frac12\Psi(t/2)$. By symmetry we can assume   {without loss of generality} that $t\ge s>0$.
			To obtain the estimate, we begin by writing 
			\[
			t\Psi(t)-s\Psi(s)=\int_s^t (r\Psi(r))' \d r \ge (1+\lambda)\int_s^t\Psi(r)\d r \ge (1+\lambda)\int_{s\vee \frac{ t}{2}}^t\Psi(r) \d r.
			\]
			By monotonicity {of $r \Psi(r)$} we have $\Psi(r)\ge \frac{t}{2r}\Psi\big(\frac{t}{2}\big)\ge \frac12\Psi\big(\frac{ t}{2}\big)$ for all $r \in (s\vee \frac{ t}{2},t)$ . Hence, 
			\[
			t\Psi(t)-s\Psi(s)\ge (1+\lambda) \Big(t-s\vee \frac{ t}{2}\Big)\frac12\Psi\big(\frac{ t}{2}\big)\ge \frac {t-s}4\Psi\big(\frac{ t}{2}\big),
			\]
			where we used that  $t-s\vee \frac{ t}{2}\ge \frac{t-s}{2}.$ Multiplying by $(s-t)$ we obtain \eqref{eq:quant mon Psi}.}	
            
            We pass to \eqref{eq:quant mon Psi 2}. We assume $S \le  T$, the other case is similar. By \eqref{eq:the usual psi'} we have  $\Phi''(t)\ge (1+\lambda)\Psi(t)$ for a.e.\ $t\ge 0$.
		By the fundamental theorem of calculus
		\begin{equation}\label{eq:taylor Phi}
			\Phi(T)-\Phi(S)-\Phi'({S})(T-S)=\int_S^T\int_S^r \Phi''(t)\d t \d r \ge (1+\lambda)\int_S^T\int_S^r \Psi(t)\d t \d r. 
		\end{equation}
		We {set $x=S+\frac{1}{3}(T-S)$ and $y = S+\frac{2}{3}(T-S)$}. 
		We have {$(x,y) \subset (S,T)$ and $|x-y|=|x-S|=|T-{y}|=|S-T|/3$. Hence, by the positivity of $\Psi$}, 
		\[
		\int_S^T\int_S^r \Psi(t)\d t \d r\ge \int_y^T\int_x^y \Psi(t)\d t \d r\ge |T-y||y-x| \inf_{(x,y)} \Psi =\frac19|T-S|^2 \inf_{(x,y)} \Psi.
		\]
		The conclusion follows noting that 
		{$x = \frac{T}{3}+\frac{2}{3}S \geq \frac{T}{3} = \frac{T \vee S}{3}$ and $y \leq T = T\vee S$.}
	\end{proof}
	
   {The result below extends the standard one  in $\rr^n$ for $\Psi(t)=t^{p-2}$, see  \textit{e.g.} \cite[Lemma A.0.5]{peral}. }  
	\begin{lemma}[Effective monotonicity]\label{lem:effective monot}
		Let $\Xdm$ be an infinitesimally Hilbertian metric measure space. Let $\Psi,\Phi:[0,\infty)\to [0,\infty)$ be as in Lemma \ref{lem:quant mon Psi}. Then for all $v,w\in L^0(T\X)$, 
		\begin{align}
			\la \Psi(|v|)v-\Psi(|w|)w,v-w\ra &\ge \frac{1+\lambda}4 |v-w|^2\Psi\bigg(\frac{|w|\vee |v|}2\bigg), \quad &&\alme \label{eq:quant mon} \\ 
			\Phi(|v|)-\Phi(|w|)-\Psi(|w|)\la w,v-w\ra &\ge \frac{1+\lambda}{36} |v-w|^2 \Big( \inf_{\big[\frac{|w|\vee |v|}3,|v|\vee |w|\big]} \Psi\Big), \quad &&\alme \label{eq:quant mon 2}
		\end{align}
	\end{lemma}
	\begin{proof}
		{We start with \eqref{eq:quant mon}.  It is sufficient to show its validity $\mm$-a.e.\ in the set $\{|w|\ge |v|\}$ and then argue by symmetry. By simple manipulation we have
			\begin{align*}
				\la \Psi(|v|)v\!-\!\Psi(|w|)w,v\!-\!w\ra \!=\!  (\Psi(|v|)|v|\!-\!\Psi(|w|)|w|)(|v|\!-\!|w|)\!+\!  (\Psi(|v|)\!+\!\Psi(|w|))\big(|v||w|\!-\!\la v, w\ra\big).
			\end{align*}
			Applying \eqref{eq:quant mon Psi} to the first term in the right-hand side and restricting to the set $\{|w|\ge |v|\}$, 
			\begin{align*}
				\la \Psi(|v|)v-\Psi(|w|)w,v-w\ra&\ge  \frac{1+\lambda}4 ||w|-|v||^2\Psi\bigg(\frac{|v|}2\bigg)+  \Psi(|v|)\, \big(|v||w|-\la v, w\ra\big)\\
				&\ge   \frac{1+\lambda}4 |w-v|^2\Psi\bigg(\frac{|v|}2\bigg),
			\end{align*}
			where we used that  $\Psi(|v|)\ge \frac12 \Psi(|v|/2)$, since $s\Psi(s)$ is non-decreasing,} { and \eqref{eq:quant mon} follows.}

		We move to \eqref{eq:quant mon 2}. Expanding and manipulating, {using also that $\Phi'(t)=t\Psi(t)$}, we have 
		\[
		\Phi(|v|)\!-\!\Phi(|w|)-\!\Psi(|w|)\la w,v\!-\!w\ra\!=\!\Phi(|v|)\!-\!\Phi(|w|)\!-\!\Phi'(|w|)(|v|\!-\!|w|)+\Psi(|w|)\,(|v||w|\!-\!\la w,v\ra).
		\]
		Hence, by applying inequality \eqref{eq:quant mon Psi 2} pointwise, we get
		\begin{align*}
			\Phi(|v|)\!-\!\Phi(|w|)\!-\!\Psi(|w|)\la w,v\!-\!w\ra &\ge  \frac{1\!+\!\lambda}{9} (|v|\!-\!|w|)^2 \Big( \inf_{\big[\frac{|w|\vee |v|}3,|v|\vee |w|\big]} \Psi\Big) \!+\! \Psi(|w|)\,(|v||w|\!-\!\la w,v\ra).
		\end{align*}
		If $|w|\le |v|/3$ we claim that $(|v|\!-\!|w|)^2\ge \frac{|v-w|^2}4 $. This {is enough to prove \eqref{eq:quant mon 2}} in this case since the last term is non-negative. If $|v|=|w|=0$ it is clear, otherwise setting $t\coloneqq |w|/|v|\in(0,1/3]$
		\[
		(|v|-|w|)^{-2}|v-w|^2\le (|v|-|w|)^{-2}(|v|+|w|)^2={(1-t)^{-2}(1+t)^2}\le 4.
		\]
		If $|w|\ge |v|/3$, then  $\Psi(|w|)\ge \inf_{\big[\frac{|w|\vee |v|}3,|v|\vee |w|\big]} \Psi$ and we conclude again, {this time using the final term on the right side to cancel the $-\frac{2}{9}(1\!+\!\lambda)|v||w|$ from the first term on the right side}.
	\end{proof}
	
	{The next convergence result is  used in all the proofs of \S \ref{sec:proof of propositions}.}
	\begin{lemma}\label{lem:lp lemma}
		Fix $p\in(1,\infty).$
		Let $\Psi:[0,\infty)\to (0,\infty)$ be  continuous in $(0,\infty)$ and such that 
		$$\Psi(t)\ge ct^{p-2}, \quad \text{ for all $t\ge 1$} \qquad { (\text{for some given } c>0).  }$$ 
		Let $(\mathcal X,\mathcal A, \mu)$ be a finite measure space, $\{f_n\}_n, \{g_n\}_n \subset L^p(\mu)$ Borel functions and  ${ A_n}\subset \mathcal X$ sets, such that $g_n\ge |f_n|$ in ${ A_n}$, and $\sup_n\|g_n\|_{L^p({ A_n})}<\infty$. {Define, for all $\delta \in (0,1)$, $n\in\mathbb{N}$, $x \in \mathcal{X}$, the {{closed} interval}  $\omega_{\delta,n}(x) := [\delta g_n(x) , \delta^{-1}g_n(x)]$. Suppose that}, for some $\delta \in (0,1)$, there holds 
		\[
		\lim_{n\to +\infty}\int_{{ A_n}} \Big(\inf_{\omega_{\delta,n}(x)} \Psi \Big)|f_n(x)|^2\d \mu(x)=0. 
		\]
		Then $\|f_n\|_{L^p({ A_n})}\to 0$.
	\end{lemma}
	\begin{proof}
		By continuity and positivity  for all $\eps>0$ there exists $c_\eps>0$ such that $\Psi(t)\ge c_\eps t^{p-2}$ for all $t\ge \delta \eps$; {for instance, we set $c_\varepsilon = \min\big\{c,\frac{\min_{[\varepsilon \delta , 1]} \Psi}{\max\{1,(\varepsilon \delta)^{p-2}\}}\big\}$}. Note that $\delta$ is fixed. 
		In particular 
		\begin{equation}\label{eq:inf psi ineq}
			\inf_{\omega_{\delta,n}(x)} \Psi\ge c_\eps g_n(x)^{p-2}\, (\delta^{p-2}\wedge \delta^{2-p}), \quad \forall \,x \in \{g_n\ge \eps\}, 
		\end{equation}
		whence the quantity $\inf_{\omega_{\delta,n}(x)} \Psi$ effectively controls $|f_n|^{p-2}$. In detail, we have:
		\begin{align*}
			\int_{{ A_n}} |f_n|^p\d \mu& \le \int_{\{|f_n|\le \eps\}} |f_n|^p\d \mu  +\int_{\{|f_n|\ge \eps  \}\cap { A_n}} |f_n|^p\d \mu\le\eps^p \mu(\mathcal{X})  +\int_{\{|f_n|\ge \eps  \}\cap { A_n}} |f_n|^p\d \mu.
		\end{align*}
		To conclude it is enough to show that the final term on the right-hand side vanishes as $n\to \infty$; indeed, one may take the limsup as $n\to\infty$ followed by the limit as $\varepsilon \to 0$ (noting that the left-hand side is independent of $\varepsilon$) to deduce the result. We split the cases $p<2$ and $p\ge 2.$ For $p<2$ by applying H\"older's inequality and using that $\{|f_n|\ge \eps\}\subset \{g_n\ge \eps\}$, we get 
		\begin{align*}
			\int_{\{|f_n|\ge \eps  \}\cap { A_n}} \!|f_n|^p\d \mu(x)&\le 
			\left(\int_{ A_n} \Big(\inf_{\omega_{\delta,n}(x)} \Psi \Big) |f_n|^2\d \mu\right)^\frac{p}{2} \! \left(\int_{\{g_n\ge \eps  \}\cap { A_n}} \frac1{(\inf_{\omega_{\delta,n}(x)} \Psi)^\frac{p}{2-p}} \d \mu\right)^\frac{2-p}{2}\\
			&\le \left(\int_{ A_n} \Big(\inf_{\omega_{\delta,n}(x)} \Psi \Big) |f_n|^2\d \mu(x)\right)^\frac{p}{2} \left(\int_{{ A_n}} c_\eps^{{}\frac{p}{p-2}}\delta^{-p}g_n^p \d \mu\right)^\frac{2-p}{2}\to 0,
		\end{align*}
		in the limit as $n\to\infty$, where we used \eqref{eq:inf psi ineq} to obtain the final line.
		If instead $p\ge 2$
		\begin{align*}
			\int_{\{|f_n|\ge \eps  \}\cap { A_n}} |f_n|^p\d \mu(x)&= \int_{\{|f_n|\ge \eps  \}\cap { A_n}} |f_n|^{p-2}|f_n|^2\d \mu(x)\le \int_{\{g_n\ge \eps  \}\cap { A_n}} g_n^{p-2}|f_n|^2\d \mu(x)\\
			&\le c_\eps^{-1} \delta^{-(p-2)}\!\!\int_{{ A_n}}\!\!\!\! \Big(\inf_{\omega_{\delta,n}(x)} \Psi \Big)|f_n|^2\d \mu\to 0,
		\end{align*}
		in the limit as $n\to\infty$, where in the last line we used \eqref{eq:inf psi ineq}.
	\end{proof}

	\smallskip 
	
	 \noindent\textbf{Acknowledgements.} The authors thank Cristiana De Filippis, Nicola Gigli and Endre S\"uli for useful discussions. I.Y.V.\ was partially supported by the European Union (ERC, ConFine, 101078057).

\def\cprime{$'$} \def\cprime{$'$}

\end{document}